\documentclass[10pt,a4paper]{article}
\usepackage[utf8]{inputenc}
\usepackage[T1]{fontenc}
\usepackage{geometry}
\usepackage{amsthm}
\usepackage[english]{babel}
\usepackage{amssymb}
\usepackage{lmodern}
\usepackage{calligra,amsmath,amsfonts}
\usepackage{mathrsfs} 
\usepackage{dsfont}
\usepackage{stmaryrd}

\usepackage{graphicx}
\usepackage{fullpage}
\usepackage[nottoc,notlot,notlof]{tocbibind}
\usepackage[colorlinks=true,urlcolor=black,linkcolor=black,citecolor=black]{hyperref}
\numberwithin{equation}{section}
\usepackage{tikz-cd}
\usepackage{pst-node}
\usetikzlibrary{matrix}
\usepackage{mathtools}
\usepackage{verbatimbox}
\usepackage{diagbox}
\usepackage{array}
\newcolumntype{C}{>{$\displaystyle} c <{$}}
\usepackage{makecell}
\usepackage{float}
\usepackage{tabu}
\usepackage[babel=true]{csquotes}
\usepackage{cancel}
\usepackage{xcolor}

\makeatletter
\def\env@dmatrix{\hskip -\arraycolsep
	\let\@ifnextchar\new@ifnextchar
	\def\arraystretch{2}%
	\array{*{\c@MaxMatrixCols}{>{\displaystyle}c}}%
}

\begin{document}
	
	\title{
		Morse Index Estimates of Min-Max Willmore Surfaces}
	\author{Alexis Michelat\footnote{Department of Mathematics, ETH Zentrum, CH-8093 Zürich, Switzerland.}\setcounter{footnote}{0}}
	\date{\today}
	
	\maketitle
	
	\vspace{1.5em}
	
	\begin{abstract}
		We show that the sum of the Morse indices of the Willmore spheres realising the width of Willmore type sweep-outs is bounded by the number of the parameters of the min-max. As an application, we deduce that among the \emph{true} Willmore  spheres realising the min-max sphere eversion, at most one of them one has index $1$, while the others must be stable.
	\end{abstract}

	\tableofcontents
	\vspace{0.5cm}
	\begin{center}
	\end{center}
	\theoremstyle{plain}
	\newtheorem*{theorem*}{Theorem}
	\newtheorem{theorem}{Theorem}[section]
	\newtheorem{theoremdef}{Théorème-Définition}[section]
	\newtheorem{lemme}[theorem]{Lemma}
	\newtheorem{propdef}[theorem]{Proposition-Definition}
	\newtheorem{prop}[theorem]{Proposition}
	\newtheorem{cor}[theorem]{Corollary}
	\theoremstyle{definition}
	\newtheorem*{definition}{Definition}
	\newtheorem*{definitions}{Definitions}
	\newtheorem{defi}[theorem]{Definition}
	\newtheorem{rem}[theorem]{Remark}
	\newtheorem{rems}[theorem]{Remarks}
	\newtheorem*{rems2}{Remarks}
	\newtheorem{exemple}[theorem]{Example}
	\newcommand{\N}{\ensuremath{\mathbb{N}}}
	\parskip 1ex
	\newcommand{\vc}[3]{\overset{#2}{\underset{#3}{#1}}}
	\newcommand{\conv}[1]{\ensuremath{\underset{#1}{\longrightarrow}}}
	\newcommand{\A}{\ensuremath{\mathscr{A}}}
	\newcommand{\D}{\ensuremath{\nabla}}
	\renewcommand{\N}{\ensuremath{\mathbb{N}}}
	\newcommand{\Z}{\ensuremath{\mathbb{Z}}}
	\newcommand{\I}{\ensuremath{\mathbb{I}}}
	\newcommand{\R}{\ensuremath{\mathbb{R}}}
	\newcommand{\W}{\ensuremath{\mathscr{W}}}
	\newcommand{\Q}{\ensuremath{\mathscr{Q}}}
	\newcommand{\C}{\ensuremath{\mathbb{C}}}
	\newcommand{\z}{\ensuremath{\bar{z}}}
	\newcommand{\vol}{\ensuremath{\mathrm{vol}}}
	\renewcommand{\tilde}{\ensuremath{\widetilde}}
	\newcommand{\p}[1]{\ensuremath{\partial_{#1}}}
	\newcommand{\Res}{\ensuremath{\mathrm{Res}}}
	\newcommand{\lp}[2]{\ensuremath{\mathrm{L}^{#1}(#2)}}
	\renewcommand{\wp}[3]{\ensuremath{\left\Vert #1\right\Vert_{\mathrm{W}^{#2}(#3)}}}
	\newcommand{\hp}[3]{\ensuremath{\left\Vert #1\right\Vert_{\mathrm{H}^{#2}(#3)}}}
	\newcommand{\np}[3]{\ensuremath{\left\Vert #1\right\Vert_{\mathrm{L}^{#2}(#3)}}}
	\newcommand{\h}{\ensuremath{\vec{h}}}
	\renewcommand{\Re}{\ensuremath{\mathrm{Re}\,}}
	\renewcommand{\Im}{\ensuremath{\mathrm{Im}\,}}
	\newcommand{\diam}{\ensuremath{\mathrm{diam}\,}}
	\newcommand{\leb}{\ensuremath{\mathscr{L}}}
	\newcommand{\supp}{\ensuremath{\mathrm{supp}\,}}
	\renewcommand{\phi}{\ensuremath{\vec{\Phi}}}
	\newcommand{\Perp}{\ensuremath{\perp}}
	\renewcommand{\perp}{\ensuremath{N}}
	\renewcommand{\H}{\ensuremath{\vec{H}}}
	\newcommand{\norm}[1]{\ensuremath{\Vert #1\Vert}}
	\newcommand{\e}{\ensuremath{\vec{e}}}
	\newcommand{\f}{\ensuremath{\vec{f}}}
	\renewcommand{\epsilon}{\ensuremath{\varepsilon}}
	\renewcommand{\bar}{\ensuremath{\overline}}
	\newcommand{\s}[2]{\ensuremath{\langle #1,#2\rangle}}
	\newcommand{\bs}[2]{\ensuremath{\left\langle #1,#2\right\rangle}}
	\newcommand{\n}{\ensuremath{\vec{n}}}
	\newcommand{\ens}[1]{\ensuremath{\left\{ #1\right\}}}
	\newcommand{\w}{\ensuremath{\vec{w}}}
	\newcommand{\vg}{\ensuremath{\mathrm{vol}_g}}
	\renewcommand{\d}[1]{\ensuremath{\partial_{x_{#1}}}}
	\newcommand{\dg}{\ensuremath{\mathrm{div}_{g}}}
	\renewcommand{\Res}{\ensuremath{\mathrm{Res}}}
	\newcommand{\totimes}{\ensuremath{\,\dot{\otimes}\,}}
	\newcommand{\un}[2]{\ensuremath{\bigcup\limits_{#1}^{#2}}}
	\newcommand{\res}{\mathbin{\vrule height 1.6ex depth 0pt width
			0.13ex\vrule height 0.13ex depth 0pt width 1.3ex}}
	\newcommand{\ala}[5]{\ensuremath{e^{-6\lambda}\left(e^{2\lambda_{#1}}\alpha_{#2}^{#3}-\mu\alpha_{#2}^{#1}\right)\left\langle \nabla_{\vec{e}_{#4}}\vec{w},\vec{\mathbb{I}}_{#5}\right\rangle}}
	\setlength\boxtopsep{1pt}
	\setlength\boxbottomsep{1pt}
	\allowdisplaybreaks
	\newcommand*\mcup{\mathbin{\mathpalette\mcapinn\relax}}
	\newcommand*\mcapinn[2]{\vcenter{\hbox{$\mathsurround=0pt
				\ifx\displaystyle#1\textstyle\else#1\fi\bigcup$}}}
	\def\Xint#1{\mathchoice
		{\XXint\displaystyle\textstyle{#1}}%
		{\XXint\textstyle\scriptstyle{#1}}%
		{\XXint\scriptstyle\scriptscriptstyle{#1}}%
		{\XXint\scriptscriptstyle\scriptscriptstyle{#1}}%
		\!\int}
	\def\XXint#1#2#3{{\setbox0=\hbox{$#1{#2#3}{\int}$ }
			\vcenter{\hbox{$#2#3$ }}\kern-.58\wd0}}
	\def\ddashint{\Xint=}
	\newcommand{\dashint}[1]{\ensuremath{{\Xint-}_{\mkern-10mu #1}}}

\section{Introduction}

We present in this paper a proof of the lower semi-continuity of the Morse index for min-max Willmore spheres constructed by the viscosity method of Tristan Rivi\`{e}re (see \cite{eversion}, \cite{viscosity}, \cite{geodesics}). Combining this result with the general Morse index bounds in viscosity methods of \cite{index2} and the classification of branched Willmore spheres (\cite{classification}, \cite{sagepaper}), we obtain the expected Morse index bound for the critical points arising as solutions of Willmore min-max problems.

We recall briefly the principal definitions related to the Willmore energy. Let $\Sigma$ be a closed Riemann surface, $n\geq 3$, and $\phi:\Sigma\rightarrow \R^n$ be a smooth immersion. Then its Willmore energy is defined by
\begin{align}\label{lagrange}
	W(\phi)=\int_{\Sigma}|\H_g|^2d\vg,
\end{align}
where $g=\phi^{\ast}g_{\,\R^n}$ is the pull-back metric on $\Sigma$ of the flat metric $g_{\,\R^n}$ on $\R^n$, and $\H_g$ is the mean curvature of the immersion $\phi$, which is half of the trace of the second fundamental form $\vec{\I}$ of the immersion $\phi$, \text{i.e.} 
\begin{align*}
	\H_g=\frac{1}{2}\sum_{i,j=1}^{2}g^{i,j}\vec{\I}_{i,j}.
\end{align*}
Critical points of the Lagrangian in \eqref{lagrange} are called Willmore surfaces. We will not recall in details the main developments obtained for minimisers of the Willmore energy (see \cite{lieyau} and \cite{marqueswillmore}). Let us mention that besides direct estimates given by explicit examples and sphere rigidity for $4\pi$ energy Willmore spheres, nothing in known in codimension at least $2$, with the notable exception of $\R\mathbb{P}^2$ (see \cite{lieyau}, \cite{kusnerpacific}, \cite{bryant3}). In this article, we are rather interested in min-max problems, for which one aims at obtaining Morse index bounds depending on the dimension (defined in some natural way to be described later in explicit examples) the considered admissible family of min-max. 

The admissible min-max families for which the following result applies are any of the families described in \cite{index2} (including classical examples of Lazer and Solimini \cite{lazer}), excepted for the dual admissible families. Rather than giving a general definition, we will give examples of $k$-dimensional admissible families after the statement of the following main result (see Section \ref{admissible} for a precise definition).

\begin{theorem}\label{main}
	Let $n\geq 3$ and let $\mathscr{A}$ be a \emph{non-trivial} $d$-dimensional admissible
	family of $W^{3,2}$ immersions of the sphere $S^2$ into $\R^n$. Assume that
	\begin{align*}
	\beta_0=\inf_{A\in \mathscr{A}}\sup W(A)<\infty.
	\end{align*}
	Then there exists finitely many \emph{true} branched compact Willmore spheres $\phi_1
	,\cdots,\phi_p:S^2\rightarrow\R^n$, and \emph{true} branched compact Willmore spheres $\vec{\Psi}_1,\cdots,\vec{\Psi}_q:S^2\rightarrow\R^n$ such that
	\begin{align}\label{quantization2}
	\beta_0=&\sum_{i=1}^{p}W(\phi_i)+\sum_{j=1}^{q}\left(W(\vec{\Psi}_j)-4\pi\theta_j\right),
	\end{align}
	where the integer $\theta_1,\cdots,\theta_q$ correspond respectively to the highest multiplicities of $\vec{\Psi}_1,\cdots,\vec{\Psi}_q$. Furthermore, we have
	\begin{align*}
	\sum_{i=1}^{p}\mathrm{Ind}_{W}(\phi_i)+\sum_{j=1}^{q}\mathrm{Ind}_{W}(\vec{\Psi}_j)\leq d.
	\end{align*}
	Furthermore, if $n=3$ or $n=4$, we have $\beta_0\in 4\pi \N$.
\end{theorem}

One of the main motivation for this theorem comes from the following corollary.

\begin{theorem}\label{main2}
	Let $\iota_+:S^2\rightarrow \R^3$ be the standard embedding of the $2$-sphere  into the $3$-dimensional Euclidean space, $\iota_-:S^2\rightarrow\R^3$, and $\mathrm{Imm}(S^2,\R^3)$ be the space of smooth immersions from $S^2$ to $\R^3$. We denote by $\Omega$ the set of paths between the two spheres, defined by
	\begin{align*}
	\Omega=C^0\left([0,1],\mathrm{Imm}(S^2,\R^3)\right)\cap\ens{\{\phi_t\}_{t\in [0,1]},\, \phi_0=\iota_+,\; \phi_1=\iota-}.
	\end{align*}
	and we define the \emph{cost of the sphere eversion} as
	\begin{align*}
	\beta_0=\min_{\phi\in \Omega}\max_{t\in [0,1]}W(\phi_t)\geq 16\pi.
	\end{align*}
	Then there exists finitely many \emph{true} branched Willmore spheres $\phi_1,\cdots,\phi_m:S^2\rightarrow \R^3$ such that
	\begin{align}\label{cost}
	\beta_0=\sum_{j=1}^{m}W(\phi_j)\in 4\pi\N.
	\end{align}
	Furthermore, we have
	\begin{align*}
	\displaystyle \sum_{j=1}^{m}\mathrm{Ind}_W(\phi_j)\leq 1.
	\end{align*}
\end{theorem}

More generally, consider the following generalisation of Theorem \ref{main2} (we restrict the discussion to the specific case of codimension $1$ for simplicity). Let $k>0$ and $\Gamma\in \pi_k(\mbox{Imm}(S^2,{\mathbb R}^3))$ be a non-zero element (provided $\pi_k(\mathrm{Imm}(S^2,\R^3))$ is not trivial). Thanks of \cite{eversion} and \cite{classification}, we have
\[
\beta_\Gamma=\inf_{\vec{\Phi}(t,\cdot)\simeq\Gamma}\ \max_{t\in S^k}\ W(\vec{\Phi}(t,\cdot))\in 4\pi \, {\mathbb N}^\ast.
\]
Furthermore, there exists finitely many \emph{true} branched Willmore spheres $\phi_1,\cdots, \phi_{m(\Gamma)}:S^2\rightarrow \R^3$ such that
	\begin{align*}
		\beta_{\Gamma}=\sum_{j=1}^{m(\Gamma)}W(\phi_j),\quad \text{and}\;\, \sum_{j=1}^{m(\Gamma)}\mathrm{Ind}_{W}(\phi_j)\leq k.
	\end{align*}
This furnishes a map
\begin{align}\label{map}
\Gamma\in  \pi_k(\mbox{Imm}(S^2,{\mathbb R}^3))\ \longrightarrow \ \frac{\beta_\Gamma}{4\pi}\in {\mathbb N}^\ast.
\end{align}
Now, recall from Smale (\cite{smale}) that $\pi_k(\mbox{Imm}(S^2,{\mathbb R}^3))=\pi_k(SO(3))\times\pi_{k+2}(SO(3))$, and it would be interesting to study the map \eqref{map} giving the Willmore energy of the optimal representative of a non-zero element of these homotopy groups. Notice in particular, that $\beta_0$ given by \eqref{cost} in Theorem \ref{main2} corresponds to $\beta_{\Gamma}$ where $\Gamma\in \pi_1(\mathrm{Imm}(S^2,\R^3))=\Z_2\times \Z$ is a generator.

\textbf{Acknowledgments.} I wish to thank my advisor Tristan Rivi\`{e}re for helpful comments and suggestions on this work.

\section{Second derivative of the Onofri energy and main estimate}

\subsection{A brief introduction to weak immersions}

A weak immersion from a closed Riemann surface $\Sigma$ is a map $\phi\in W^{2,2}\cap W^{1,\infty}(\Sigma,\R^n)$ such that 
\begin{align*}
	\inf_{p\in \Sigma}|d\phi\wedge d\phi|_{g_0}>0,
\end{align*}
for some smooth metric $g_0$ on $\Sigma$. We denote this space $\mathscr{E}(\Sigma,\R^n)$ (more generally, branched weak immersions will be described in Section \ref{bfinite}). We refer to \cite{rivierecrelle} and the papers cited within this article for more informations about this subject.

For all weak immersion $\phi\in \mathscr{E}(\Sigma,\R^n)$, we denote by $g_{\phi}=\phi^{\ast}g_{\,\R^n}$ the pull-back metric, by $\Gamma(\phi^{\ast}T\R^n)$ the \emph{continuous sections} of the pull-back bundle $\phi^{\ast}T\R^n$ and we define
\begin{align*}
    W^{2,2}_{\phi}\cap W^{1,\infty}(\Sigma,T\R^n)=W^{2,2}\cap W^{1,\infty}(\Sigma,T\R^n)\cap\ens{\w: \w(p)\in T_{\phi(p)}^{\perp}\R^n\;\, \text{for a.e.}\;\, p\in \Sigma}\subset \Gamma(\phi^{\ast}T\R^n)
\end{align*}
the space of weak \emph{normal} sections from the pull-back bundle $\phi^{\ast}T\R^n$ and by
\begin{align*}
\mathscr{E}_{\phi}(\Sigma,\R^n)=W^{2,2}_{\phi}\cap W^{1,\infty}(\Sigma,T\R^{n})\cap\ens{\w:\;\,\Delta_{g_{\phi}}^{\perp}\w\in \mathscr{L}^2(\Sigma,d\mathrm{vol}_{g_{\phi}})} 
\end{align*}
the space of all admissible normal variations, where $\Delta_{g_{\phi}}^{\perp}$ is the normal Laplacian, given for an orthonormal moving frame $(\e_1,\e_2)$ of the tangent bundle by
\begin{align*}
	\Delta_{g_{\phi}}^{\perp}=\sum_{i,j=1}^{2}\D^2_{\e_i,\e_i}=\sum_{i,j=1}^{2}\left(\D^{\perp}_{\e_i}\D^{\perp}_{\e_i}-\D^{\perp}_{\bar{\D}_{\e_i}\e_i}\right).
\end{align*} 
We will see in the next section that the space $\mathscr{E}_{\phi}(\Sigma,\R^n)$ constitutes in a sense the tangent space of weak immersions at $\phi$ (although this is not Finsler manifold), and we endow it with the norm $\norm{\,\cdot\,}_{\mathscr{E}_{\phi}(\Sigma)}$ (it is a slight modification of the norm presented in \cite{viscosity})
\begin{align*}
\norm{\w}_{\mathscr{E}_{\phi}(\Sigma)}=\left(\int_{\Sigma}\left(|\w|^2+|d\w|_{g_{\phi}}^2+|\Delta_{g_{\phi}}^{\perp}\w|_{g_{\phi}}^2\right)d\mathrm{vol}_{g_{\phi}}\right)^{\frac{1}{2}}+\np{|d\w|_{g_{\phi}}}{\infty}{\Sigma},
\end{align*}
where $g_{\phi}=\phi^{\ast}g_{\,\R^n}$ is the pull-back metric on the weak bundle $\phi^{\ast}T\R^n$. More precisely, we will show that a branched Willmore immersion $\phi$, the second derivative $D^2W(\phi)(\w,\w)$ is well-defined for a variation $\w\in W^{2,2}_{\phi}\cap W^{1,\infty}(\Sigma,T\R^n)$ if and only if $\norm{\w}_{\mathscr{E}_{\phi}(\Sigma,\R^n)}<\infty$.

Notice that by standard elliptic estimate and as $\Sigma$ is compact $\w\in \mathscr{E}_{\phi}(\Sigma,\R^n)$ implies that  for some constant $C_{\phi}$ (depending on $\phi$)
\begin{align}\label{ests}
	\left(\int_{\Sigma}|\D^{\perp}d\w|_{g_{\phi}}^2d\mathrm{vol}_{g_{\phi}}\right)^{\frac{1}{2}}\leq C_{\phi}\norm{\w}_{\mathscr{E}_{\phi}(\Sigma)}
\end{align}
 However, this estimate is in general false for branched immersions (the left-hand of \eqref{ests} might be infinite) and we refer to Lemma \ref{est2} for a similar estimate.
 
 \subsection{The Onofri energy}

Let $\Sigma$ be a compact connected Riemann surface, and for all smooth immersion $\phi:\Sigma\rightarrow\R^n$, let
\begin{align*}
g=\phi^\ast g_{\,\R^n},
\end{align*}
where $g_{\,\R^n}$ is the Euclidean flat metric on $\R^n$. By the uniformisation theorem, there exists a metric ${g}_0$ of constant Gauss curvature and unit volume and a smooth function $\alpha:\Sigma^2\rightarrow\R$ such that
\begin{align*}
g=e^{2\alpha}\,{g}_0.
\end{align*}
If $\Sigma$ has genus at least $1$, this function $\alpha$ is uniquely well-defined, while for the sphere, $\alpha$ is defined up to the positive conformal diffeomorphims $\mathscr{M}^+(S^2)$. Then we define the Onofri energy $\mathscr{O}(\phi)$ by
\begin{align*}
\mathscr{O}(\phi)=\frac{1}{2}\int_{\Sigma}|d\alpha|_g^2d\vg+K_{{g}_0}\int_{\Sigma}\alpha e^{-2\alpha}d\vg-\frac{K_{g_0}}{2}\log\int_{\Sigma}d\vg.
\end{align*}
where $K_{g_0}$ is the Gauss curvature of $g_0$. Therefore, if $\Sigma$ is of genus $\gamma$, we have 
\begin{align}\label{Kbar}
K_{{g}_0}=4\pi(1-\gamma)
\end{align}
As ${g}_0$ has- volume one, $d\mathrm{vol}_{g_0}$ is a probability measure and we have by Jensen's inequality
\begin{align*}
\exp\int_{\Sigma}2\alpha\,e^{-2\alpha}d\vg=\exp\int_{\Sigma}2\alpha\, d\mathrm{vol}_{g_0}\leq \int_{\Sigma}e^{2\alpha}d\mathrm{vol}_{g_0}=\int_{\Sigma}d\vg
\end{align*}
so for $\gamma\geq 2$ (for $\gamma=1$, there is nothing to prove), we have 
\begin{align*}
K_{{g}_0}\int_{\Sigma}\alpha\,e^{-2\alpha}d\vg\geq \frac{K_{{g}_0}}{2}\log\int_{\Sigma}d\vg,
\end{align*}
so
\begin{align*}
\mathscr{O}(\phi)\geq \frac{1}{2}\int_{\Sigma^2}|d\alpha|_g^2d\vg\geq 0
\end{align*}
while for $\gamma=0$, the Onofri inequality implies that 
\begin{align*}
\frac{1}{2}\int_{S^2}|d\alpha|_g^2d\vg+4\pi\int_{S^2}\alpha\,e^{-2\alpha}d\vg-2\pi\log\int_{S^2}d\vg\geq 0.
\end{align*}
Furthermore, for a function $\alpha$ satisfying an Aubin gauge (see \cite{eversion}), we have the improvement
\begin{align*}
	\frac{1}{3}\int_{S^2}|d\alpha|_g^2d\vg+4\pi\int_{S^2}\alpha e^{-2\alpha}d\vg-2\pi\log\int_{\Sigma}d\vg\geq 0,
\end{align*}
which implies in particular that
\begin{align*}
	\frac{1}{6}\int_{S^2}|d\alpha|_g^2d\vg\leq \mathscr{O}(\phi).
\end{align*}
an inequality which will prove to be crucial here as in \cite{eversion}. 
For the sake of completeness, let us recall the definition of an Aubin gauge as given by Tristan Rivi\`{e}re in \cite{eversion}.
\begin{defi}
	Let $\phi:S^2\rightarrow \R^n$ be a smooth immersion and $g=\phi^{\ast}g_{\R^n}$ the induced metric on $S^2$. We see that a couple  $(\Psi,\alpha)$ of a diffeomorphism $\Psi:S^2\rightarrow S^2$ and of a smooth function $\alpha:S^2\rightarrow \R$ is an \emph{Aubin gauge} if there exists a constant Gauss curvature metric $g_0$ of unit volume such that the following conditions are satisfied
	\begin{align*}
		&g=e^{2\alpha}g_0, \quad \Psi^{\ast}g_0=\frac{g_{S^2}}{4\pi},\quad\text{and}\;\,
		\int_{S^2}x_j\,e^{2\alpha(\Psi(x))}d\mathrm{vol}_{S^2}(x)=0,
	\end{align*}
	where $g_{S^2}$ is the standard round metric on $S^2$.
\end{defi}
\begin{rem}
As will appear clearly in the section dedicated to the study of the Onofri energy for weak immersions thanks of the measurable Riemann mapping theorem of Ahlfors-Bers, this definition does not need to assume smoothness of the immersion in consideration.
\end{rem}

In the proof of the semi-continuity of the index in the viscosity method for the Willmore energy, we have compute the second derivative of the Lagrangian
\begin{align*}
W_{\sigma}(\phi)=W(\phi)+\sigma^2\int_{\Sigma}(1+|\H|^2)^2d\vg+\frac{1}{\log\frac{1}{\sigma}}\mathscr{O}(\phi),
\end{align*}
and to estimate each term to show that the components depending on $\sigma$ are negligible as $\sigma\rightarrow 0$, while the second derivative $D^2W(\phi)$ converges in an appropriate sense. Notice also that we will have to check a key technical condition (defined in \cite{index2}) called the \textbf{Energy bound} in order to be able to apply the general theory of \cite{index2}.

\subsection{Measurable Riemann mapping theorem and an estimate for varying conformal parameters}

As we want to pass to the limit in the second derivative of the approximation $W_{\sigma}$ of the Willmore energy $W$, certain estimates are pretty straightforward thanks of the $\epsilon$-regularity. However, there is a non-local contribution coming from the second derivative of the Onofri energy and as this is not clear how such a quantity behaves as $\sigma\rightarrow 0$ so we will have to bound them directly by a local quantity which can be estimated thanks of an \emph{a priori} inequality. 

The forthcoming estimate of Theorem \ref{mainestimate} is largely based on the regular dependence of the solution of the Beltrami equation of Ahlfors and Bers (\cite{ahlfors}) with respect to the parameters which is recalled below (see \cite{ahlfors2} and also \cite{imayoshi}). The estimate is only necessary for the sphere, while in higher genus we can obtain an essentially sharp inequality by a direct computation. It permits to show that we can differentiate the conformal parameter of varying metrics of weak immersions as elements of suitable Banach spaces (to be defined below) and provides an \emph{a priori} estimate which will is the key point to obtain the lower semi-continuity of the index.

First, recall that the Beurling transform (the following integral should be understood as a principal value)
\begin{align*}
Sf(z)=\frac{1}{2\pi i}\int_{\C}\frac{f(\zeta)d\zeta\wedge d\bar{\zeta}}{(\zeta-z)^2}
\end{align*}
defined for all function $f\in L^p(\C)$ for all $1<p<\infty$ is a continuous operator $L^p(\C)\rightarrow L^p(\C)$. For all $2<p<\infty$ and for all $R>0$, define the following norm on $C^{1}_c(B(0,R),\C)/\C$
\begin{align}
\norm{f}_{B_{R,p}}=\sup_{z_1,z_2\in B(0,R),  z_1\neq z_2}\frac{|f(z_1)-f(z_2)|}{|z_1-z_2|^{\alpha_p}}+\np{f_z}{p}{B(0,R)}+\np{f_{\z}}{p}{B(0,R)}
\end{align}
where $\alpha_p=1-\dfrac{2}{p}$. We define $B_{R,p}$ as the completion of $C^1_c(B(0,R),\C)/\C$ with respect to the $B_{R,p}$ norm.

\begin{theorem}[Ahlfors-Bers\label{ab}]
	Let $2<p<\infty$, $0<k<1$ and assume that $k\norm{S}_{L^p(\C)}<1$, where $\norm{S}_{L^p(\C)}$ is the operator norm of the Beurling transform $S:L^p(\C)\rightarrow L^p(\C)$. Let $I\subset \R$ be an open interval containing $0\in \R$ and $\mu=\ens{\mu_t}_{t\in I}: I\rightarrow L^{\infty}(\C)$ be a differentiable mapping at $0$. For all $t\in I$, let $f_t$ be the unique solution in $z+W^{1,p}(\C)$ of the Beltrami equation
	\begin{align*}
	\p{\z}f_t=\mu_t\,\p{z}f_t.
	\end{align*}
	Then for all $R>0$, the map $f_t$ is differentiable at $t=0$ as an element of $B_{R,p}$ and 
	\begin{align*}
	\left(\frac{d}{dt}f_t(z)\right)_{|t=0}=\frac{1}{2\pi i}\int_{\C}\left(\frac{d}{dt}\mu_t(\zeta)\right)_{|t=0}R(f_0(\zeta),f_0(z))f_0^2(\zeta)d\zeta\wedge d\bar{\zeta}
	\end{align*} 
	where $R(z,\zeta)=\dfrac{1}{z-\zeta}-\dfrac{\zeta}{z-1}+\dfrac{\zeta-1}{z}$.
\end{theorem}
\begin{rem}
	The operators $\partial_z$ and $\partial_{\z}$ are the standard Cauchy-Riemann operators. Their definition will be recalled in the proof of Lemma \ref{d2k}.
\end{rem}

\begin{theorem}\label{mainestimate}
	Let $\Sigma$ be a closed Riemann surface and $\phi\in \mathscr{E}(\Sigma,\R^n)$ be a weak immersion of finite total curvature and $\w\in \mathscr{E}_{\phi}(\Sigma,T\R^n)$ be an admissible variation of $\phi$ and define for all $t\in (-\epsilon,\epsilon)$ (for some $\epsilon>0$ small enough) the weak immersion $\phi_t=\phi+t\w$. For all $t\in (-\epsilon,\epsilon)$,  define the metric $g_t=\phi_t^{\ast}g_{\,\R^n}$ and let (by the uniformisation theorem) $g_{0,t}$ be a constant Gauss curvature metric of volume $1$ and $\alpha_t:\Sigma\rightarrow \R$ be a measurable function such that
	\begin{align}\label{eq1}
	g_t=e^{2\alpha_t}g_{0,t}.
	\end{align}
	Then the map $t\mapsto \alpha_t$ is differentiable as a map with values into $L^2(\Sigma)$ and we have
	\begin{align*}
	\alpha'_0=\frac{d}{dt}\alpha_t|_{|t=0}\in W^{1,2}(\Sigma),
	\end{align*}
	where the $L^2(\Sigma)$ and $W^{1,2}(\Sigma)$ space are taken with respect to a constant Gauss curvature metric independent of $t$. Furthermore, we have for some universal constant $0<C<\infty$ if $\Sigma=S^2$ the estimate
	\begin{align}
	\int_{S^2}|d\alpha'_0|_g^2d\vg&\leq  \left(C+\sqrt{W(\phi)}+\sqrt{W(\phi)-4\pi}+4\left(\int_{S^2}|d\alpha|_g^2d\vg\right)^{\frac{1}{2}}\right)\np{|d\w|_g}{\infty}{S^2},
	\end{align}
	while for $\Sigma$ of genus at least $1$, we have
	\begin{align*}
		\left(\int_{\Sigma}|d\alpha_0'|^2d\vg\right)^{\frac{1}{2}}\leq \left(\sqrt{\pi|\chi(\Sigma)|}+\sqrt{W(\phi)}+\sqrt{W(\phi)-2\pi\chi(\Sigma)}+4\left(\int_{\Sigma}|d\alpha|^2d\vg\right)^{\frac{1}{2}}\right)\np{|d\w|_g}{\infty}{\Sigma}.
	\end{align*}
\end{theorem}
\begin{proof}
	We first consider estimate for metrics of immersions of $S^2$.
	Fix some conformal chart $z=x_1+ix_2:U\subset S^2\rightarrow D^2\subset \C$ and for all $t\in I$, let $g_{i,j}(t)=\s{\p{x_i}\phi_t}{\p{x_j}\phi_t}$ ($1\leq i,j\leq 2$) be the coefficients of the metric of the immersion $\phi_t:S^2\rightarrow \R^n$ and $\lambda_t,\mu_t:\Sigma\rightarrow\R$ be such that 
	\begin{align*}
	&g_t=e^{2\lambda_t}|dz+\mu_td\z|^2,\quad e^{2\lambda_t}=\frac{1}{4}\left(g_{1,1}(t)+g_{2,2}(t)+2\sqrt{g_{1,1}(t)g_{2,2}(t)-g_{1,2}^2(t)}\right)\\
	&
	\mu_t=\frac{g_{1,1}(t)-g_{2,2}(t)+2i\,g_{1,2}(t)}{g_{1,1}(t)+g_{2,2}(t)+2\sqrt{g_{1,1}(t)g_{2,2}(t)-g_{1,2}^2(t)}}.
	\end{align*}
	Furthermore, denote by $\bar{\mu_t}\in L^{\infty}(\C)$ the extension by $0$ of $\mu_t:D^2\rightarrow \C$.
	As $\mu_0=0$, for all $2<p<\infty$, there exists $0<\delta(p)<\epsilon$ such that
	\begin{align*}
	C_p\sup_{t\in (-\delta(p),\delta(p))}\np{\mu_t}{\infty}{D^2}<1.
	\end{align*}
	Therefore, for all $2<p<\infty$, there exists $\delta=\delta(p)>0$ such that for all $t\in (-\delta(p),\delta(p))$ the equation
	\begin{align*}
	\p{\z}f_t=\bar{\mu_t}\p{z}f_t
	\end{align*}
	admits a unique $f_t:\C\rightarrow \C$ in $z+W^{1,p}(\C)$ and satisfying
	\begin{align}\label{s0}
	\left\Vert\left(\frac{d}{dt}f_t\right)_{|t=0}\,\right\Vert{B_{1,p}}\leq C_p'\np{\left(\frac{d}{dt}\mu_t\right)_{|t=0}}{\infty}{D^2}
	\end{align}
	for some universal constant $C_p'$. We fix in the rest of the proof some real number $2<p<\infty$.
	As
	\begin{align*}
	e^{2\alpha_t}\frac{|df_t|^2}{\pi(1+|f_t|^2)^2}
	=e^{2\alpha_t}\frac{|\p{z}f_tdz+\p{\z}f_td\z|^2}{\pi(1+|f_t|^2)^2}
	=\frac{e^{2\alpha_t}|\p{z}f_t|^2}{\pi(1+|f_t|^2)^2}\left|dz+\frac{\p{\z}f_t}{\p{z}f_t}d\z\right|^2
	=\frac{e^{2\alpha_t}|\p{z}f_t|^2}{\pi(1+|f_t|^2)^2}|dz+\mu_td\z|^2,
	\end{align*}
	if $\alpha_t:D^2\rightarrow \R$ is defined by
	\begin{align}\label{diff}
	e^{2\lambda_t}=e^{2\alpha_t}\frac{|\p{z}f_t|^2}{\pi(1+|f_t(z)|^2)^2}
	\end{align}
	this implies that
	\begin{align*}
	g_t=e^{2\alpha_t}\frac{|df_t|^2}{\pi(1+|f_t|^2)^2}=e^{2\alpha_t}g_{0,t}.
	\end{align*}
	As $g_{0,t}=\dfrac{1}{4\pi}f_t^{\ast}g_{S^2}$ and $f_t$ is a diffeomorphism, we deduce that $g_{0,t}$ is a volume $1$  metric of constant sectional curvature. By Theorem \ref{ab}, the functions $f_t$ and $\p{z}f_t$ are differentiable at $0$ and continuous in $L^p(D^2)$ for all $2<p<\infty$, so $\alpha_t$ is also differentiable as en element of $L^p(D^2)$ and all the following identities will be valid in the distributional sense.
	
	First, observe as we took a conformal chart that we have by the normalisation $f_0(z)=z$. We first compute
	\begin{align*}
	\frac{d}{dt}\log\left(\frac{|\p{z}f_t|^2}{\pi(1+|f_t|^2)^2}\right)&=\left(\frac{|\p{z}f_t|^2}{(1+|f_t|^2)^2}\right)^{-1}\left(2\,\Re\left(\frac{\p{z}\left(\frac{d}{dt}f_t\right)\,\bar{\p{z}f_t}}{(1+|f_t|^2)^2}\right)-\frac{|\p{z}f_t|^2\,4\,\Re(\frac{d}{dt}f_t\,\bar{f_t})}{(1+|f_t|^2)^3}\right)\\
	&=2\,\Re\left(\frac{\p{z}\left(\frac{d}{dt}f_t\right)}{\p{z}f_t}\right)-\frac{4\,\Re(\frac{d}{dt}f_t\,\bar{f_t})}{(1+|f_t|^2)}.
	\end{align*}
	As $\mu_0=0$, we also obtain
	\begin{align*}
	\left(\frac{d}{dt}\mu_t\right)_{|t=0}&=\frac{1}{2}\left(e^{-2\lambda}\s{\D_{\e_1}\w}{\e_1}-e^{-2\lambda}\s{\D_{\e_2}\w}{\e_2}+i\,e^{-2\lambda}\left(\s{\D_{\e_1}\w}{\e_2}+\s{\D_{\e_2}\w}{\e_1}\right)\right)\frac{dz}{d\z}\\
	&=2e^{-2\lambda}\s{\p{z}\phi}{\D_{\p{z}}\w}\frac{dz}{d\z}
	=2\,g^{-1}\otimes \left(\partial\phi\totimes \partial\w\right),
	\end{align*}
	where $\partial$ is the first order linear elliptic operator yielding $(p+1,q)$ section from a $(p,q)$ section and defined in our local complex coordinate $z$ by
	\begin{align*}
		\partial=\D_{\p{z}}\left(\,\cdot\,\right)\otimes dz,
	\end{align*}
	where $\D$ is the pull-back connexion by $\phi$ of the flat connexion of $\R^n$ (we will adopt consistently the notations given in the introduction of \cite{classification}).
	In particular, we have
	\begin{align*}
		\np{\left(\frac{d}{dt}\mu_t\right)_{|t=0}}{\infty}{D^2}=\np{2g^{-1}\otimes \left(\partial\phi\totimes\partial\w\right)}{\infty}{D^2}\leq \np{|d\w|_g}{\infty}{\Sigma}
	\end{align*}
	By differentiating \eqref{diff}, we obtain 
	\begin{align*}
	\s{d\phi_t}{d\w_t}_{g_t}e^{2\lambda_t}=2\left(\frac{d}{dt}\alpha_t\right)e^{2\lambda_t}+\left(2\,\Re\left(\frac{\p{z}\left(\frac{d}{dt}f_t\right)}{\p{z}f_t}\right)-\frac{4\,\Re(\frac{d}{dt}f_t\,\bar{f_t})}{(1+|f_t|^2)}\right)e^{2\lambda_t}
	\end{align*}
	Therefore, taking $t=0$ and writing for all differentiable function $h_t:\C\rightarrow \C$
	\begin{align*}
	\frac{d}{dt}h_t|_{|t=0}=h'_0,
	\end{align*}
	we obtain
	\begin{align*}
	\s{d\phi}{d\w}_g=2\,\alpha_0'+2\,\Re(\p{z}f_0')-4\,\Re\left(\frac{f_0'{\z}}{(1+|z|^2)}\right)
	\end{align*}
	We have immediately
	\begin{align*}
	\np{\s{d\phi}{d\w}_g}{\infty}{D^2}\leq 2\np{|d\w|_g}{\infty}{D^2}<\infty
	\end{align*}
	and
	\begin{align*}
	\left|\Re\left(\frac{z}{1+|z|^2}\right)\right|\leq \frac{1}{2}.
	\end{align*}
	Therefore, by \eqref{s0}, we have
	\begin{align*}
	\np{\alpha_0'}{p}{D^2}\leq \np{f_0'}{p}{D^2}+\np{\p{z}f'_0}{p}{D^2}+{\pi}\np{|d\w|_g}{\infty}{D^2}\leq C_p''\np{|d\w|_g}{\infty}{D^2}
	\end{align*}
	so there exists in particular a constant $C>0$ independent of $\phi$ such that
	\begin{align}\label{estend}
	\np{\alpha_0'}{2}{D^2}\leq C\np{|d\w|_g}{\infty}{D^2}
	\end{align}
	By the Liouville equation, we have
	\begin{align*}
	-\Delta_{g_t}\alpha_t=K_{g_t}-e^{-2\alpha_t}K_{g_0}.     
	\end{align*}
	By differentiating this equation, we get
	\begin{align*}
	-\Delta_{g_t}\left(\frac{d}{dt}\alpha_t\right)=2\left(\frac{d}{dt}\alpha_t\right)e^{-2\alpha_t}K_{g_0}+\left(\frac{d}{dt}\Delta_{g_t}\right)\alpha_t+\frac{d}{dt}(K_{g_t}).
	\end{align*}
	Therefore, we have by \cite{eversion} or the proof of Theorem \ref{d2o}
	\begin{align*}
	-\Delta_{g}\alpha_0'=2K_{g_0}e^{-2\alpha}\alpha_0'+\s{d\s{d\phi}{d\w}_g}{d\alpha}_g-\ast_g d\ast\left(\left(d\phi\totimes d\w+d\w\totimes d\phi\right)\res_g d\alpha\right)+\frac{d}{dt}(K_{g_t})|_{t=0}.
	\end{align*}
	Now, multiply this equation by $\alpha_0'd\vg$, and integrate by parts to find
	\begin{align}\label{end0}
		&\int_{\Sigma}|d\alpha_0'|_g^2d\vg=\int_{\Sigma}-\alpha_0'\,\Delta_{g}\alpha_0' \,d\vg=2K_{g_0}\int_{\Sigma}|\alpha_0'|^2d\mathrm{vol}_{g_0}+\int_{\Sigma}\s{d\s{d\phi}{d\w}_g}{d\alpha}_g\alpha_0'd\vg\nonumber\\
		&-\int_{\Sigma}\alpha_0'\ast_g\,d\,\ast\left(\left(d\phi\totimes d\w+d\w\totimes d\phi\right)\res_g d\alpha\right)d\vg+\int_{\Sigma}\alpha_0'\frac{d}{dt}\left(K_{g_t}\right)_{t=0}d\vg
	\end{align}
	Now, by integrating by parts and using the Liouville equation, we obtain
	\begin{align}\label{end1}
		&\int_{\Sigma}\s{d\s{d\phi}{d\w}_g}{d\alpha}\alpha_0'd\vg=-\int_{\Sigma}\s{d\phi}{d\w}_g\s{d\alpha
		}{d\alpha_0'}_gd\vg+\int_{\Sigma}\s{d\phi}{d\w}_g\left(-\Delta_g\alpha\right)\,\alpha_0'd\vg\nonumber\\
	    &=-\int_{\Sigma}\s{d\phi}{d\w}_g\s{d\alpha
	    }{d\alpha_0'}_gd\vg+\int_{\Sigma}\s{d\phi}{d\w}_gK_g\alpha_0'd\vg-K_{g_0}\int_{\Sigma}\s{d\phi}{d\w}_g\alpha_0'd\mathrm{vol}_{g_0}
	\end{align}
	Furthermore, as $\frac{d}{dt}\left(d\mathrm{vol}_{g_t}\right)_{|t=0}=\s{d\phi}{d\w}_gd\vg$, and by \cite{indexS3}, we have
	\begin{align}\label{end2}
		&\int_{\Sigma}\alpha_0'\frac{d}{dt}\left(K_{g_t}\right)_{|t=0}d\vg=\int_{\Sigma}\alpha_0'\left(\frac{d}{dt}\left(K_{g_t}d\mathrm{vol}_{g_t}\right)_{|t=0}-K_{g}\frac{d}{dt}\left(d\mathrm{vol}_{g_t}\right)_{|t=0}\right)\nonumber\\
		&=\int_{\Sigma}\alpha_0'\,d\,\Im\left(2\s{\H}{\partial\w}-2\,g^{-1}\otimes\left(\h_0\totimes\bar{\partial}\w\right)\right)-\int_{\Sigma}\s{d\phi}{d\w}K_g\alpha_0'd\vg\nonumber\\
		&=-\int_{\Sigma}d\alpha_0'\wedge \Im\left(2\s{\H}{\partial\w}-2\,g^{-1}\otimes\left(\h_0\totimes\bar{\partial}\w\right)\right)-\int_{\Sigma}\s{d\phi}{d\w}_gK_g\alpha_0'd\vg
	\end{align}
	Finally, we have
	\begin{align}\label{end3}
		-\int_{\Sigma}\alpha_0'\ast_g\,d\,\ast\left(\left(d\phi\totimes d\w+d\w\totimes d\phi\right)\res_g d\alpha\right)d\vg=\int_{\Sigma}\s{d\phi\totimes d\w+d\w\totimes d\phi}{d\alpha\otimes d\alpha_0'}_gd\vg.
	\end{align}
	Gathering \eqref{end0}, \eqref{end1}, \eqref{end2}, \eqref{end3}, we obtain
	\begin{align*}
		&\int_{\Sigma}|d\alpha_0'|_g^2d\vg=2K_{g_0}\int_{\Sigma}|\alpha_0'|^2d\mathrm{vol}_{g_0}-\int_{\Sigma}\s{d\phi}{d\w}_g\s{d\alpha
		}{d\alpha_0'}_gd\vg+\int_{\Sigma}\s{d\phi}{d\w}_gK_g\alpha_0'd\vg\\
	&-K_{g_0}\int_{\Sigma}\s{d\phi}{d\w}_g\alpha_0'd\mathrm{vol}_{g_0}+\int_{\Sigma}\s{d\phi\totimes d\w+d\w\totimes d\phi}{d\alpha\otimes d\alpha_0'}_gd\vg\\
	&-\int_{\Sigma}d\alpha_0'\wedge \Im\left(2\s{\H}{\partial\w}-2\,g^{-1}\otimes\left(\h_0\totimes\bar{\partial}\w\right)\right)-\int_{\Sigma}\s{d\phi}{d\w}_gK_g\alpha_0'd\vg\\
	&=2K_{g_0}\int_{\Sigma}|\alpha_0'|^2d\mathrm{vol}_{g_0}-K_{g_0}\int_{\Sigma}\s{d\phi}{d\w}_g\alpha_0'd\mathrm{vol}_{g_0}-\int_{\Sigma}\s{d\phi}{d\w}_g\s{d\alpha
	}{d\alpha_0'}_gd\vg
	\\
	&+\int_{\Sigma}\s{d\phi\totimes d\w+d\w\totimes d\phi}{d\alpha\otimes d\alpha_0'}_gd\vg
	-\int_{\Sigma}d\alpha_0'\wedge \Im\left(2\s{\H}{\partial\w}-2\,g^{-1}\otimes\left(\h_0\totimes\bar{\partial}\w\right)\right)
	\end{align*}
	Now, we estimate directly by Cauchy-Schwarz inequality
	\begin{align*}
		&\left|\int_{\Sigma}\s{d\phi}{d\w}_g\s{d\alpha}{d\alpha_0'}_gd\vg\right|\leq 2\np{|d\w|_g}{\infty}{\Sigma}\left(\int_{\Sigma}|d\alpha|_g^2d\vg\right)^{\frac{1}{2}}\left(\int_{\Sigma}|d\alpha_0'|^2d\vg\right)^{\frac{1}{2}}\\
		&\left|\int_{\Sigma}\s{d\phi\totimes d\w+d\w\totimes d\phi}{d\alpha\otimes d\alpha_0'}d\vg\right|\leq 2\np{|d\w|_g}{\infty}{\Sigma}\left(\int_{\Sigma}|d\alpha|_g^2d\vg\right)^{\frac{1}{2}}\left(\int_{\Sigma}|d\alpha_0'|^2d\vg\right)^{\frac{1}{2}}\\
		&\left|\int_{\Sigma}d\alpha_0'\wedge \Im\left(2\s{\H}{\partial\w}-2\,g^{-1}\otimes\left(\h_0\totimes\bar{\partial}\w\right)\right)\right|\\
		&\leq \np{|d\w|_g}{\infty}{\Sigma}\left(\sqrt{W(\phi)}+\sqrt{W(\phi)-2\pi\chi(\Sigma)}\right)\left(\int_{\Sigma}|d\alpha_0'|^2d\vg\right)^{\frac{1}{2}}.
	\end{align*}
	Finally, we have
	\begin{align*}
		&\left|K_{g_0}\int_{\Sigma}\s{d\phi}{d\w}_g\alpha_0'd\mathrm{vol}_{g_0}\right|\leq \sqrt{2}|K_{g_0}|\np{|d\w|_g}{\infty}{\Sigma}\left(\int_{\Sigma}|\alpha_0'|^2d\mathrm{vol}_{g_0}\right)^{\frac{1}{2}}\\
		&\leq 2|K_{g_0}|\int_{\Sigma}|\alpha_0'|^2d\mathrm{vol}_{g_0}+\frac{1}{4}|K_{g_0}|\np{|d\w|_g}{\infty}{\Sigma}^2
	\end{align*}
	If $\Sigma$ has genus $0$, \textit{i.e} $\Sigma=S^2$, then we obtain by \eqref{estend}
	\begin{align}\label{read}
		&\int_{S^2}|d\alpha_0'|^2d\vg\leq 16\pi \int_{S^2}|\alpha_0'|^2d\mathrm{vol}_{g_0}+\pi \np{|d\w|_g}{\infty}{S^2}^2\nonumber\\
		&+\np{|d\w|_g}{\infty}{S^2}\left(\sqrt{W(\phi)}+\sqrt{W(\phi)-4\pi}+4\left(\int_{S^2}|d\alpha|_g^2d\vg\right)^{\frac{1}{2}}\right)\left(\int_{S^2}|d\alpha_0'|^2d\vg\right)^{\frac{1}{2}}
		\leq C\np{|d\w|_g}{\infty}{S^2}^2\nonumber\\
		&+\left(\sqrt{W(\phi)}+\sqrt{W(\phi)-4\pi}+4\left(\int_{S^2}|d\alpha|_g^2d\vg\right)^{\frac{1}{2}}\right)\left(\int_{S^2}|d\alpha_0'|^2d\vg\right)^{\frac{1}{2}}\np{|d\w|_g}{\infty}{S^2}.
	\end{align}
	Now, let 
	\begin{align*}
		&X=\left(\int_{S^2}|d\alpha_0'|^2d\vg\right)^{\frac{1}{2}}\\
		&a=\left(\sqrt{W(\phi)}+\sqrt{W(\phi)-4\pi}+4\left(\int_{S^2}|d\alpha|_g^2d\vg\right)^{\frac{1}{2}}\right)\np{|d\w|_g}{\infty}{S^2}\\
		&b=C\np{|d\w|_g}{\infty}{S^2}^2.
	\end{align*}
	The inequality \eqref{read} reads
	\begin{align*}
		X^2\leq aX+b
	\end{align*}
	and as $X\geq 0$, we deduce that
	\begin{align*}
		&X\leq \frac{1}{2}\left(a+\sqrt{a^2+4b}\right)\leq a+\sqrt{b}\leq \left(\sqrt{W(\phi)}+\sqrt{W(\phi)-4\pi}+4\left(\int_{S^2}|d\alpha|_g^2d\vg\right)^{\frac{1}{2}}\right)\np{|d\w|_g}{\infty}{S^2}\\
		&+\sqrt{C}\np{|d\w|_g}{\infty}{S^2}
	\end{align*}
	Now, if $\Sigma$ has genus at least $1$, we obtain as $K_{g_0}\leq 0$
	\begin{align*}
		&2K_{g_0}\int_{\Sigma}|\alpha_0'|^2d\mathrm{vol}_{g_0}-K_{g_0}\int_{\Sigma}\s{d\phi}{d\w}_g\alpha_0'd\mathrm{vol}_{g_0}\\
		&\leq 2K_{g_0}\int_{\Sigma}|\alpha_0'|^2d\mathrm{vol}_{g_0}+2|K_{g_0}|\int_{\Sigma}|\alpha_0'|^2d\mathrm{vol}_{g_0}+\frac{|K_{g_0}|}{8}\int_{\Sigma}\s{d\phi}{d\w}_g^2d\vg=\frac{|K_{g_0}|}{8}\int_{\Sigma}\s{d\phi}{d\w}_g^2d\mathrm{vol}_{g_0}
	\end{align*}
	\begin{align*}
		\int_{\Sigma}|d\alpha|_g^2d\vg\leq \frac{|K_{g_0}|}{8}\np{|d\w|_g}{\infty}{\Sigma}
	\end{align*}
	and the preceding reasoning gives
	\begin{align*}
		&\left(\int_{S^2}|d\alpha_0'|^2d\vg\right)^{\frac{1}{2}}\leq \left(4\left(\int_{\Sigma}|d\alpha|_g^2d\vg\right)^{\frac{1}{2}}+\sqrt{W(\phi)}+\sqrt{W(\phi)-2\pi\chi(\Sigma)}\right)\np{|d\w|_g}{\infty}{\Sigma}\\
		&+\frac{\sqrt{\pi|\chi(\Sigma)|}}{2}\left(\int_{\Sigma}\s{d\phi}{d\w}_g^2d\mathrm{vol}_{g_0}\right)^{\frac{1}{2}}\\
		&\leq \left(\sqrt{\pi|\chi(\Sigma)|}+\sqrt{W(\phi)}+\sqrt{W(\phi)-2\pi\chi(\Sigma)}+4\left(\int_{\Sigma}|d\alpha|^2d\vg\right)^{\frac{1}{2}}\right)\np{|d\w|_g}{\infty}{\Sigma}.
	\end{align*}
	This concludes the proof of the theorem.
\end{proof}

\subsection{Second derivative of the Onofri energy and main estimate}

\begin{lemme}\label{d2k}
	Let $\Sigma$ be a closed Riemann surface, $\phi\in \mathscr{E}(\Sigma,\R^n)$ be a weak immersion, and $\w\in W^{2,2}_{\phi}\cap W^{1,\infty}(\Sigma,T\R^n)$ be an admissible \emph{normal} variation and $\epsilon>0$ be small enough such that for all $t\in (-\epsilon,\epsilon)$, the map $\phi_t=\phi+t\w:\Sigma\rightarrow \R^n$ be a weak immersion and denote by $g_t=\phi_t^{\ast}g_{\,\R^n}$ the pull-back metric. Then we have
	\begin{align}
	&\frac{d}{dt}(K_{g_t}d\mathrm{vol}_{g_t})_{|t=0}=d\,\Im\left(2\s{\H}{\partial\w}-2\,g^{-1}\otimes\left(\h_0\totimes\bar{\partial}\w\right)\right)\label{k1}\\
	&\frac{d^2}{dt^2}\left(K_{g_t}d\mathrm{vol}_{g_t}\right)=d\,\Im\,\bigg(2\s{\Delta_g^{\perp}\w+4\,\Re\left(g^{-2}\otimes\left(\bar{\partial}\phi\totimes\bar{\partial}\w\right)\otimes \h_0\right)}{\partial\w}-\partial\left(|\D^{\perp}\w|^2_g\right)\nonumber\\
	&-2\s{d\phi}{d\w}_g\left(\s{\H}{\partial \w}-\,g^{-1}\otimes\left(\h_0\totimes\bar{\partial}\w\right)\right)
	-8\,g^{-1}\otimes \left(\partial\phi\totimes\partial\w\right)\otimes \s{\H}{\bar{\partial}\w}\bigg)\label{k2}.
	\end{align}
	In particular, we have
	\begin{align*}
	\frac{d}{dt}(K_{g_t}d\mathrm{vol}_{g_t})_{|t=0}\in H^{-1}(\Sigma),
	\end{align*}
	where $H^{-1}(\Sigma)=\left(W^{1,2}(\Sigma)\right)'$ is the dual space of the Sobolev space $W^{1,2}(\Sigma)$ with respect to any fixed smooth metric on $\Sigma$.
\end{lemme}
\begin{rem}
	However, we do not have in general for a weak immersion the estimate
	\begin{align*}
		\frac{d^2}{dt^2}\left(K_{g_t}d\mathrm{vol}_{g_t}\right)\in H^{-1}(\Sigma).
	\end{align*}
	Indeed, we have rather the estimate (as the quantity under the exterior is in $L^2(\Sigma)$)
	\begin{align*}
		\frac{d^2}{dt^2}\left(K_{g_t}d\mathrm{vol}_{g_t}\right)+d\,\Im\left(\partial |\D^{\perp}\w|_g^2\right)=\frac{d^2}{dt^2}\left(K_{g_t}d\mathrm{vol}_{g_t}\right)+\frac{1}{2}\Delta_g|\D^{\perp}\w|_g^2d\vg\in H^{-1}(\Sigma),
	\end{align*}
	and as for a general admissible variation $\w$, we have the optimal estimate $|\D^{\perp}\w|_g^2\in L^{\infty}(\Sigma)$, we can only define $\Delta_g|\D^{\perp}\w|_g^2d\vg$ as an element of the space of distributions $\mathscr{D}'(\Sigma)$.
\end{rem}
\begin{proof}
	We recall that for all $t\in I$ (see \cite{indexS3})
	\begin{align}\label{qf}
	\frac{d}{dt}\left(K_{g_t}\vol_{g_t}\right)&=d\bigg(\frac{1}{\sqrt{\mathrm{det}(g_t)}}\Big\{\left(-\s{\vec{\I}(\e_1,\e_1)}{\D_{\vec{e_2}}\w}+\s{\vec{\I}(\e_1,\e_2)}{\D_{\vec{e}_1}\w}\right)dx_1\nonumber\\
	&\qquad\qquad\qquad\qquad+\left(\s{\vec{\I}(\e_2,\e_2)}{\D_{\vec{e}_1}\w}-\s{\vec{\I}(\e_2,\e_1)}{\D_{\vec{e}_2}\w}\right)dx_2 \Big\}\bigg).
	\end{align}
	where $\e_i=\p{x_i}\phi$ for $i=1,2$.
	Now, we introduce the Cauchy-Riemann operators
	\begin{align*}
	\p{z}=\frac{1}{2}\left(\d{1}-i\,\d{2}\right),\quad \p{\z}=\frac{1}{2}\left(\d{1}+i\,\d{2}\right)
	\end{align*}
	and the $\partial$ and $\bar{\partial}$ operators defined by
	\begin{align*}
	\partial=\p{z}\left(\,\cdot\,\right)\otimes dz,\quad \bar{\partial}=\p{\z}\left(\,\cdot\,
	\right)\otimes d\z,
	\end{align*}
	so that
	\begin{align*}
	\d{1}=\partial_z+\partial_{\z},\quad \d{2}=-i(\p{z}-\p{\z}).
	\end{align*}
	We also introduce the notations
	\begin{align*}
	\e_z=\p{z}\phi=\frac{1}{2}(\e_1-i\e_2),\quad\e_{\z}=\p{\z}\phi=\frac{1}{2} (\e_1+i\e_2)
	\end{align*}
	so that
	\begin{align*}
	\e_1=\e_z+\e_{\z},\quad \e_2=i\left(\e_z-\e_{\z}\right).
	\end{align*}
	Then we have in this conformal chart
	\begin{align}\label{fcomplexes}
	\left\{\begin{alignedat}{1}
	&\H=2\,e^{-2\lambda}\vec{\I}(\e_z,\e_{\z})\\
	&\H_0=2\,e^{-2\lambda}\vec{\I}(\e_{z},\e_z)\\
	&\h_0=e^{2\lambda}H_0dz^2=2\,\partial^N\partial\phi=2\,\vec{\I}(\e_z,\e_z)dz^2.
	\end{alignedat}\right.
	\end{align}
	Therefore, we obtain the identities
	\begin{align*}
	&\vec{\I}(\e_1,\e_1)=\vec{\I}(\e_z+\e_{\z},\e_z+\e_{\z})=2\,\Re\,\vec{\I}(\e_z,\e_z)+2\vec{\I}(\e_z,\e_{\z}).\\
	&\vec{\I}(\e_2,\e_2)=-2\,\Re\,\vec{\I}(\e_z,\e_z)+2\,\vec{\I}(\e_z,\e_{\z})\\
	&\vec{\I}(\e_1,\e_2)=-2\,\Im\,\vec{\I}(\e_z,\e_z)
	\end{align*}
	To simplify notations in the following computation, we will write (abusively) $\p{z}\vec{w}=\partial$ and $\p{\z}\vec{w}=\bar{\partial}$.	Finally, as we have trivially 
	\begin{align*}
	dx_1=\frac{1}{2}(dz+d\z)\quad dx_2=-\frac{i}{2}(dz-d\z),
	\end{align*}
	we obtain the identity
	\begin{align}\label{q0}	&\left(-\s{\vec{\I}(\e_1,\e_1)}{\D_{\vec{e}_2}\w}+\s{\vec{\I}(\e_1,\e_2)}{\D_{\vec{e}_1}\w}\right)dx_1+\left(\s{\vec{\I}(\e_2,\e_2)}{\D_{\vec{e}_1}\w}-\s{\vec{\I}(\e_2,\e_1)}{\D_{\vec{e}_2}\w}\right)dx_2\nonumber\\
	&=\frac{1}{2}\bigg\{-\s{2\,\Re\,\vec{\I}(\e_z,\e_z)+2\,\vec{\I}(\e_z,\e_{\z})}{i(\partial-\bar{\partial})(dz+d\z)}-\s{2\,\Im(\e_z,\e_z)}{(\partial+\bar{\partial})(dz+d\z)}\nonumber\\
	&+\s{-2\,\Re\,\vec{\I}(\e_z,\e_z)+2\,\vec{\I}(\e_z,\e_{\z})}{(\partial+\bar{\partial})(-i(dz-d\z))}-\s{-2\,\Im\,\vec{\I}(\e_z,\e_z)}{i(\partial-\bar{\partial})(-i(dz-d\z))}\bigg\}\nonumber\\
	&=\s{\,\Re\,\vec{\I}(\e_z,\e_z)+\,\vec{\I}(\e_z,\e_{\z})}{-i(\partial-\bar{\partial})(dz+d\z)}+\s{\,\Im\,\vec{\I}(\e_z,\e_z)}{-(\partial+\bar{\partial})(dz+d\z)}\nonumber\\
	&+\s{-\,\Re\,\vec{\I}(\e_z,\e_z)+\,\vec{\I}(\e_z,\e_{\z})}{-i(\partial+\bar{\partial})(dz-d\z))}+\s{\Im\,\vec{\I}(\e_z,\e_z)}{(\partial-\bar{\partial})(dz-d\z))}
	\end{align}
	We first compute
	\begin{align*}
	-i(\partial-\bar{\partial})(dz+d\z)-i(\partial+\bar{\partial})(dz-d\z)=-2i(\partial dz-\bar{\partial}d\z)=4\,\Im(\partial dz)=4\,\Im(\partial\w).
	\end{align*}
	Then we have
	\begin{align}\label{q1}
	&\s{\,\Re\,\vec{\I}(\e_z,\e_z)}{-i(\partial-\bar{\partial})(dz+d\z)}+\s{\,\Im\,\vec{\I}(\e_z,\e_z)}{-(\partial+\bar{\partial})(dz+d\z)}\nonumber\\
	&+\s{-\,\Re\,\vec{\I}(\e_z,\e_z)}{-i(\partial+\bar{\partial})(dz-d\z))}+\s{\Im\,\vec{\I}(\e_z,\e_z)}{(\partial-\bar{\partial})(dz-d\z))}\nonumber\\
	&=\s{\Re\,\vec{\I}(\e_z,\e_z)}{-2i(\partial d\z-\bar{\partial}dz)}+\s{\Im\,\vec{\I}(\e_z,\e_z)}{-2(\partial d\z+\bar{\partial} dz)}\nonumber\\
	&=4\s{\Re\,\vec{\I}(\e_z,\e_z)}{\Im(\partial_z\w \,d\z)}-4\s{\Im\,\vec{\I}(\e_z,\e_z)}{\Re(\D_{\p{z}}\w\, d\z)}.
	\end{align}
	Now, we observe that for all $a,b\in \C^n$, we have
	\begin{align*}
	-\s{\Re(a)}{\Im(b)}+\s{\Im(a)}{\Re(b)}=\Im\s{a}{\bar{b}},
	\end{align*}
	so that
	\begin{align}\label{q2}
	&\s{\,\Re\,\vec{\I}(\e_z,\e_z)}{-i(\partial-\bar{\partial})(dz+d\z)}+\s{\,\Im\,\vec{\I}(\e_z,\e_z)}{-(\partial+\bar{\partial})(dz+d\z)}\nonumber\\
	&+\s{-\,\Re\,\vec{\I}(\e_z,\e_z)}{-i(\partial+\bar{\partial})(dz-d\z))}+\s{\Im\,\vec{\I}(\e_z,\e_z)}{(\partial-\bar{\partial})(dz-d\z))}\nonumber\\
	&=-4\,\Im\left(\s{\vec{\I}(\e_z,\e_{z})}{\p{\z}\w\,dz}\right).
	\end{align}
	Finally, we obtain by \eqref{qf}, \eqref{q0}, \eqref{q1} and \eqref{q2} that for all $t\in I$,
	\begin{align}\label{f1}
	\frac{d}{dt}\left(K_{g_t}d\mathrm{vol}_{g_t}\right)=d\,\Im\left(\frac{1}{\sqrt{\det(g_t)}}\left(4\s{\vec{\I}(\e_z,\e_{\z})}{\D_{\p{z}}\w\,dz} -4\s{\vec{\I}(\e_z,\e_z)}{\p{\z}\w\,dz}\right)\right)
	\end{align}
	so that
	\begin{align}
	\frac{d}{dt}\left(K_{g_t}d\mathrm{vol}_{g_t}\right)_{|t=0}=d\,\Im\left(2\s{\H}{\partial\w}-2\,g^{-1}\otimes \left(\h_0\totimes\bar{\partial}\w\right)\right).
	\end{align}
	Now, we introduce the second order differential operator $\D^2$ on the normal bundle (see \cite{federer}, $5.4.12$ for the standard operator and \cite{indexS3}) defined for all $X,Y\in T\Sigma$ by
	\begin{align}\label{opd2}
	\D^2_{X,Y}=\D^{\perp}_X\D^{\perp}_Y-\D^{\perp}_{(\D_XY)^{\top}}.
	\end{align}
	Notice that this operator is not symmetric in general, as
	\begin{align*}
	R^{\perp}(X,Y)=\D^{2}_{X,Y}-\D^{2}_{Y,X}.
	\end{align*}
	In particular, we have
	\begin{align*}
	\left(\left(\D_{X}\D_{Y}-\D_{(\D_{X}Y)^{\top}}\right)\w\right)^{\perp}=\D^2_{X,Y}\w+e^{-2\lambda}\sum_{i=1}^{2}\s{\D_{Y}\w}{\e_k}\vec{\I}(X,\e_k).
	\end{align*}
	We recall  the formulas (see \cite{indexS3})
	\begin{align*}
	&\frac{d}{dt}\left(\sqrt{\mathrm{det}(g_t)}\right)_{|t=0}=\s{d\phi}{d\w}_g\sqrt{\mathrm{det}(g)}\\
	&\D_{\frac{d}{dt}}^{\perp}\vec{\I}(\e_i,\e_j)=\D^{2}_{\e_i,\e_j}\w+e^{-2\lambda}\sum_{k=1}^{2}\s{\D_{\e_j}\w}{\e_k}\vec{\I}(\e_j,\e_k).
	\end{align*}
	Now, by \eqref{f1}, we obtain as $\left(\D_{\frac{d}{dt}}\w_t\right)_{|t=0}=0$ and $\s{d\phi}{d\w}_g=-2\s{\H}{\w}$
	\begin{align}\label{q3}
	\frac{d^2}{dt^2}\left(K_{g_t}d\mathrm{vol}_{g_t}\right)_{|t=0}&=d\,\Im\bigg(-\s{d\phi}{d\w}_g\left(2\s{\H}{\partial\w}-2\,g^{-1}\otimes \left(\h_0\totimes\bar{\partial}\w\right)\right)
	+e^{-2\lambda}\Big(4\s{\D_{\frac{d}{dt}}^{\perp}\vec{\I}(\e_z,\e_{\z})}{\D_{\p{z}}^{\perp}\w\,dz}\nonumber\\
	&+4\s{\vec{\I}(\e_z,\e_{\z})}{\D_{\frac{d}{dt}}^{\perp}\D_{\p{z}}^{\perp}\w\,dz}
	-4\s{\D_{\frac{d}{dt}}^{\perp}\vec{\I}(\e_z,\e_z)}{\D_{\p{\z}}^{\perp}\w\,dz}-4\s{\vec{\I}(\e_z,\e_z)}{\D_{\frac{d}{dt}}^{\perp}\D_{\p{z}}^{\perp}\w\,dz}\Big)\bigg)\nonumber\\
	&=d\,\Im\bigg(4\s{\H}{\w}\s{\H}{\partial \w}-4\s{\H}{\w}\,g^{-1}\otimes \left(\h_0\totimes\bar{\partial}\w\right)
	+e^{-2\lambda}\Big(4\s{\D_{\frac{d}{dt}}^{\perp}\vec{\I}(\e_z,\e_{\z})}{\D_{\p{z}}^{\perp}\w\,dz}\nonumber\\
	&+4\s{\vec{\I}(\e_z,\e_{\z})}{\D_{\frac{d}{dt}}^{\perp}\D_{\p{z}}^{\perp}\w\,dz}
	-4\s{\D_{\frac{d}{dt}}^{\perp}\vec{\I}(\e_z,\e_z)}{\D_{\p{\z}}^{\perp}\w\,dz}-4\s{\vec{\I}(\e_z,\e_z)}{\D_{\frac{d}{dt}}^{\perp}\D_{\p{\z}}^{\perp}\w\,dz}\Big)\bigg)\nonumber\\
	&=d\,\Im\left(4\s{\H}{\w}\s{\H}{\partial\w}-4\s{\H}{\w}\,g^{-1}\otimes\left(\h_0\totimes\bar{\partial}\w\right)+\mathrm{(I)}+\mathrm{(II)}+\mathrm{(III)}+\mathrm{(IV)}\right).
	\end{align}
	Now, notice that
	\begin{align*}
	\D_{\frac{d}{dt}}\D_{\p{z}}\w=\D_{\p{z}}\D_{\frac{d}{dt}}\w=0,
	\end{align*}
	so that
	\begin{align}\label{q4}
	\D_{\frac{d}{dt}}^{\perp}\D_{\p{z}}^{\perp}\w=-\left(\D_{\frac{d}{dt}}\D_{\p{z}}^{\top}\w\right)^{\perp}.
	\end{align}
	By an immediate formula using the conformity relation and as $\w$ is a normal variation (we can use this relation without creating non-integrable terms as the formula we obtained for the first derivative of the Gauss curvature is already in divergence form, we refer to \cite{riviere1} for more informations on this point), we obtain
	\begin{align*}
	\D_{\p{z}}^{\top}\w&=2e^{-2\lambda}\s{\D_{\p{z}}\w}{\e_{\z}}\e_z+2e^{-2\lambda}\s{\D_{\p{z}}\w}{\e_{z}}\e_{\z}=-2e^{-2\lambda}\s{\w}{\vec{\I}(\e_z,\e_{\z})}\e_{z}-2e^{-2\lambda}\s{\w}{\vec{\I}(\e_{z},\e_{z})}\e_{\z}\\
	&=-\s{\w}{\H}\e_z-\s{\w}{\H_0}\e_{\z}
	\end{align*}
	and using \eqref{q4}, we have
	\begin{align}
	\D_{\frac{d}{dt}}^{\perp}\D_{\p{z}}^{\perp}\w&=\s{\w}{\H}\D_{\e_z}^{\perp}\w+\s{\w}{\H_0}\D_{\e_{\z}}^{\perp}\w
	\end{align}
	By taking the conjugate, we                                                                                                                                                                                                                                                                                                                                                                                                                                                                                                                                                                                                                                                                                                                                                                                                                                       obtain 
	\begin{align}\label{q4bis}
	\D_{\frac{d}{dt}}^{\perp}\D_{\e_{\z}}^{\perp}\w=\s{\w}{\bar{\H_0}}\D_{\e_z}^{\perp}\w+\s{\w}{\H}\D_{\e_{\z}}^{\perp}\w.
	\end{align}
	Therefore, we have by \eqref{q4} and \eqref{q4bis} the identities
	\begin{align}\label{q01}
	\mathrm{(II)}=4\,e^{-2\lambda}\s{\vec{\I}(\e_z,\e_{\z})}{\D_{\frac{d}{dt}}^{\perp}\D_{\e_z}^{\perp}\w}dz&=2\s{\w}{\H}\s{\H}{\D_{\e_{z}}\w}dz+2\s{\w}{\H_0}\s{\H}{\D_{\p{\z}}\w}dz\nonumber\\
	&=2\s{\H}{\w}\s{\H}{\partial\w}+2\,g^{-1}\otimes \s{\w}{\h_0}\otimes \s{\H}{\bar{\partial}\w}
	\end{align}
	and
	\begin{align}\label{q02}
	\mathrm{(IV)}=-4\,e^{-2\lambda}\s{\vec{\I}(\e_z,\e_{z})}{\D_{\frac{d}{dt}}^{\perp}\D_{\e_{\z}}^{\perp}\w}dz&=-4\,e^{-2\lambda}\s{\vec{\I}(\e_z,\e_z)}{\D_{\e_z}^{\perp}\w}\s{\bar{\H_0}}{\w}-4\,e^{-2\lambda}\s{\H}{\w}\s{\vec{\I}(\e_z,\e_z)}{\D_{\e_{\z}}\w}\nonumber\\
	&=-2\,g^{-2}\otimes \s{\bar{\h_0}}{\w}\otimes\left(\h_0\totimes \partial\w\right)-2\s{\H}{\w}\,g^{-1}\otimes\left(\h_0\totimes\bar{\partial}\w\right).
	\end{align}
	Then, we have using the normality of $\w$
	\begin{align}\label{q5}
	&4\D_{\frac{d}{dt}}^{\perp}\vec{\I}(\e_z,\e_z)=\D_{\frac{d}{dt}}^{\perp}\left(\vec{\I}(\e_1,\e_1)-\vec{\I}(\e_2,\e_2)-2i\vec{\I}(\e_1,\e_2)\right)\nonumber\\
	&=\left(\D^2_{\e_1,\e_1}-\D^2_{\e_2,\e_2}-i\D_{\e_1,\e_2}^2-i\D_{\e_2,\e_1}^{2}\right)\w\nonumber\\
	&+e^{-2\lambda}\sum_{k=1}^{2}\left(\s{\D_{\e_1}\w}{\e_k}\vec{\I}(\e_1,\e_k)-\s{\D_{\e_2}\w}{\e_k}\vec{\I}(\e_2,\e_k)-i\s{\D_{\e_2}\w}{\e_k}\vec{\I}(\e_1,\e_k)-i\s{\D_{\e_1}\w}{\e_k}\vec{\I}(\e_2,\e_k)\right)\nonumber\\
	&=4\,\D^2_{\e_z,\e_z}\w+e^{-2\lambda}\Big\{-\s{\w}{\vec{\I}(\e_1,\e_1)}\vec{\I}(\e_1,\e_1)-\s{\w}{\vec{\I}(\e_1,\e_2)}\vec{\I}(\e_1,\e_2)+\s{\vec{w}}{\vec{\I}(\e_1,\e_2)}\vec{\I}(\e_1,\e_2)\nonumber\\
	&+\s{\w}{\vec{\I}(\e_2,\e_2)}\vec{\I}(\e_2,\e_2)
	+i\s{\w}{\vec{\I}(\e_1,\e_2)}\vec{\I}(\e_1,\e_1)+i\s{\vec{w}}{\vec{\I}(\e_2,\e_2)}\vec{\I}(\e_1,\e_2)+i\s{\vec{w}}{\vec{\I}(\e_1,\e_1)}\vec{\I}(\e_1,\e_2)\nonumber\\
	&+i\s{\vec{w}}{\vec{\I}(\e_1,\e_2)}\vec{\I}(\e_2,\e_2)\Big\}\nonumber\\
	&=\D^2_{\e_z,\e_z}\w+e^{-2\lambda}\Big\{-\s{\w}{\vec{\I}(\e_1,\e_1)}\vec{\I}(\e_1,\e_1)+\s{\w}{\vec{\I}(\e_2,\e_2)}\vec{\I}(\e_2,\e_2)
	+i\s{\w}{\vec{\I}(\e_1,\e_2)}\vec{\I}(\e_1,\e_1)\nonumber\\
	&+i\s{\vec{w}}{\vec{\I}(\e_2,\e_2)}\vec{\I}(\e_1,\e_2)+i\s{\vec{w}}{\vec{\I}(\e_1,\e_1)}\vec{\I}(\e_1,\e_2)+i\s{\vec{w}}{\vec{\I}(\e_1,\e_2)}\vec{\I}(\e_2,\e_2)\Big\}.
	\end{align}
	Using the previous formulas \eqref{fcomplexes}, we obtain from \eqref{q5} 
	\begin{align}\label{q6}
	&4\D_{\frac{d}{dt}}^{\perp}\vec{\I}(\e_z,\e_z)=4\,\D^2_{\e_z,\e_z}\w+e^{-2\lambda}\Big\{-\s{\w}{2\,\Re\,\vec{\I}(\e_z,\e_z)+2\vec{\I}(\e_z,\e_{\z})}\left(2\,\Re\,\vec{\I}(\e_z,\e_z)+2\,\vec{\I}(\e_z,\e_{\z})\right)\nonumber\\
	&+\s{\w}{-2\,\Re\,\vec{\I}(\e_z,\e_z)+2\vec{\I}(\e_z,\e_z)}\left(-2\,\Re\,\vec{\I}(\e_z,\e_z)+2\vec{\I}(\e_z,\e_{\z})\right)
	\nonumber\\
	&+i\s{\w}{-2\,\Im\,\vec{\I}(\e_z,\e_z)}\left(2\,\Re\,\vec{\I}(\e_z,\e_z)+2\vec{\I}(\e_z,\e_{\z})\right)+i\s{\w}{-2\,\Re\,\vec{\I}(\e_z,\e_z)+2\,\vec{\I}(\e_z,\e_{\z})}\left(-2\,\Im\,\vec{\I}(\e_z,\e_z)\right)\nonumber\\
	&+i\s{\w}{2\,\Re\,\vec{\I}(\e_z,\e_z)+2\,\vec{\I}(\e_z,\e_{\z})}\left(-2\,\Im\,\vec{\I}(\e_z,\e_z)\right)+i\s{\w}{-2\,\Im\,\vec{\I}(\e_z,\e_z)}\left(-2\,\Re\,\vec{\I}(\e_z,\e_z)+2\,\vec{\I}(\e_z,\e_{\z})\right)\Big\}\nonumber\\
	&=4\,\D^2_{\e_z,\e_z}\w+4\,e^{-2\lambda}\Big\{\s{\w}{\vec{\I}(\e_z,\e_{\z})}\left(-\Re\,\vec{\I}(\e_z,\e_z)-\vec{\I}(\e_z,\e_{\z})-\Re\,\vec{\I}(\e_z,\e_z)+\vec{\I}(\e_z,\e_{\z})-2i\,\Im\,\vec{\I}(\e_z,\e_z)\right)\nonumber\\
	&+\s{\w}{\Re\,\vec{\I}(\e_z,\e_z)}\left(-\Re\,\vec{\I}(\e_z,\e_z)-\vec{\I}(\e_z,\e_{\z})+\Re\,\vec{\I}(\e_z,\e_z)-\vec{\I}(\e_z,\e_{\z})+i\,\Im\,\vec{\I}(\e_z,\e_z)-i\Im\,\vec{\I}(\e_z,\e_z)\right)\nonumber\\
	&+i\s{\w}{\Im\,\vec{\I}(\e_z,\e_z)}\left(-\Re\,\vec{\I}(\e_z,\e_z)-\vec{\I}(\e_z,\e_z)+\Re\,\vec{\I}(\e_z,\e_z)-\vec{\I}(\e_z,\e_{\z})\right)\Big\}\nonumber\\
	&=4\,\D^2_{\e_z,\e_z}\w-8e^{-2\lambda}\s{\w}{\vec{\I}(\e_z,\e_{\z})}\vec{\I}(\e_z,\e_z)-8e^{-2\lambda}\s{\w}{\vec{\I}(\e_z,\e_z)}\vec{\I}(\e_{z},\e_{\z})\nonumber\\
	&=4\,\D^2_{\e_z,\e_z}\w-2e^{2\lambda}\s{\w}{\H}\H_0-2e^{2\lambda}\s{\w}{\H_0}\H.
	\end{align}
	Now, we have
	\begin{align*}
	&\s{\D_{\e_z}\e_z}{\e_z}=\frac{1}{2}\p{z}\left(\s{\e_z}{\e_z}\right)=0\\
	&\s{\e_z}{\D_{\e_z}\e_{\z}}=\frac{e^{2\lambda}}{2}\s{\e_z}{\H}=0,
	\end{align*}
	so that
	\begin{align*}
	\s{\D_{\e_z}\e_z}{\e_{\z}}=\p{z}\left(\s{\e_z}{\e_{\z}}\right)=\frac{1}{2}\p{z}\left(e^{2\lambda}\right).
	\end{align*}
	Therefore, we obtain
	\begin{align*}
	\D^{\top}_{\e_z}\e_z=2e^{-2\lambda}\s{\D_{\e_z}\e_z}{\e_{\z}}\e_z+2e^{-2\lambda}\s{\D_{\e_z}\e_z}{\e_z}\e_{\z}=e^{-2\lambda}\p{z}\left(e^{2\lambda}\right)\e_z.
	\end{align*}
	Now, we finally have
	\begin{align}\label{q7}
	\D_{\e_z,\e_z}^2\w=\D_{\e_z}^{\perp}\D_{\e_z}^{\perp}\w-\D_{\D_{\e_z}^{\top}\e_z}^{\perp}\w=\D_{\e_z}^{\perp}\D_{\e_z}^{\perp}\w-e^{-2\lambda}\p{z}(e^{2\lambda})\D_{\e_z}^{\perp}\w.
	\end{align}
	This implies that
	\begin{align}\label{q8}
	&e^{-2\lambda}\s{\D^{2}_{\e_z,\e_z}\w}{\D_{\e_{\z}}^{\perp}\w}dz=\bigg\{e^{-2\lambda}\s{\D_{\e_z}^{\perp}\D_{\e_z}^{\perp}\w}{\D_{\e_{\z}}^{\perp}\w}-e^{-4\lambda}\p{z}(e^{2\lambda})\s{\D_{\e_z}^{\perp}\w}{\D_{\e_{\z}}^{\perp}\w}\bigg\}dz\nonumber\\
	&=\bigg\{\p{z}\left(e^{-2\lambda}\s{\D_{\e_{z}}^{\perp}\w}{\D_{\e_{\z}}^{\perp}\w}\right)-\p{z}(e^{-2\lambda})\s{\D_{\e_z}^{\perp}w}{\D_{\e_{\z}}^{\perp}\w}-e^{-2\lambda}\s{\D_{\e_z}^{\perp}\w}{\D_{\e_z}^{\perp}\D_{\e_{\z}}^{\perp}\w}\nonumber\\
	&-e^{-4\lambda}\p{z}(e^{2\lambda})\s{\D_{\e_z}^{\perp}\w}{\D_{\e_{\z}}^{\perp}\w}\bigg\}dz
	=\,\partial\Big(\left|{\partial}^{\perp}\w\right|^2_g\Big)-e^{-2\lambda}\s{\D_{\e_z}^{\perp}\D_{\e_{\z}}^{\perp}\w}{\D_{\e_z}^{\perp}\w}dz.
	\end{align}
	Now, one checks easily by (see \cite{classification}) that
	\begin{align*}
	\Delta_g^{\perp}=2e^{-2\lambda}\left(\D^2_{\e_z,\e_{\z}}+\D^2_{\e_{\z},\e_z}\right)=2e^{-2\lambda}\left(\D^{\perp}_{\e_z}\D_{\e_{\z}}^{\perp}+\D^{\perp}_{\e_{\z}}\D_{\e_{z}}^{\perp}\right).
	\end{align*}
	Furthermore, by the definition of the normal curvature and thanks of the Ricci equation, we have
	\begin{align*}
	\D^{\perp}_{\e_z}\D_{\e_{\z}}^{\perp}\w&=\D_{\e_{\z}}^{\perp}\D_{\e_z}^{\perp}\w+R^{\perp}(\e_z,\e_{\z})\w\\
	&=\D_{\e_{\z}}^{\perp}\D_{\e_z}^{\perp}\w+\vec{\I}(\D^{\top}_{\e_z}\w,\e_{\z})-\vec{\I}(\e_z,\D^{\top}_{\e_{\z}}\w)
	\end{align*}
	Recalling that
	\begin{align*}
	\D^{\top}_{\e_z}\w=2e^{-2\lambda}\s{\D_{\e_z}\w}{\e_{\z}}\e_z+2e^{-2\lambda}\s{\D_{\e_z}\w}{\e_{z}}\e_{\z}=-\s{\w}{\H}\e_z-\s{\w}{\H_0}\e_{\z},
	\end{align*}
	we obtain
	\begin{align*}
	\D^{\perp}_{\e_z}\D_{\e_{\z}}^{\perp}\w&=\D_{\e_{\z}}^{\perp}\D_{\e_z}^{\perp}\w-\s{\w}{\H}\vec{\I}(\e_z,\e_{\z})-\s{\w}{\H_0}\vec{\I}(\e_{\z},\e_{\z})+\s{\w}{\H}\vec{\I}(\e_z,\e_{\z})+\s{\w}{\bar{\H_0}}\vec{\I}(\e_z,\e_z)\\
	&=\D_{\e_{\z}}^{\perp}\D_{\e_z}^{\perp}\w-\frac{e^{2\lambda}}{2}\s{\w}{\H_0}\bar{\H_0}+\frac{e^{2\lambda}}{2}\s{\w}{\bar{\H_0}}\H_0.
	\end{align*}
	This implies that
	\begin{align}\label{q9}
	&4e^{-2\lambda}\D_{\e_z}^{\perp}\D_{\e_{\z}}^{\perp}\w=2e^{-2\lambda}\left(\D_{\e_z,\e_{\z}}^{2}+\D_{\e_{\z},\e_z}^{2}\right)\w+\s{\w}{\bar{\H_0}}\H_0-\s{\w}{\H_0}\bar{\H_0}\nonumber\\
	&=\Delta_g^{\perp}\w+g^{-2}\otimes\left(\s{\w}{\bar{\h_0}}\otimes \h_0-\s{\w}{\h_0}\otimes \bar{\h_0}\right)
	\end{align}
	Now, by  \eqref{q8} and \eqref{q9}, we finally have
	\begin{align*}
	&4e^{-2\lambda}\s{\D^2_{\e_z,\e_z}\w}{\D_{\e_{\z}}^{\perp}\w}dz=4\,\partial\left(|\partial^{\perp}\w|^2_g\right)-\s{\Delta_g^{\perp}\w}{\partial\w}+g^{-2}\otimes \s{\h_0}{\w}\otimes \left(\bar{\h_0}\totimes {\partial}\w\right)\\
	&-g^{-2}\otimes \s{\bar{\h_0}}{\w}\otimes \left(\h_0\totimes{\partial}\w\right)\\
	&=\partial\left(|\D^{\perp}\w|^2_g\right)-\s{\Delta_g^{\perp}\w}{\partial\w}+g^{-2}\otimes \s{\h_0}{\w}\otimes (\bar{\h_0}\totimes {\partial}\w)-g^{-2}\otimes \s{\bar{\h_0}}{\w}\otimes \left(\h_0\totimes{\partial}\w\right).
	\end{align*}
	Therefore, we have
	\begin{align*}
	&4e^{-2\lambda}\s{\D_{\frac{d}{dt}}^{\perp}\vec{\I}(\e_z,\e_z)}{\D_{\e_{\z}}^{\perp}\w}dz=\partial\left(|\D^{\perp}\w|^2_g\right)-\s{\Delta_g^{\perp}\w}{\partial\w}+g^{-2}\otimes \s{\h_0}{\w}\otimes (\bar{\h_0}\totimes {\partial}\w)\nonumber\\
	&-g^{-2}\otimes \s{\bar{\h_0}}{\w}\otimes \left(\h_0\totimes{\partial}\w\right)
	-2\,\s{\w}{\H}\,g^{-1}\otimes \left(\h_0\totimes\bar{\partial}\w\right)
	-2\,g^{-1}\otimes \s{\h_0}{\w}\otimes \s{\H}{\bar{\partial} \w}\\
	\end{align*}
	and
	\begin{align}\label{q03}
	\mathrm{(III)}&=-4e^{-2\lambda}\s{\D_{\frac{d}{dt}}^{\perp}\vec{\I}(\e_z,\e_z)}{\D_{\e_{\z}}^{\perp}\w}dz=-\partial\left(|\D^{\perp}\w|^2_g\right)+\s{\Delta_g^{\perp}\w}{\partial\w}-g^{-2}\otimes \s{\h_0}{\w}\otimes \left(\bar{\h_0}\totimes {\partial}\w\right)\nonumber\nonumber\\
	&+g^{-2}\otimes \s{\bar{\h_0}}{\w}\otimes \left(\h_0\totimes{\partial}\w\right)
	+2\,\s{\H}{\w}\,g^{-1}\otimes \left(\h_0\totimes\bar{\partial}\w\right)
	+2\,g^{-1}\otimes \s{\h_0}{\w}\otimes \s{\H}{\bar{\partial} \w}
	\end{align}
	Now, we compute
	\begin{align}\label{q10}
	4\,\D_{\frac{d}{dt}}^{\perp}\vec{\I}(\e_z,\e_{\z})&=\D_{\frac{d}{dt}}^{\perp}\left(\vec{\I}(\e_1,\e_1)+\vec{\I}(\e_2,\e_2)\right)=\left(\D_{\e_1,\e_1}^{\perp}+\D_{\e_2,\e_2}^{\perp}\right)\w+e^{-2\lambda}\sum_{k=1}^{2}\s{\D_{\e_1}\w}{\e_k}\vec{\I}(\e_1,\e_k)\nonumber\\
	&+e^{-2\lambda}\sum_{k=1}^{2}\s{\D_{\e_2}\w}{\e_k}\vec{\I}(\e_2,\e_k)\nonumber\\
	&=e^{2\lambda}\Delta_g^{\perp}\w+e^{-2\lambda}\sum_{k=1}^{2}\s{\D_{\e_1}\w}{\e_k}\vec{\I}(\e_1,\e_k)
	+e^{-2\lambda}\sum_{k=1}^{2}\s{\D_{\e_2}\w}{\e_k}\vec{\I}(\e_2,\e_k)\nonumber\\
	&=e^{2\lambda}\Delta_g^{\perp}\w-e^{-2\lambda}\left(\s{\w}{\vec{\I}(\e_1,\e_1)}\vec{\I}(\e_1,\e_1)+2\s{\w}{\vec{\I}(\e_1,\e_2)}\vec{\I}(\e_1,\e_2)+\s{\w}{\vec{\I}(\e_2,\e_2)}\vec{\I}(\e_2,\e_2)\right).
	\end{align}
	By \eqref{fcomplexes}, we have
	\begin{align}\label{q11}
	&\s{\w}{\vec{\I}(\e_1,\e_1)}\vec{\I}(\e_1,\e_1)+2\s{\w}{\vec{\I}(\e_1,\e_2)}\vec{\I}(\e_1,\e_2)+\s{\w}{\vec{\I}(\e_2,\e_2)}\vec{\I}(\e_2,\e_2)\nonumber\\
	&=\s{\w}{2\,\Re\,\vec{\I}(\e_z,\e_z)+2\,\vec{\I}(\e_z,\e_{\z})}\left(2\,\Re\,\vec{\I}(\e_z,\e_z)+2\,\vec{\I}(\e_z,\e_{\z})\right)+2\s{\w}{-2\,\Im\,\vec{\I}(\e_z,\e_z)}\left(-2\,\Im\,\vec{\I}(\e_z,\e_z)\right)\nonumber\\
	&+\s{\w}{-2\,\Re\,\vec{\I}(\e_z,\e_z)+2\,\vec{\I}(\e_z,\e_{\z})}\left(-2\,\Re\,\vec{\I}(\e_z,\e_z)+2\,\vec{\I}(\e_z,\e_{\z})\right)
	\nonumber\\
	&=4\Big\{\s{\w}{\vec{\I}(\e_z,\e_{\z})}\left(\Re\,\vec{\I}(\e_z,\e_z)+\vec{\I}(\e_z,\e_{\z})-\Re\,\vec{\I}(\e_z,\e_z)+\vec{\I}(\e_z,\e_{\z})\right)\nonumber\\
	&+\s{\w}{\Re\,\vec{\I}(\e_z,\e_z)}\left(\Re\,\vec{\I}(\e_z,\e_z)+\vec{\I}(\e_z,\e_{\z})+\Re\,\vec{\I}(\e_z,\e_z)-\vec{\I}(\e_z,\e_{\z})\right)
	+2\s{\w}{\Im\,\vec{\I}(\e_z,\e_z)}\Im\,\vec{\I}(\e_z,\e_z)\Big\}\nonumber\\
	&=8\s{\w}{\vec{\I}(\e_z,\e_{\z})}\vec{\I}(\e_z,\e_{\z})+8\s{\w}{\Re\,\vec{\I}(\e_z,\e_z)}\Re\,\vec{\I}(\e_z,\e_{z})+8\s{\w}{\Im\,\vec{\I}(\e_z,\e_{z})}\Im\,\vec{\I}(\e_z,\e_z)\nonumber\\
	&=8\s{\w}{\vec{\I}(\e_z,\e_{\z})}\vec{\I}(\e_z,\e_{\z})+8\,\Re\left(\s{\w}{\vec{\I}(\e_{\z},\e_{\z})}\vec{\I}(\e_z,\e_z)\right)\nonumber\\
	&=2e^{4\lambda}\s{\w}{\H}\H+2e^{4\lambda}\,\Re\left(\s{\w}{\bar{\H_0}}\H_0\right)
	\end{align}
	where we use in the last line the following identity true for all $a,b\in \C^n$
	\begin{align*}
	\s{\,\cdot\,}{\Re(a)}\Re(a)+\s{\,\cdot\,}{\Im(a)}\Im(a)=\frac{1}{4}\s{\,\cdot\,}{a+\bar{a}}\left(a+\bar{a}\right)-\frac{1}{4}\s{\,\cdot\,}{a-\bar{a}}\left(a-\bar{a}\right)=\,\Re(\s{\,\cdot\,}{\bar{a}}a).
	\end{align*}
	By \eqref{q10} and \eqref{q11}, we obtain
	\begin{align*}
	4\,\D_{\frac{d}{dt}}^{\perp}\vec{\I}(\e_z,\e_{\z})=e^{2\lambda}\left(\Delta_g^{\perp}\w-2\,\s{\w}{\H}\H-2\,\Re\left(\s{\w}{\bar{\H_0}}\H_0\right)\right)
	\end{align*}
	and
	\begin{align}
	4\,e^{-2\lambda}\,\D_{\frac{d}{dt}}^{\perp}\vec{\I}(\e_z,\e_{\z})=\Delta_g^{\perp}\w-2\s{\w}{\H}\H-2\,\Re\left(g^{-2}\otimes \s{\w}{\bar{\h_0}}\otimes\h_0\right).
	\end{align}
	This implies that
	\begin{align}\label{q04}
	\mathrm{(I)}=4\,e^{-2\lambda}\s{\D_{\frac{d}{dt}}^{\perp}\vec{\I}(\e_z,\e_{\z})}{\D_{\e_z}^{\perp}\w}dz &=\s{\Delta_g^{\perp}\w}{\partial\w}-2\s{\H}{\w}\s{\H}{\partial\w}-g^{-2}\otimes \s{\bar{\h_0}}{\w}\otimes \left(\h_0\totimes \partial \w \right)\nonumber\\
	&-g^{-2}\otimes \s{\h_0}{\w}\otimes \left(\bar{\h_0}\totimes \partial \w\right).
	\end{align}
	Finally, we obtain by \eqref{q3}, \eqref{q01}, \eqref{q02}, \eqref{q03}, \eqref{q04}
	\begin{align}
	&\frac{d^2}{dt^2}\left(K_{g_t}d\mathrm{vol}_{g_t}\right)
	=d\,\Im\bigg(4\s{\H}{\w}\s{\H}{\partial\w}-4\s{\H}{\w}\,g^{-1}\otimes \left(\h_0\totimes\bar{\partial}\h_0\right)\nonumber\\
	&+\s{\Delta_g^{\perp}\w}{\partial\w}-2\s{\H}{\w}\s{\H}{\partial\w}-g^{-2}\otimes \s{\bar{\h_0}}{\w}\otimes \left(\h_0\totimes \partial \w \right)
	-g^{-2}\otimes \s{\h_0}{\w}\otimes \left(\bar{\h_0}\totimes \partial \w\right)\nonumber\\
	&+2\s{\H}{\w}\s{\H}{\partial\w}+2\,g^{-1}\otimes \s{\w}{\h_0}\otimes \s{\H}{\bar{\partial}\w}\nonumber\\
	&-\partial\left(|\D^{\perp}\w|^2_g\right)+\s{\Delta_g^{\perp}\w}{\partial\w}-g^{-2}\otimes \s{\h_0}{\w}\otimes \left(\bar{\h_0}\totimes {\partial}\w\right)\nonumber\\
	&+g^{-2}\otimes \s{\bar{\h_0}}{\w}\otimes \left(\h_0\totimes{\partial}\w\right)
	+2\,\s{\w}{\H}\,g^{-1}\otimes \left(\h_0\totimes\bar{\partial}\w\right)
	+2\,g^{-1}\otimes \s{\h_0}{\w}\otimes \s{\H}{\bar{\partial} \w}\nonumber\\
	&-2\,g^{-2}\otimes \s{\bar{\h_0}}{\w}\otimes\left(\h_0\totimes \partial\w\right)-2\s{\H}{\w}\,g^{-1}\otimes\left(\h_0\totimes\bar{\partial}\w\right)\bigg)\nonumber\\
	&=d\,\Im\,\bigg(2\s{\Delta_g^{\perp}\w-2\,\Re\left(g^{-2}\otimes\s{\bar{\h_0}}{\w}\otimes \h_0\right)}{\partial\w}-\partial\left(|\D^{\perp}\w|^2_g\right)\nonumber\\
	&+4\s{\H}{\w}\left(\s{\H}{\partial \w}- g^{-1}\otimes\left(\h_0\totimes\bar{\partial}\w\right)\right)
	+4\,g^{-1}\otimes \s{\h_0}{\w}\otimes \s{\H}{\bar{\partial}\w}\bigg)\nonumber\\
	&=d\,\Im\,\bigg(2\s{\Delta_g^{\perp}\w+4\,\Re\left(g^{-2}\otimes\left(\bar{\partial}\phi\totimes\bar{\partial}\w\right)\otimes \h_0\right)}{\partial\w}-\partial\left(|\D^{\perp}\w|^2_g\right)\nonumber\\
	&-2\s{d\phi}{d\w}_g\left(\s{\H}{\partial \w}-\,g^{-1}\otimes\left(\h_0\totimes\bar{\partial}\w\right)\right)
	-8\,g^{-1}\otimes \left(\partial\phi\totimes\partial\w\right)\otimes \s{\H}{\bar{\partial}\w}\bigg)
	\end{align}
	by normality of $\w$.
	In particular, observe that in codimension $1$ we have $\w=w\,\n$ for some function $w:\Sigma\rightarrow \R$, and
	\begin{align*}
	\Delta_g^{\perp}\w=\left(\Delta_gw\right)\,\n,\qquad |\D^{\perp}\w|^2_g=|dw|_g^2 ,\qquad|\h_0|^2_{WP}=|\H|^2-K_g,
	\end{align*}
	and as for a real function $f$, the following identity holds $2\,\Im(\partial f)=\ast df$, we recover (see \cite{indexS3}) for a minimal surface ($\H=0$)
	\begin{align*}
	&\frac{d^2}{dt^2}\left(K_{g_t}d\mathrm{vol}_{g_t}\right)_{|t=0}=\left((\Delta_gw+2K_gw)\ast dw-\frac{1}{2}\ast d|dw|_g^2\right)
	\end{align*}
	For a minimal surface in arbitrary codimension ($\H=0$), we have likewise
	\begin{align}\label{p2}
	\frac{d^2}{dt^2}\left(K_{g_t}d\mathrm{vol}_{g_t}\right)_{|t=0}=\left(\s{\Delta_g^{\perp}\w+\mathscr{A}(\w)}{\ast\, d\w}-\frac{1}{2}\ast d |\D^{\perp}\w|_g^2\right),
	\end{align}
	where $\mathscr{A}$ is the Simons operator, defined by 
	\begin{align*}
	\mathscr{A}(\w)=e^{-2\lambda}\sum_{i,j=1}^{2}\s{\vec{\I}(\e_i,\e_j)}{\w}\vec{\I}(\e_i,\e_j)=2\s{\H}{\w}\w+2\,\Re\left(g^{-2}\otimes \s{\bar{\h_0}}{\w}\h_0\right).
	\end{align*}
	The last remark is immediate, thanks of the estimates
	\begin{align*}
	|d\w|_g\in L^{\infty}(\Sigma),\quad \Delta_g^{\perp}\w\in L^2(\Sigma,d\vg),\quad  \H,|\h_0|_{WP}\in L^2(\Sigma,d\vg).
	\end{align*}
	This concludes the proof of the lemma.
\end{proof}

Thanks of this result, we can now compute the second derivative of the Onofri energy and show that it satisfies the \textbf{Entropy condition} of \cite{index2}, making it a natural smoother in the viscosity method (see \cite{viscosity} and \cite{eversion}).

\begin{theorem}\label{d2o}
	Let $\phi:\Sigma\rightarrow\R^n$ be a weak immersion, $\epsilon>0$ be a fixed positive real number, $I=(-\epsilon,\epsilon)$ be a fixed interval, $\{\phi_t\}_{t\in I}$ be a $C^2$ family of weak immersions, $\{g_{0,t}\}_{t\in I}$ be a family of constant Gauss curvature metrics of volume $1$ and $\alpha_t,\alpha,\alpha'_0:\Sigma\rightarrow \R$ and $\w_t,\w:\Sigma\rightarrow T\R^n$ be functions such that for all $t\in I$
	\begin{align*}
	&g_t=\phi^{\ast}_t g_{\,\R^n},\quad g_t=e^{2\alpha_t}g_{0,t},\quad \alpha=\alpha_0,\quad \alpha_0'=\left(\frac{d}{dt}\alpha_t\right)_{|t=0}\\
	&\w_t=\frac{d}{dt}\phi_t,\quad \w=\w_0=\left(\frac{d}{dt}\phi_t\right)_{|t=0}.
	\end{align*}
	Then we have for all $t\in I$,
	\begin{align}
	&D\mathscr{O}(\phi)(\w)=-\frac{1}{2}\int_{\Sigma}|d\alpha|_g^2\s{d\phi}{d\w}_gd\mathrm{vol}_g+\frac{1}{2}\int_{\Sigma}\s{d\phi\totimes d\w+d\w\totimes d\phi}{d\alpha\otimes d\alpha}_gd\vg\nonumber\\
	&+\frac{K_{g_0}}{2}\int_{\Sigma}\s{d\phi}{d\w}_gd\mathrm{vol}_{g_0}
	-2\int_{\Sigma}d\alpha \wedge \Im\,\left(\s{\H}{\partial\w}-g^{-1}\otimes\left(\h_0\totimes \bar{\partial}\w\right)\right)-\frac{K_{g_0}}{2}\frac{\int_{\Sigma}\s{d\phi}{d\w}_gd\vg}{\int_{\Sigma}d\vg}
	\end{align}
	and if $\{\phi_t\}_{t\in I}$ satisfies the osculating property $\D_{\frac{d}{dt}}\frac{d}{dt}\phi_t=0$, we have
	\begin{align}
	&D^2\mathscr{O}(\phi)(\w,\w)=\int_{\Sigma}\s{d\alpha\otimes d\alpha}{d\w\totimes d\w}_gd\vg-\int_{\Sigma}\s{d\alpha\otimes d\alpha}{d\phi\totimes d\w+d\w\totimes d\phi}_g\s{d\phi}{d\w}_gd\vg\nonumber\\
	&-\frac{1}{2}\int_{\Sigma}|d\alpha|_g^2\left(|d\w|_g^2-2\s{d\phi}{d\w}_g^2-16|\partial\phi\totimes\partial\w|_{WP}^2\right)d\vg
	+\frac{K_{g_0}}{2}\int_{\Sigma}\left(|d\w|_g^2-16|\partial\phi\totimes\partial\w|_{WP}^2\right)d\mathrm{vol}_{g_0}\nonumber\\
	&+\frac{1}{2}\int_{\Sigma}\s{d\phi\totimes d\w+d\w\totimes d\phi}{d\alpha\otimes d\alpha_0'+d\alpha_0'\otimes d\alpha}_gd\vg-\int_{\Sigma}\s{d\alpha}{d\alpha_0'}_g\s{d\phi}{d\w}_gd\vg\nonumber\\
	&-2\int_{\Sigma}d\alpha_0'\wedge \Im\left(\s{\H}{\partial\w}-g^{-1}\otimes\left(\h_0\totimes\bar{\partial}\w\right)\right)-K_{g_0}\int_{\Sigma}\alpha_0'\s{d\phi}{d\w}d\mathrm{vol}_{g_0}\nonumber\\
	&-2\int_{\Sigma}d\alpha\wedge\,\Im\bigg(\s{\Delta_{g}^{\perp}\w+4\,\Re\left(g^{-2}\otimes\left(\bar{\partial\phi}\totimes\bar{\partial}\w\right)\otimes \h_0\right)}{\partial\w}
	-\s{d\phi}{d\w}_g\left(\s{\H}{\partial\w}-g^{-1}\otimes\left(\h_0\totimes\bar{\partial}\w\right)\right)\nonumber\\
	&-4\,g^{-1}\otimes\left(\partial\phi\totimes\partial\w\right)\otimes\s{\H}{\bar{\partial}\w}\bigg)-\frac{1}{2}\int_{\Sigma}|\D^{\perp}\w|_g^2\,\Delta_{g}\alpha\,d\mathrm{vol}_{g}
	+\frac{K_{g_0}}{2}\left(\frac{1}{\int_{\Sigma}d\vg}\int_{\Sigma}\s{d\phi}{d\w}_gd\vg\right)^2\nonumber\\
	&-\frac{K_{g_0}}{2}\frac{1}{\int_{\Sigma}d\vg}\int_{\Sigma}\left(|d\w|_g^2-16|\partial\phi\totimes\partial\w|_{WP}^2\right)d\vg.
	\end{align}
	Furthermore, $D \mathscr{O}(\phi)$ 
	is a continuous linear map on $\mathscr{E}_{\phi}(\Sigma,\R^n)$ and we have the estimates
	\begin{align*}
		&\norm{D\mathscr{O}(\phi)}_{\mathscr{E}'_{\phi}(\Sigma,\R^n)}\leq 4\pi|\chi(\Sigma)|+3\int_{\Sigma}|d\alpha|_g^2d\vg+\left(\sqrt{W(\phi)}+\sqrt{W(\phi)-2\pi\chi(\Sigma)}\right)\left(\int_{\Sigma}|d\alpha|_g^2d\vg\right)^{\frac{1}{2}}.\\
		&|D^2\mathscr{O}(\phi)(\w,\w)|\leq \left(W(\phi)+12\pi|\chi(\Sigma)|+16\int_{\Sigma}|d\alpha|_g^2d\vg\right)\norm{\w}_{\mathscr{E}_{\phi}(\Sigma,\R^n)}^2\nonumber\\
		&+\left(2+5\sqrt{W(\phi)}+5\sqrt{W(\phi)-2\pi\chi(\Sigma)}\right)\left(\int_{\Sigma}|d\alpha|_g^2d\vg\right)^{\frac{1}{2}}\norm{\w}_{\mathscr{E}_{\phi}(\Sigma,\R^n)}^2\nonumber\\
		&+\left(4\pi C_{\mathrm{PW}}|\chi(\Sigma)|+\sqrt{W(\phi)}+\sqrt{W(\phi)-2\pi\chi(\Sigma)}+6\left(\int_{\Sigma}|d\alpha|_g^2d\vg\right)^{\frac{1}{2}}\right)\left(\int_{\Sigma}|d\alpha_0'|^2_g\vg\right)^{\frac{1}{2}}\norm{\w}_{\mathscr{E}_{\phi}(\Sigma,\R^n)},
	\end{align*}
	where $C_{\mathrm{PW}}=C_{\mathrm{PW}}(g_0)$ is the Poincar\'{e}-Wirtinger constant for the injection $L^{2}(\Sigma,g_0)/\R\hookrightarrow \dot{W}^{1,2}(\Sigma,g_0)$.
\end{theorem}

\begin{proof}
	\textbf{Step 1: First derivative of the Onofri energy.}
	
	Let $g_t=\phi_t^\ast g_{\,\R^n}$, and ${g}_{0,t}$ a constant Gauss curvature metric of volume $1$ and $\alpha_t:\Sigma\rightarrow \R$ as in the introduction, \textit{i.e.}
	\begin{align*}
	g_t=e^{2\alpha_t}\,{g}_{0,t}.
	\end{align*}
	As ${g}_{0,t}$ is of constant volume $1$ for all $t$, we have
	\begin{align}\label{i}
	0=\frac{d}{dt}\int_{\Sigma}d\vol{{g}_{0,t}}=-2\int_{\Sigma}\left(\frac{d\alpha_t}{dt}\right)e^{-2\alpha_t}d\vol_{g_t}+\int_{\Sigma^2}e^{-2\alpha_t}\s{d\phi_t}{d\w_t}_{g_t}d\vol_{g_t}.
	\end{align}
	where we noted
	\begin{align*}
	\w_t=\frac{d}{dt}\phi_t.
	\end{align*}
	By the Liouville equation, we have
	\begin{align}\label{liouville}	-\Delta_{g_t}\alpha_t=K_{g_t}-K_{{g}_0}e^{-2\alpha_t}
	\end{align}
	by \eqref{Kbar}. Now, by deriving this equation, we have
	\begin{align*}
		-\left(\frac{d}{dt}\Delta_{g_t}\right)\alpha_t-\Delta_{g_t}\left(\frac{d}{dt}\alpha_t\right)=\frac{d}{dt}\left(K_{g_t}\right)+2K_{g_0}\,\left(\frac{d}{dt}\alpha_t\right)e^{-2\alpha_t}
	\end{align*}
	so that
	\begin{align}\label{dliouville}
		-\Delta_{g_t}\left(\frac{d}{dt}\alpha_t\right)=2K_{g_0}\left(\frac{d}{dt}\alpha_t\right)e^{-2\alpha_t}+\frac{d}{dt}\left(K_{g_t}\right)+\left(\frac{d}{dt}\Delta_{g_t}\right)\alpha_t.
	\end{align}
	Now, we recall the formulas
	\begin{align*}
		&\frac{d}{dt}d\mathrm{vol}_{g_t}=\s{d\phi_t}{d\w_t}_{g_t}d\mathrm{vol}_{g_t}\\
		&\frac{d}{dt}g^{i,j}_t=-\frac{1}{\mathrm{det}(g_t)}\left(\s{\D_{\e_i(t)}\w_t}{\e_j(t)}+\s{\D_{\e_j(t)}\w_t}{\e_i(t)}\right)
	\end{align*}
	so that
	\begin{align}\label{a1}
		&\frac{d}{dt}\int_{\Sigma}|d\alpha_t|_{g_t}^2d\mathrm{vol}_{g_t}=\frac{d}{dt}\int_{\Sigma}\sum_{i,j=1}^{2}g^{i,j}_t\,\p{x_i}\alpha_t\cdot\p{x_j}\alpha_t\,d\mathrm{vol}_{g_t}
		=\int_{\Sigma}\left(\frac{d}{dt}g^{i,j}_t\right)\,\p{x_i}\alpha_t\cdot\p{x_j}\alpha_t\,d\mathrm{vol}_{g_t}\nonumber\\
		&+\int_{\Sigma}\sum_{i,j=1}^{2}g^{i,j}_t\left(\p{x_i}\left(\frac{d}{dt}\alpha_t\right)\cdot\p{x_j}\alpha_t+\p{x_i}\alpha_t\cdot\p{x_j}\left(\frac{d}{dt}\alpha_t\right)\right)d\mathrm{vol}_{g_t}
		+\int_{\Sigma}|d\alpha_t|_{g_t}^2\s{d\phi_t}{d\w_t}_{g_t}d\mathrm{vol}_{g_t}\nonumber\\
		&=\int_{\Sigma}\left(\frac{d}{dt}g^{i,j}\right)\p{x_i}\alpha_t\cdot\p{x_j}\alpha_t\,d\mathrm{vol}_{g_t}+\int_{\Sigma}|d\alpha_t|_{g_t}^2\s{d\phi_t}{d\w_t}_{g_t}d\mathrm{vol}_{g_t}+2\int_{\Sigma}\s{d\left(\frac{d}{dt}\alpha_t\right)}{d\alpha_t}_{g_t}d\mathrm{vol}_{g_t}\nonumber\\
		&=\mathrm{(1)}+\mathrm{(2)}+\mathrm{(3)}.
	\end{align}
	Now, by integrating by parts and using the derivative of the Liouville equation \eqref{dliouville}, we obtain
	\begin{align}\label{a2}
     	\mathrm{(3)}&=2\int_{\Sigma}\s{d\left(\frac{d}{dt}\alpha_t\right)}{d\alpha_t}_{g_t}d\mathrm{vol}_{g_t}=-2\int_{\Sigma}\Delta_{g_t}\left(\frac{d}{dt}\alpha_t\right)\,\alpha_td\mathrm{vol}_{g_t}\nonumber\\
		&=4K_{g_0}\int_{\Sigma}\left(\frac{d}{dt}\alpha_t\right)\alpha_t e^{-2\alpha_t}d\mathrm{vol}_{g_t}+2\int_{\Sigma}\frac{d}{dt}(K_{g_t})\alpha_td\mathrm{vol}_{g_t}+2\int_{\Sigma}\alpha_t\left(\frac{d}{dt}\Delta_{g_t}\right)\alpha_t\,d\mathrm{vol}_{g_t}\nonumber\\
		&=\mathrm{(I)}+\mathrm{(II)}+\mathrm{(III)}.
	\end{align}
	We have
	\begin{align}\label{a3}
		\mathrm{(II)}=2\int_{\Sigma}\frac{d}{dt}(K_{g_t})\alpha_td\mathrm{vol}_{g_t}=2\int_{\Sigma}\alpha_t\frac{d}{dt}(K_{g_t}d\mathrm{vol}_{g_t})-2\int_{\Sigma}\alpha_tK_{g_t}\s{d\phi_t}{d\w_t}_{g_t}d\mathrm{vol}_{g_t}
	\end{align}
	Recalling that their exists a $1$-form $\omega_t$ (see \cite{indexS3}) such that
	\begin{align}\label{a4}
		&\frac{d}{dt}\left(K_{g_t}d\mathrm{vol}_{g_t}\right)=d\,\omega_t\nonumber\\
		&\omega_0=\Im\left(2\s{\H}{\partial\phi}-2\,g^{-1}\otimes\left(\h_0\totimes\bar{\partial}\phi\right)\right)
	\end{align}
	we get by \eqref{a3} and \eqref{a4}
	\begin{align}\label{a5}
		\mathrm{(II)}=-2\int_{\Sigma}d\alpha_t\wedge \omega_t-2\int_{\Sigma}\alpha_tK_{g_t}\s{d\phi_t}{d\w_t}_{g_t}d\mathrm{vol}_{g_t}.
	\end{align}
	Now, we compute
	\begin{align}\label{newref1}
		&\frac{d}{dt}\left(\Delta_{g_t}\right)
		=\frac{d}{dt}\left(\frac{1}{\sqrt{\det(g_t)}}\sum_{i,j=1}^{2}\partial_{x_i}\left(\sqrt{\det(g_t)}g^{i,j}_t\partial_{x_j}\left(\,\cdot\,\right)\right)\right)\nonumber\\
		&=\frac{d}{dt}\left(\sum_{i,j=1}^{2}\frac{\partial_{x_i} \sqrt{\det(g_t)}}{\sqrt{\det(g_t)}}g^{i,j}_t\d{j}\left(\,\cdot\,\right)+\partial_{x_i}\left(g_t^{i,j}\partial_{x_j}\left(\,\cdot\,\right)\right)\right)\nonumber\\
		&=\frac{1}{\sqrt{\det(g_t)}}\sum_{i,j=1}^{2}g^{i,j}_t\partial_{x_i}\left(\s{d\phi_t}{d\w_t}_{g_t}\right)\partial_{x_j}\left(\,\cdot\,\right)-\frac{\s{d\phi_t}{d\w_t}_{g_t}}{\sqrt{\det(g_t)}}\sum_{i,j=1}^{2}g^{i,j}_t\partial_{x_i}\left(\sqrt{\det(g_t)}\right)\partial_{x_j}\left(\,\cdot\,\right)\nonumber\\
		&-\sum_{i,j=1}^{2}\frac{\d{i}\sqrt{\det(g_t)}}{\det(g_t)^{\frac{3}{2}}}\left(\s{\d{i}\phi_t}{\D_{\d{j}}\w}+\s{\d{j}\phi_t}{\D_{\d{i}}\w_t}\right)\d{j}\left(\,\cdot\,\right)\nonumber\\
		&-\sum_{i,j=1}^{2}\partial_{x_i}\left(\frac{1}{\det(g_t)}\left(\s{\d{i}\phi_t}{\D_{\d{j}}\w_t}+\s{\d{j}\phi_t}{\D_{\d{i}}\w_t}\right)\d{j}\left(\,\cdot\,\right)\right)\nonumber\\
		&=\frac{1}{\sqrt{\det(g_t)}}\sum_{i,j=1}^{2}g^{i,j}_t\left(\d{i}\left(\s{d\phi_t}{d\w_t}_{g_t}\right)\right)\d{j}\left(\,\cdot\,\right)\nonumber\\
		&-\frac{1}{\sqrt{\det(g_t)}}\sum_{i,j=1}^{2}\d{i}\left(\frac{1}{\sqrt{\det(g_t)}}\left(\s{\d{i}\phi_t}{\D_{\d{j}}\w_t}+\s{\d{j}\phi_t}{\D_{\d{i}}\w_t}\right)\d{j}\left(\,\cdot\,\right)\right).
	\end{align}
	Therefore, we have by \eqref{newref1}
	\begin{align*}
		\frac{d}{dt}(\Delta_{g_t})
		&=\s{d\s{d\phi_t}{d\w_t}_{g_t}}{d(\,\cdot\,)}_{g_t}+\frac{1}{\sqrt{\det(g_t)}}\sum_{i,j=1}^{2}\d{i}\left(\sqrt{\det(g_t)}\left(\frac{d}{dt}g^{i,j}_t\right)\d{j}\left(\,\cdot\,\right)\right).
	\end{align*}
	As $d\mathrm{vol}_{g_t}=\sqrt{\mathrm{det}(g_t)}dx_1\wedge dx_2$ in our local chart, we obtain thanks of an integration by parts
	\begin{align}\label{a6}
		\mathrm{(III)}&=2\int_{\Sigma}\alpha_t\left(\frac{d}{dt}\Delta_{g_t}\right)\alpha_t\,d\mathrm{vol}_{g_t}\nonumber\\
		&=2\int_{\Sigma}\alpha_t\s{d\s{d\phi_t}{d\w_t}_{g_t}}{d\alpha_t}_{g_t}d\mathrm{vol}_{g_t}+2\int_{\Sigma}\alpha_t\sum_{i,j=1}^2\p{x_i}\left(\sqrt{\det(g_t)}\left(\frac{d}{dt}g_t^{i,j}\right)\p{x_j}\alpha_t\right)dx_1\wedge dx_2\nonumber\\
		&=-2\int_{\Sigma}|d\alpha_t|_{g_t}^2\s{d\phi_t}{d\w_t}_{g_t}d\mathrm{vol}_{g_t}-2\int_{\Sigma}\alpha_t\s{d\phi_t}{d\w_t}\Delta_{g_t}\alpha_td\mathrm{vol}_{g_t}\nonumber\\
		&-2\int_{\Sigma}\sum_{i,j=1}^{2}\left(\frac{d}{dt}g^{i,j}\right)\p{x_i}\alpha_t\cdot\p{x_j}\alpha_t\,\sqrt{\det(g_t)}dx_1\wedge dx_2\nonumber\\
		&=-2\left\{\mathrm{(1)}+\mathrm{(2)}\right\}-2\int_{\Sigma}\alpha_t\s{d\phi_t}{d\w_t}_{g_t}\Delta_{g_t}\alpha_t\,d\mathrm{vol}_{g_t}
	\end{align}
	by the definition given in \eqref{a1}. Now, by the Liouville equation \eqref{liouville}, we have
	\begin{align}\label{a7}
		-2\int_{\Sigma}\alpha_t\s{d\phi_t}{d\w_t}_{g_t}\Delta_{g_t}\alpha_t\,d\mathrm{vol}_{g_t}=2\int_{\Sigma}\alpha_t\,K_{g_t}\s{d\phi_t}{d\w_t}_{g_t}d\mathrm{vol}_{g_t}-2K_{g_0}\int_{\Sigma}\alpha_te^{-2\alpha_t}\s{d\phi_t}{d\w_t}_{g_t}d\mathrm{vol}_{g_t}
	\end{align}
	Therefore, by \eqref{a5}, \eqref{a6} and \eqref{a7}, we obtain
	\begin{align}\label{a8}
		\mathrm{(II)}+\mathrm{(III)}=-2\left\{\mathrm{(1)}+\mathrm{(2)}\right\}-2\int_{\Sigma}d\alpha_t\wedge \omega_t-2K_{g_0}\int_{\Sigma}\alpha_te^{-2\alpha_t}\s{d\phi_t}{d\w_t}_{g_t}d\mathrm{vol}_{g_t}.
	\end{align}
	Now, as 
	\begin{align*}
		\int_{\Sigma}e^{-2\alpha_t}d\mathrm{vol}_{g_t}=\int_{\Sigma}d\mathrm{vol}_{g_{0,t}}=1,
	\end{align*}
	we obtain for all $\lambda\in \R$
	\begin{align*}
		&\frac{d}{dt}\int_{\Sigma}\alpha_te^{-2\alpha_t}d\mathrm{vol}_{g_t}=\frac{d}{dt}\int_{\Sigma}\left(\alpha_t+\lambda\right)e^{-2\alpha_t}d\mathrm{vol}_{g_t}=\int_{\Sigma}(\alpha_t+\lambda)e^{-2\alpha_t}\s{d\phi_t}{d\w_t}_{g_t}d\mathrm{vol}_{g_t}\\
		&+\int_{\Sigma}\left(\frac{d}{dt}\alpha_t\right)e^{-2\alpha_t}d\mathrm{vol}_{g_t}-\int_{\Sigma}2\left(\frac{d}{dt}\alpha_t\right)\left(\alpha_t+\lambda\right)e^{-2\alpha_t}d\mathrm{vol}_{g_t}\\
		&=\int_{\Sigma}(\alpha_t+\lambda)e^{-2\alpha_t}\s{d\phi_t}{d\w_t}d\mathrm{vol}_{g_t}-2\int_{\Sigma}\left(\frac{d}{dt}\alpha_t\right)\alpha_te^{-2\alpha_t}d\mathrm{vol}_{g_t}+\int_{\Sigma}(1-2\lambda)\left(\frac{d}{dt}\alpha_t\right)e^{-2\alpha_t}d\mathrm{vol}_{g_t}.
	\end{align*}
	Taking $\lambda=\dfrac{1}{2}$ yields
	\begin{align}\label{a9}
		\frac{d}{dt}\left(2K_{g_0}\int_{\Sigma}\alpha_te^{-2\alpha_t}d\mathrm{vol}_{g_t}\right)&=-4K_{g_0}\int_{\Sigma}\left(\frac{d}{dt}\alpha_t\right)\alpha_te^{-2\alpha_t}\,d\mathrm{vol}_{g_t}+K_{g_0}\int_{\Sigma}(1+2\alpha_t)e^{-2\alpha_t}\s{d\phi_t}{d\w_t}d\mathrm{vol}_{g_t}\nonumber\\
		&=-\mathrm{(I)}+K_{g_0}\int_{\Sigma}(1+2\alpha_t)e^{-2\alpha_t}\s{d\phi_t}{d\w_t}d\mathrm{vol}_{g_t}.
	\end{align}
	Therefore, by \eqref{a8} and \eqref{a9}, we obtain as $(3)=\mathrm{(I)}+\mathrm{(II)}+\mathrm{(III)}$ the identity
	\begin{align}\label{a10}
		\frac{d}{dt}\left(2K_{g_0}\int_{\Sigma}\alpha_te^{-2\alpha_t}d\mathrm{vol}_{g_t}\right)+\mathrm{(3)}=-2\left\{\mathrm{(1)}+\mathrm{(2)}\right\}+K_{g_0}\int_{\Sigma}e^{-2\alpha_t}\s{d\phi_t}{d\w_t}_{g_t}d\mathrm{vol}_{g_t}-2\int_{\Sigma}d\alpha_t\wedge \omega_t
	\end{align}
	Finally, by \eqref{a1}, and \eqref{a10}, we obtain
	\begin{align}
		&\frac{d}{dt}\left(\int_{\Sigma}|d\alpha_t|_{g_t}^2d\mathrm{vol}_{g_t}+2K_{g_0}\int_{\Sigma}\alpha_te^{-2\alpha_t}d\mathrm{vol}_{g_t}\right)\nonumber\\
		&=-\left\{(1)+(2)\right\}+K_{g_0}\int_{\Sigma}e^{-2\alpha_t}\s{d\phi_t}{d\w_t}_{g_t}d\mathrm{vol}_{g_t}-2\int_{\Sigma}d\alpha_t\wedge \omega_t\nonumber\\
		&=-\int_{\Sigma}\sum_{i,j=1}^2\left(\frac{d}{dt}g^{i,j}_t\right)\p{x_i}\alpha_t\cdot\p{x_j}\alpha_t\,d\mathrm{vol}_{g_t}-\int_{\Sigma}|d\alpha_t|_{g_t}^2\s{d\phi_t}{d\w_t}_{g_t}d\mathrm{vol}_{g_t}+K_{g_0}\int_{\Sigma}e^{-2\alpha_t}\s{d\phi_t}{d\w_t}d\mathrm{vol}_{g_t}\nonumber\\
		&-2\int_{\Sigma}d\alpha_t\wedge \omega_t
	\end{align}
	By \eqref{a4}, we obtain
	\begin{align}\label{aend}
		&\frac{d}{dt}\left(\int_{\Sigma}|d\alpha_t|_{g_t}^2d\mathrm{vol}_{g_t}+2K_{g_0}\int_{\Sigma}\alpha_te^{-2\alpha_t}d\mathrm{vol}_{g_t}\right)_{|t=0}\nonumber\\
		&=\int_{\Sigma}\sum_{i,j=1}^{2}\frac{1}{\det(g)}\left(\s{\D_{\e_i}\w}{\e_j}+\s{\D_{\e_j}\w}{\e_i}\right)\p{x_i}\alpha_t\cdot \p{x_j}\alpha_td\vg-\int_{\Sigma}|d\alpha|_g^2\s{d\phi}{d\w}_gd\vg\nonumber\\
		&+K_{g_0}\int_{\Sigma}\s{d\phi}{d\w}_gd\mathrm{vol}_{g_0}
		-4\int_{\Sigma}d\alpha \wedge \Im\,\left(\s{\H}{\partial\w}-g^{-1}\otimes\left(\h_0\totimes \bar{\partial}\w\right)\right)\nonumber\\
		&=-\int_{\Sigma}|d\alpha|_g^2\s{d\phi}{d\w}_gd\mathrm{vol}_g+\int_{\Sigma}\s{d\phi\totimes d\w+d\w\totimes d\phi}{d\alpha\otimes d\alpha}_gd\vg+K_{g_0}\int_{\Sigma}\s{d\phi}{d\w}_gd\mathrm{vol}_{g_0}\nonumber
		\\
		&-4\int_{\Sigma}d\alpha \wedge \Im\,\left(\s{\H}{\partial\w}-g^{-1}\otimes\left(\h_0\totimes \bar{\partial}\w\right)\right).
	\end{align}
	Now, as 
	\begin{align*}
		\mathscr{O}(\phi)=\frac{1}{2}\int_{\Sigma}|d\alpha|_g^2d\vg+K_{g_0}\int_{\Sigma}\alpha e^{-2\alpha}\,d\vg-\frac{K_{g_0}}{2}\log\int_{\Sigma}d\vg,
	\end{align*}
	we obtain
	\begin{align*}
		&\frac{d}{dt}\mathscr{O}(\phi)_{|t=0}=-\frac{1}{2}\int_{\Sigma}|d\alpha|_g^2\s{d\phi}{d\w}_gd\mathrm{vol}_g+\frac{1}{2}\int_{\Sigma}\s{d\phi\totimes d\w+d\w\totimes d\phi}{d\alpha\otimes d\alpha}_gd\vg\\
		&+\frac{K_{g_0}}{2}\int_{\Sigma}\s{d\phi}{d\w}_gd\mathrm{vol}_{g_0}
		-2\int_{\Sigma}d\alpha \wedge \Im\,\left(\s{\H}{\partial\w}-g^{-1}\otimes\left(\h_0\totimes \bar{\partial}\w\right)\right)-\frac{K_{g_0}}{2}\frac{\int_{\Sigma}\s{d\phi}{d\w}_gd\vg}{\int_{\Sigma}d\vg}.
	\end{align*}
	Therefore, we obtain the estimate by Cauchy-Schwarz and Minkowski's inequalities (recall that $g_0$ has unit volume) and Gauss-Bonnet identity
	\begin{align*}
		&|D\mathscr{O}(\phi)(\w)|\leq 3\int_{\Sigma}|d\alpha|_g^2d\vg+|K_{g_0}|\np{|d\w|_g}{\infty}{\Sigma}\\
		&+\left(\int_{\Sigma}|d\alpha|_g^2d\vg\right)^{\frac{1}{2}}\left(\left(\int_{\Sigma}|\H|^2d\vg\right)^{\frac{1}{2}}+\left(\int_{\Sigma}|\h_0|_{WP}^2d\vg\right)^{\frac{1}{2}}\right)\np{|d\w|_g}{\infty}{\Sigma}\\
		&=\left(3\int_{\Sigma}|d\alpha|_g^2+2|K_{g_0}|+\left(\sqrt{W(\phi)}+\sqrt{W(\phi)-2\pi\chi(\Sigma)}\right)\left(\int_{\Sigma}|d\alpha|_g^2d\vg\right)^{\frac{1}{2}}\right)\np{|d\w|_g}{\infty}{\Sigma}\\
		&=\left(3\int_{\Sigma}|d\alpha|_g^2+4\pi|\chi(\Sigma)|+\left(\sqrt{W(\phi)}+\sqrt{W(\phi)-2\pi\chi(\Sigma)}\right)\left(\int_{\Sigma}|d\alpha|_g^2d\vg\right)^{\frac{1}{2}}\right)\np{|d\w|_g}{\infty}{\Sigma}\\
		&\leq \left(3\int_{\Sigma}|d\alpha|_g^2+4\pi|\chi(\Sigma)|+\left(\sqrt{W(\phi)}+\sqrt{W(\phi)-2\pi\chi(\Sigma)}\right)\left(\int_{\Sigma}|d\alpha|_g^2d\vg\right)^{\frac{1}{2}}\right)\norm{\w}_{\mathscr{E}_{\phi}(\Sigma,\R^n)}.
	\end{align*}
	Therefore, the function
	\begin{align*}
		W_{\sigma}(\phi)=W(\phi)+\sigma^2\int_{\Sigma}\left(1+|\H_g|^2\right)^2d\mathrm{vol}_g+\frac{1}{\log\left(\frac{1}{\sigma}\right)}\mathscr{O}(\phi)
	\end{align*}
	satisfies the \textbf{Energy bound} condition of the main theorem in \cite{index2}. Notice that by the refined Onofri inequality if $\alpha$ is chosen according to the Aubin gauge (see \cite{eversion})
	\begin{align*}
		\frac{1}{6}\int_{S^2}|d\alpha|_g^2d\vg\leq \mathscr{O}(\phi),
	\end{align*}
	while for $\Sigma$ of genus at least one we have by Jensen's inequality (as $K_{g_0}\leq 0$)
	\begin{align*}
		K_{g_0}\int_{\Sigma}\alpha\,e^{-2\alpha}d\vg-\frac{K_{g_0}}{2}\log\int_{\Sigma}d\vg\geq 0,
	\end{align*}
	so
	\begin{align*}
		\frac{1}{2}\int_{\Sigma}|d\alpha|_g^2d\vg\leq \mathscr{O}(\phi),
	\end{align*}
	and an upper-bound on $W_{\sigma}$ implies also a uniform bound on $D W_{\sigma}$ (actually, as we shall see, it also implies a uniform control on the second derivative $D^2W_{\sigma}$).
	
	\textbf{Step 2: Computation of the second derivative of Onofri's energy and energy bound.}
	We first recall the following formula
	\begin{align}\label{idem}
	\frac{d^2}{dt^2}\left(d\mathrm{vol}_{g_t}\right)_{|t=0}&=\left(|d\w|_g^2+\s{d\phi}{d\w}_g^2-16\left(|\partial\phi\totimes\partial\w|_{WP}^2+|\Re(\partial\phi\totimes \bar{\partial}\w)|_{WP}^2\right)\right)d\vg\nonumber\\
	&=\left(|d\w|_g^2-16|\partial\phi\totimes \partial\w|_{WP}^2\right)d\vg
	\end{align}
	where $|\,\cdot\,|_{WP}$ is the Weil-Petersson product with respect to the metric $g$. In a complex chart $z$, we identify $g=e^{2\lambda}|dz|^2$ with $e^{2\lambda}dz\otimes d\z$ and for any quadratic differential $\vec{\alpha}=\vec{F}(z)dz^2$ or symmetric tensor $\vec{\beta}=\vec{G}(z)dz\otimes d\z$ with values into $\R^n$, we define
	\begin{align*}
	|\vec{\alpha}|_{WP}^2=g^{-2}\otimes \left(\vec{\alpha}\totimes\bar{\vec{\alpha}}\right)=e^{-4\lambda}|\vec{F}|^2,\quad  |\vec{\beta}|_{WP}^2=g^{-2}\otimes\left(\vec{\beta}\totimes \bar{\vec{\beta}}\right)=e^{-4\lambda}|\vec{G}|^2.
	\end{align*}	
	The last formula may not be presented under this form in general, so we will add a derivation of it. By \cite{indexS3} for example, we have
	\begin{align*}
	\frac{d^2}{dt^2}\left(d\mathrm{vol}_{g_t}\right)_{|t=0}=|d\w|_g^2+\s{d\phi}{d\w}_g^2-e^{-4\lambda}\left(2\s{\D_{\e_1}\w}{\e_1}^2+2\s{\D_{\e_2}\w}{\e_2}^2+\left(\s{\D_{\e_1}\w}{\e_2}+\s{\D_{\e_2}\w}{\e_1}\right)^2\right).
	\end{align*}
	Now, we have by identifying  $\D_{\p{z}}$ and $\partial=\D_{\p{z}}(\,\cdot\,)\otimes dz$ the identities
	\begin{align*}
	&\e_1=\e_{\z}+\e_{\z}=(\partial+\bar{\partial})\phi\\
	&\e_2=i(\e_z-\e_{\z})=i(\partial-\bar{\partial})\phi.
	\end{align*}
	Then we find that
	\begin{align*}
	&e^{-4\lambda}\Big\{2\s{\D_{\e_1}\w}{\e_1}^2+2\s{\D_{\e_2}\w}{\e_2}^2+\left(\s{\D_{\e_1}\w}{\e_2}+\s{\D_{\e_2}\w}{\e_1}\right)^2=2\s{(\partial+\bar{\partial})\phi}{(\partial+\bar{\partial})\w}^2\nonumber\\
	&+2\s{(\partial-\bar{\partial})\phi}{(\partial-\bar{\partial})\w}^2
	+\left(\s{(\partial+\bar{\partial})\phi}{i(\partial-\bar{\partial})\phi}+\s{i(\partial-\bar{\partial})\w}{(\partial+\bar{\partial})\w}\right)^2\Big\}\\
	&=e^{-4\lambda}\Big\{4(\s{\partial\phi}{\bar{\partial}\w})^2+4(\s{\bar{\partial}\phi}{\partial\w})^2+8\s{\partial\phi}{\bar{\partial}\w}\s{\bar{\partial}\phi}{\partial\w}+16\s{\partial\phi}{\partial\w}\s{\bar{\partial}\phi}{\bar{\partial}\w}\Big\}\\
	&=4e^{-4\lambda}\left(\partial\phi\totimes\bar{\partial}\w+\bar{\partial}\phi\totimes\partial\w\right)^2+16e^{-4\lambda}|\partial\phi\totimes{\partial}\w|^2\\
	&=16|\Re(\partial\phi\totimes\bar{\partial}\w)|_{WP}^2+16|\partial\phi\totimes\partial\w|_{WP}^2.
	\end{align*}
	Therefore, we have
	\begin{align*}
	\frac{d^2}{dt^2}\left(d\mathrm{vol}_{g_t}\right)_{|t=0}=|d\w|_g^2+\s{d\phi}{d\w}_g^2-16\left(|\Re(\partial\phi\totimes\bar{\partial}\w)|_{WP}^2+|\partial\phi\totimes\partial\w|_{WP}^2\right)
	\end{align*}
	The formula given in \cite{hierarchies} is
	\begin{align*}
	\frac{d^2}{dt^2}\left(d\mathrm{vol}_{g_t}\right)_{|t=0}=|d\w|_g^2+\s{d\phi}{d\w}_g^2-\frac{1}{2}|d\phi\totimes d\w+d\w\totimes d\phi|_g^2,
	\end{align*}
	where in local coordinates, we have
	\begin{align*}
	d\phi\totimes d\w=\sum_{i,j=1}^{2}\s{\p{x_i}\w}{\D_{\p{x_i}}\w}dx_i\otimes dx_j,
	\end{align*}
	which implies that
	\begin{align*}
	d\phi\totimes d\w+d\w\totimes d\phi=\sum_{i,j=1}^2\left(\s{\p{x_i}\phi}{\D_{\p{x_j}}\w}+\s{\p{x_j}\w}{\D_{\p{x_i}}\w}\right)dx_i\otimes dx_j
	\end{align*}
	and
	\begin{align*}
	|d\phi\totimes d\w+d\phi\totimes d\w|^2_{g}&=e^{-4\lambda}\sum_{i,j=1}^{2}\left(\s{\p{x_i}\phi}{\D_{\p{x_j}}\w}+\s{\p{x_j}\w}{\D_{\p{x_i}}\w}\right)^2\\
	&=e^{-4\lambda}\left(4\s{\D_{\e_1}\w}{\e_1}^2+4\s{\D_{\e_2}\w}{\e_2}^2+2\left(\s{\D_{\e_1}\w}{\e_2}+\s{\D_{\e_2}\w}{\e_1}\right)^2\right)		
	\end{align*}
	which coincides with formula \ref{idem}.
	
	Recall that for all $t\in  I$
	\begin{align}\label{b0}
		&\frac{d}{dt}\left(\int_{\Sigma}|d\alpha_t|_{g_t}^2d\mathrm{vol}_{g_t}+2K_{g_0}\int_{\Sigma}\alpha_te^{-2\alpha_t}d\mathrm{vol}_{g_t}\right)\nonumber\\
		&=-\int_{\Sigma}\sum_{i,j=1}^2\left(\frac{d}{dt}g^{i,j}_t\right)\p{x_i}\alpha_t\cdot\p{x_j}\alpha_t\,d\mathrm{vol}_{g_t}-\int_{\Sigma}|d\alpha_t|_{g_t}^2\s{d\phi_t}{d\w_t}_{g_t}d\mathrm{vol}_{g_t}+K_{g_0}\int_{\Sigma}e^{-2\alpha_t}\s{d\phi_t}{d\w_t}d\mathrm{vol}_{g_t}\nonumber\\
		&-2\int_{\Sigma}d\alpha_t\wedge d\omega_t\nonumber\\
		&=-\int_{\Sigma}\sum_{i,j=1}^2\left(\frac{d}{dt}g^{i,j}_t\right)\p{x_i}\alpha_t\cdot\p{x_j}\alpha_t\,d\mathrm{vol}_{g_t}-\int_{\Sigma}|d\alpha_t|_{g_t}^2\frac{d}{dt}\left(d\mathrm{vol}_{g_t}\right)+K_{g_0}\int_{\Sigma}e^{-2\alpha_t}\s{d\phi_t}{d\w_t}d\mathrm{vol}_{g_t}\nonumber\\
		&-2\int_{\Sigma}d\alpha_t\wedge d\omega_t
	\end{align}
	Furthermore, we recall the formulas
	\begin{align}\label{form}
		&\frac{d^2}{dt^2}\left(d\mathrm{vol}_{g_t}\right)_{|t=0}=\left(|d\w|_g^2+\s{d\phi}{d\w}_g^2-16|\Re(\partial\phi\totimes\bar{\partial}\phi)|_{WP}^2-16|\partial\phi\totimes\partial\w|_{WP}^2\right)d\vg\nonumber\\
		&\left(\frac{d^2}{dt^2}g^{i,j}_{t}\right)_{|t=0}=2\bigg\{-e^{-4\lambda}\s{\D_{\e_i}\w}{\D_{\e_j}\w}+2e^{-4\lambda}\s{d\phi}{d\w}_g\left(\s{\D_{\e_i}\w}{\e_j}+\s{\D_{\e_j}\w}{\e_i}\right)\nonumber\\
		&-\delta_{i,j}e^{-2\lambda}\left(2\s{d\phi}{d\w}_g^2-16|\Re(\partial\phi\totimes\bar{\partial}\w)|_{WP}^2-16|\partial\phi\totimes\partial\w|_{WP}^2\right)\bigg\}.
	\end{align}
	Therefore, we have (in the following expression, we will not always write explicitly that the expression we consider are taken at $t=0$)
	\begin{align}\label{b00}
		&\frac{d^2}{dt^2}\left(\int_{\Sigma}|d\alpha_t|_{g_t}^2d\mathrm{vol}_{g_t}+2K_{g_0}\int_{\Sigma}\alpha_te^{-2\alpha}d\mathrm{vol}_{g_t}\right)_{|t=0}\nonumber\\
		&=-\int_{\Sigma}\sum_{i,j=1}^{2}\left(\frac{d^2}{dt^2}g^{i,j}_t\right)\p{x_i}\alpha\cdot\p{x_j}\alpha\,d\mathrm{vol}_g\nonumber\\
		&-\int_{\Sigma}\sum_{i,j=1}^{2}\left(\frac{d}{dt}g^{i,j}_t\right)\left(\p{x_i}\left(\frac{d}{dt}\alpha_t \right)\cdot\p{x_j}\alpha+\p{x_i}\alpha\cdot\p{x_j}\left(\frac{d}{dt}\alpha_t\right)\right)d\vg\nonumber\\
		&-\int_{\Sigma}\sum_{i,j=1}^{2}\left(\frac{d}{dt}g_t^{i,j}\right)\p{x_i}\alpha\cdot\p{x_j}\alpha\,\s{d\phi}{d\w}_gd\vg-\int_{\Sigma}\sum_{i,j=1}^{2}\left(\frac{d}{dt}g^{i,j}_t\right)\p{x_i}\alpha\cdot\p{x_j}\alpha\s{d\phi}{d\w}_gd\vg\nonumber\\
		&-\int_{\Sigma}\sum_{i,j=1}^{2}g^{i,j}\left(\p{x_i}\left(\frac{d}{dt}\alpha_t\right)\cdot\p{x_j}\alpha+\p{x_i}\alpha\cdot\p{x_j}\left(\frac{d}{dt}\alpha_t\right)\right)\s{d\phi}{d\w}_gd\vg\nonumber\\
		&-\int_{\Sigma}|d\alpha|_g^2\left(|d\w|_g^2+\s{d\phi}{d\w}_g^2-16|\Re(\partial\phi\totimes\bar{\partial}\phi)|_{WP}^2-16|\partial\phi\totimes\partial\w|_{WP}^2\right)d\vg\nonumber\\
		&-2K_{g_0}\int_{\Sigma}\left(\frac{d}{dt}\alpha_t\right)\s{d\phi}{d\w}_gd\mathrm{vol}_{g_0}\nonumber\\
		&+K_{g_0}\int_{\Sigma}\left(|d\w|_g^2+\s{d\phi}{d\w}_g^2-16|\Re(\partial\phi\totimes\bar{\partial}\phi)|_{WP}^2-16|\partial\phi\totimes\partial\w|_{WP}^2\right)d\mathrm{vol}_{g_0}\nonumber\\
		&-2\int_{\Sigma}d\left(\frac{d}{dt}\alpha_t\right)\wedge\omega_t-2\int_{\Sigma}d\alpha\wedge \left(\frac{d}{dt}\omega_t\right)\nonumber\\
		&=-\int_{\Sigma}\sum_{i,j=1}^{2}\left(\frac{d^2}{dt^2}g^{i,j}_t\right)\p{x_i}\alpha\cdot\p{x_j}\alpha\,d\mathrm{vol}_g\nonumber\\
		&-\int_{\Sigma}\sum_{i,j=1}^{2}\left(\frac{d}{dt}g^{i,j}_t\right)\left(\p{x_i}\left(\frac{d}{dt}\alpha_t \right)\cdot\p{x_j}\alpha+\p{x_i}\alpha\cdot\p{x_j}\left(\frac{d}{dt}\alpha_t\right)\right)d\vg\nonumber\\
		&-2\int_{\Sigma}\sum_{i,j=1}^{2}\left(\frac{d}{dt}g_t^{i,j}\right)\p{x_i}\alpha\cdot\p{x_j}\alpha\,\s{d\phi}{d\w}_gd\vg\nonumber\\
		&-\int_{\Sigma}\sum_{i,j=1}^{2}g^{i,j}\left(\p{x_i}\left(\frac{d}{dt}\alpha_t\right)\cdot\p{x_j}\alpha+\p{x_i}\alpha\cdot\p{x_j}\left(\frac{d}{dt}\alpha_t\right)\right)\s{d\phi}{d\w}_gd\vg\nonumber\\
		&-\int_{\Sigma}|d\alpha|_g^2\left(|d\w|_g^2+\s{d\phi}{d\w}_g^2-16|\Re(\partial\phi\totimes\bar{\partial}\phi)|_{WP}^2-16|\partial\phi\totimes\partial\w|_{WP}^2\right)d\vg\nonumber\\
		&-2K_{g_0}\int_{\Sigma}\left(\frac{d}{dt}\alpha_t\right)\s{d\phi}{d\w}_gd\mathrm{vol}_{g_0}\nonumber\\
		&+K_{g_0}\int_{\Sigma}\left(|d\w|_g^2+\s{d\phi}{d\w}_g^2-16|\Re(\partial\phi\totimes\bar{\partial}\phi)|_{WP}^2-16|\partial\phi\totimes\partial\w|_{WP}^2\right)d\mathrm{vol}_{g_0}\nonumber\\
		&-2\int_{\Sigma}d\left(\frac{d}{dt}\alpha_t\right)\wedge\omega_t-2\int_{\Sigma}d\alpha\wedge \left(\frac{d}{dt}\omega_t\right)\nonumber\\
		&=\mathrm{(I)}+\mathrm{(II)}
		+\mathrm{(III)}+\mathrm{(IV)}
		+\mathrm{(V)}
		+\mathrm{(VI)}+\mathrm{(VII)}
		+\mathrm{(VIII)}+\mathrm{(IX)}.
	\end{align}
	We will use the convenient notation
	\begin{align*}
		\widetilde{\mathscr{O}}(\phi)=\int_{\Sigma}|d\alpha|_g^2d\vg+2K_{g_0}\int_{\Sigma}\alpha\,e^{-2\alpha}\,d\vg.
	\end{align*}
	Notice that $\widetilde{\mathscr{O}}$ are linked by the formula
	\begin{align*}
		\mathscr{O}(\phi)=\frac{1}{2}\widetilde{\mathscr{O}}(\phi)-\frac{K_{g_0}}{2}\log\int_{\Sigma}d\vg
	\end{align*}
	We compute thanks of \eqref{form}
	\begin{align}\label{b1}
		\mathrm{(I)}&=-\int_{\Sigma}\sum_{i,j=1}^{2}\left(\frac{d^2}{dt^2}g^{i,j}_t\right)\p{x_i}\alpha\cdot\p{x_j}\alpha\,d\vg
		=2\int_{\Sigma}\sum_{i,j=1}^{2}e^{-4\lambda}\s{\D_{\e_i}\w}{\D_{\e_j}\w}\p{x_i}\alpha\cdot\p{x_j}\alpha\,d\vg\nonumber\\
		&-4\int_{\Sigma}\sum_{i,j=1}^{2}e^{-4\lambda}\left(\s{\D_{\e_i}\w}{\e_j}+\s{\D_{\e_j}\w}{\e_i}\right)\p{x_i}\alpha\cdot\p{x_j}\alpha\,\s{d\phi}{d\w}_gd\vg\nonumber\\
		&+\int_{\Sigma}\sum_{i=1}^{2}e^{-2\lambda}\left(\p{x_i}\alpha\right)^2\left(2\s{d\phi}{d\w}_g^2-16|\Re(\partial\phi\totimes\bar{\partial}\w)|_{WP}^2-16|\partial\phi\totimes\partial\w|_{WP}^2\right)d\vg\nonumber\\
		&=2\int_{\Sigma}\s{d\w\totimes d\w}{d\alpha\otimes d\alpha}_gd\vg-4\int_{\Sigma}\s{d\phi\totimes d\w+d\w\totimes d\phi}{d\alpha\otimes d\alpha}_g\s{d\phi}{d\w}_gd\vg\nonumber\\
		&+2\int_{\Sigma}|d\alpha|_g^2\left(2\s{d\phi}{d\w}_g^2-16|\Re(\partial\phi\totimes\bar{\partial}\w)|_{WP}^2-16|\partial\phi\totimes\partial\w|_{WP}^2\right)d\vg.
	\end{align}
	Furthermore, observe that
	\begin{align*}
		g^{-1}\otimes \Re(\partial\phi\totimes \bar{\partial}\w)=\frac{1}{4}e^{-2\lambda}\left(\s{\p{x_1}\phi}{\D_{\e_1}\w}+\s{\p{x_2}\phi}{\D_{\e_2}\w}\right)=\frac{1}{4}\s{d\phi}{d\w}_g
	\end{align*}
	so 
	\begin{align*}
	&2\int_{\Sigma}|d\alpha|_g^2\left(2\s{d\phi}{d\w}_g^2-16|\Re(\partial\phi\totimes\bar{\partial}\w)|_{WP}^2-16|\partial\phi\totimes\partial\w|_{WP}^2\right)d\vg\nonumber\\
	&
	=2\int_{\Sigma}|d\alpha|_g^2\left(\s{d\phi}{d\w}_g^2-16|\partial\phi\totimes\partial\w|_{WP}^2\right)d\vg
\end{align*}
	We estimate directly as $|\s{d\phi}{d\w}_g|\leq \sqrt{2}|d\w|_g$ and $|\D\phi|=\sqrt{2}e^{\lambda}$
	\begin{align}\label{e1}
		|\mathrm{(I)}|\leq (2+4\cdot (2\sqrt{2})\cdot \sqrt{2}+4)\left(\int_{\Sigma}|d\alpha|_g^2d\vg\right)\,\np{|d\w|_g}{\infty}{\Sigma}^2\leq 22\left(\int_{\Sigma}|d\alpha|_g^2d\vg\right)\norm{\w}_{\mathscr{E}_{\phi}(\Sigma,\R^n)}^2.
	\end{align}
	From now on, we write
	\begin{align*}
		\left(\frac{d}{dt}\alpha_t\right)_{|t=0}=\alpha'_0
	\end{align*}
	The next term is
	\begin{align}\label{b2}
		\mathrm{(II)}&=-\int_{\Sigma}\sum_{i,j=1}^{2}\left(\frac{d}{dt}g^{i,j}_t\right)\left(\p{x_i}\left(\frac{d}{dt}\alpha_t\right)\cdot\p{x_j}\alpha+\p{x_i}\alpha\cdot\p{x_j}\left(\frac{d}{dt}\alpha_t\right)\right)d\vg\nonumber\\
		&=\int_{\Sigma}\sum_{i,j=1}^{2}e^{-4\lambda}\left(\s{\D_{\e_i}\w}{\e_j}+\s{\D_{\e_j}\w}{\e_i}\right)\left(\p{x_i}\left(\frac{d}{dt}\alpha_t\right)\cdot\p{x_j}\alpha+\p{x_i}\alpha\cdot\p{x_j}\left(\frac{d}{dt}\alpha_t\right)\right)d\vg\nonumber\\
		&=\int_{\Sigma}\s{d\phi\totimes d\w+d\w\totimes d\phi}{d\alpha\otimes d\alpha_0'+d\alpha_0'\otimes d\alpha}_gd\vg.
	\end{align}
	In particular, we obtain by Cauchy-Schwarz inequality the direct estimate
	\begin{align}\label{e2}
		|\mathrm{(II)}|&\leq 8\left(\int_{\Sigma}|d\alpha|_g^2d\vg\right)^{\frac{1}{2}}\left(\int_{\Sigma}|d\alpha_0'|_g^2d\vg\right)^{\frac{1}{2}}\np{|d\w|}{\infty}{\Sigma}\nonumber\\
		&\leq 8\left(\int_{\Sigma}|d\alpha|_g^2d\vg\right)^{\frac{1}{2}}\left(\int_{\Sigma}|d\alpha_0'|_g^2d\vg\right)^{\frac{1}{2}}\norm{\w}_{\mathscr{E}_{\phi}(\Sigma,\R^n)}.
	\end{align}
	Now, we have
	\begin{align}\label{b3}
		\mathrm{(III)}&=-2\int_{\Sigma}\sum_{i,j=1}^{2}\left(\frac{d}{dt}g^{i,j}_t\right)\p{x_i}\alpha\cdot\p{x_j}\alpha\,\s{d\phi}{d\w}_gd\vg\nonumber\\
		&=2\int_{\Sigma}\sum_{i,j=1}^{2}e^{-4\lambda}\left(\s{\D_{\e_i}\w}{\e_j}+\s{\D_{\e_j}\w}{\e_i}\right)\p{x_i}\alpha\cdot \p{x_j}\alpha\,\s{d\phi}{d\w}_gd\vg\nonumber\\
		&=2\int_{\Sigma}\s{d\phi\totimes d\w+d\w\totimes d\phi}{d\alpha\otimes d\alpha}_g\s{d\phi}{d\w}_gd\vg.
	\end{align}
	This term is also directly estimated as
	\begin{align}\label{e3}
		|\mathrm{(III)}|\leq 8\left(\int_{\Sigma}|d\alpha|_g^2d\vg\right)\np{|d\w|_g}{\infty}{\Sigma}^2\leq 8\left(\int_{\Sigma}|d\alpha|_g^2d\vg\right)\norm{\w}_{\mathscr{E}_{\phi}(\Sigma,\R^n)}^2
	\end{align}
	Then
	\begin{align}\label{b4}
		\mathrm{(IV)}&=-\int_{\Sigma}\sum_{i,j=1}^2g^{i,j}\left(\p{x_i}\left(\frac{d}{dt}\alpha_t\right)\cdot\p{x_j}\alpha+\p{x_i}\alpha\cdot\p{x_j}\left(\frac{d}{dt}\alpha\right)\right)\s{d\phi}{d\w}_gd\vg\nonumber\\
		&=-2\int_{\Sigma}\s{d\alpha}{d\alpha_0'}_g\s{d\phi}{d\w}_gd\vg.
	\end{align}
	This implies that
	\begin{align}\label{e4}
		|\mathrm{(IV)}|&\leq 4\left(\int_{\Sigma}|d\alpha|_g^2d\vg\right)^{\frac{1}{2}}\left(\int_{\Sigma}|d\alpha_0'|_g^2d\vg\right)^{\frac{1}{2}}\np{|d\w|_g}{\infty}{\Sigma}\nonumber\\
		&\leq 4\left(\int_{\Sigma}|d\alpha|_g^2d\vg\right)^{\frac{1}{2}}\left(\int_{\Sigma}|d\alpha_0'|_g^2d\vg\right)^{\frac{1}{2}}\norm{\w}_{\mathscr{E}_{\phi}(\Sigma,\R^n)}.
	\end{align}
	We now see directly that
	\begin{align}\label{b5}
		\mathrm{(V)}
		&=-\int_{\Sigma}|d\alpha|_g^2\left(|d\w|_g^2+\s{d\phi}{d\w}_g^2-16|\Re(\partial\phi\totimes \bar{\partial}\w)|_{WP}^2-16|\partial\phi\totimes\partial\w|_{WP}^2\right)d\vg\nonumber\\
		&=-\int_{\Sigma}|d\alpha|_g^2\left(|d\w|_g^2-16|\partial\phi\totimes \partial\w|_{WP}^2\right)d\vg
	\end{align}
	Therefore, we have
	\begin{align}\label{e5}
		|\mathrm{(V)}|\leq 2\left(\int_{\Sigma}|d\alpha|_g^2d\vg\right)\np{|d\w|_g}{\infty}{\Sigma}\leq 2\left(\int_{\Sigma}|d\alpha|_g^2d\vg\right)\norm{\w}_{\mathscr{E}_{\phi}(\Sigma,\R^n)}^2.
	\end{align}
	Next we recall that
	\begin{align}\label{b6}
		\mathrm{(VI)}=-2K_{g_0}\int_{\Sigma}\alpha_0'\s{d\phi}{d\w}_gd\mathrm{vol}_{g_0}.
	\end{align}
	By a direct estimate we obtain
	\begin{align}\label{e61}
		\left|\mathrm{(VI)}\right|&=\left|-2K_{g_0}\int_{\Sigma}\left(\frac{d}{dt}\alpha_t
		\right)\s{d\phi}{d\w}_gd\mathrm{vol}_{g_0}\right|\leq 2\sqrt{2}|K_{g_0}|\int_{\Sigma}|\alpha_0'|d\mathrm{vol}_{g_0}\np{|d\w|_g}{\infty}{\Sigma}
	\end{align}
	As $g_{0,t}$ is a unit volume metric for all $t\in I$, we have
	\begin{align*}
		0=\frac{d}{dt}\left(\int_{\Sigma}d\mathrm{vol}_{g_{0,t}}\right)_{|t=0}=\frac{d}{dt}\left(\int_{\Sigma}e^{-2\alpha_t}d\mathrm{vol}_{g_t}\right)_{|t=0}=-2\int_{\Sigma}\alpha_0'd\mathrm{vol}_{g_0}+\int_{\Sigma}\s{d\phi}{d\w}_gd\mathrm{vol}_{g_0}.
	\end{align*}
	Therefore, we have
	\begin{align}\label{e62}
		\left|\int_{\Sigma}\alpha_0'd\mathrm{vol}_{g_0}\right|=\frac{1}{2}\left|\int_{\Sigma}\s{d\phi}{d\w}_gd\mathrm{vol}_{g_0}\right|\leq \frac{1}{\sqrt{2}}\np{|d\w|_g}{\infty}{\Sigma}.
	\end{align}
	Furthermore, by Cauchy-Schwarz and Poincar\'{e}-Wirtinger inequalities, we have for some universal constant $C_{\mathrm{PW}}=C_{\mathrm{PW}}(g_0)$ the estimate (as $g_0$ has unit volume)
	\begin{align}\label{e63}
		\left(\int_{\Sigma}\left|\alpha_0'-\int_{\Sigma}\alpha_0'd\mathrm{vol}_{g_0}\right|d\mathrm{vol}_{g_0}\right)^{2}&\leq \int_{\Sigma}\left|\alpha_0'-\int_{\Sigma}\alpha_0'd\mathrm{vol}_{g_0}\right|^2d\mathrm{vol}_{g_0}\leq C_{\mathrm{PW}}^2\int_{\Sigma}|d\alpha_0'|_{g_0}^2d\mathrm{vol}_{g_0}\nonumber\\
		&=C_{\mathrm{PW}}^2\int_{\Sigma}|d\alpha_0'|_g^2d\vg,
	\end{align}
	where the last inequality follows from the conformal invariance of the Laplacian in dimension $2$.
	Therefore, we have by \eqref{e62} and \eqref{e63} the estimate
	\begin{align}\label{e64}
		\int_{\Sigma}|\alpha_0'|d\mathrm{vol}_{g_0}\leq \int_{\Sigma}\left|\alpha_0'-\int_{\Sigma}\alpha_0'd\mathrm{vol}_{g_0}\right|d\mathrm{vol}_{g_0}+\left|\int_{\Sigma}\alpha_0'd\mathrm{vol}_{g_0}\right|\leq C\left(\int_{\Sigma}|d\alpha_0'|_g^2d\mathrm{vol}_g\right)^{\frac{1}{2}}+\frac{1}{\sqrt{2}}\np{|d\w|_g}{\infty}{\Sigma}.
	\end{align}
	Finally, by \eqref{e61} and \eqref{e64}, we have
	\begin{align}\label{e6}
		|\mathrm{(VI)}|&\leq 4C_{PW}|K_{g_0}|\left(\int_{\Sigma}|d\alpha_0'|^2d\mathrm{vol}_g\right)^{\frac{1}{2}}\np{|d\w|_g}{\infty}{\Sigma}+2|K_{g_0}|\np{|d\w|_g}{\infty}{\Sigma}^2\nonumber\\
		&= 8\pi C_{PW}|\chi(\Sigma)|\left(\int_{\Sigma}|d\alpha_0'|^2d\mathrm{vol}_g\right)^{\frac{1}{2}}\norm{\w}_{\mathscr{E}_{\phi}(\Sigma,\R^n)}+4\pi|\chi(\Sigma)|\,\norm{\w}_{\mathscr{E}_{\phi}(\Sigma,\R^n)}^2.
	\end{align}
	As
	\begin{align}\label{b7}
		\mathrm{(VII)}&=K_{g_0}\int_{\Sigma}\left(|d\w|_g^2+\s{d\phi}{d\w}_g^2-16|\Re(\partial\phi\totimes\bar{\partial}\w)|_{WP}^2-16|\partial\phi\totimes\partial\w|_{WP}^2\right)d\mathrm{vol}_{g_0}\nonumber\\
		&=K_{g_0}\int_{\Sigma}\left(|d\w|_g^2-16|\partial\phi\totimes\partial\w|_{WP}^2\right)d\mathrm{vol}_{g_0}
	\end{align}
	we trivially have as $g_0$ has unit volume
	\begin{align}\label{e7}
		|\mathrm{(VII)}|\leq 2|K_{g_0}|\np{|d\w|_g}{\infty}{\Sigma}^2=8\pi|\chi(\Sigma)|\np{|d\w|_g}{\infty}{\Sigma}^2\leq 8\pi|\chi(\Sigma)|\norm{\w}_{\mathscr{E}_{\phi}(\Sigma,\R^n)}^2.
	\end{align}
	Now, we have by Lemma \ref{d2k}
	\begin{align}\label{b8}
		\mathrm{(VIII)}=-2\int_{\Sigma}d\left(\frac{d}{dt}\alpha_t\right)\wedge \omega_t=-4\int_{\Sigma}d\alpha_0'\wedge \Im\left(\s{\H}{\partial\w}-g^{-1}\otimes\left(\h_0\totimes\bar{\partial}\w\right)\right).
	\end{align}	
	Therefore, this quantity is directly estimated thanks of Cauchy-Schwarz inequality as
	\begin{align}\label{e8}
		|\mathrm{(VIII)}|&\leq 2\left(\int_{\Sigma}|d\alpha_0'|_g^2d\mathrm{vol}_g\right)^{\frac{1}{2}}\left(\sqrt{W(\phi)}+\sqrt{W(\phi)-2\pi\chi(\Sigma)}\right)\np{|d\w|_g}{\infty}{\Sigma}\nonumber\\
		&\leq 2\left(\int_{\Sigma}|d\alpha_0'|_g^2d\mathrm{vol}_g\right)^{\frac{1}{2}}\left(\sqrt{W(\phi)}+\sqrt{W(\phi)-2\pi\chi(\Sigma)}\right)\norm{\w}_{\mathscr{E}_{\phi}(\Sigma,\R^n)}.
	\end{align}
	Finally, we have by Lemma \ref{d2k}
	\begin{align}\label{b91}
		\mathrm{(IX)}&=-2\int_{\Sigma}d\alpha\wedge \left(\frac{d}{dt}\omega_t\right)\nonumber\\
		&=-2\int_{\Sigma}d\alpha\wedge \Im\,\bigg(2\s{\Delta_g^{\perp}\w+4\,\Re\left(g^{-2}\otimes\left(\bar{\partial}\phi\totimes\bar{\partial}\w\right)\otimes \h_0\right)}{\partial\w}-\partial\left(|\D^{\perp}\w|^2_g\right)\nonumber\nonumber\\
		&-2\s{d\phi}{d\w}_g\left(\s{\H}{\partial \w}-g^{-1}\otimes\left(\h_0\totimes\bar{\partial}\w\right)\right)
		-8\,g^{-1}\otimes \left(\partial\phi\totimes\partial\w\right)\otimes \s{\H}{\bar{\partial}\w}\bigg).
	\end{align}
	Now, 
	we have as $\Sigma$ is closed
	\begin{align*}
		2\int_{\Sigma}d\alpha\wedge \Im(\partial|\D^{\perp}\w|_g^2)&=\int_{\Sigma}d\alpha\wedge \ast d|\D^{\perp}\w|_g^2=\int_{\Sigma}\s{d\alpha}{d|\D^{\perp}\w|_g^2}_gd\vg
		=-\int_{\Sigma}|\D^{\perp}\w|_g^2\,\Delta_{g}\alpha\,d\vg.
	\end{align*}
	Therefore, we have
	\begin{align}\label{b9}
		\mathrm{(IX)}&=-\int_{\Sigma}|\D^{\perp}\w|_g^2\,\Delta_{g}\alpha\,d\mathrm{vol}_{g}-4\int_{\Sigma}d\alpha\wedge\,\Im\bigg(\s{\Delta_{g}^{\perp}\w+4\,\Re\left(g^{-2}\otimes\left(\bar{\partial}\phi\totimes\bar{\partial}\w\right)\otimes \h_0\right)}{\partial\w}\nonumber\\
		&-\s{d\phi}{d\w}_g\left(\s{\H}{\partial\w}-g^{-1}\otimes\left(\h_0\totimes\bar{\partial}\w\right)\right)
		-4\,g^{-1}\otimes\left(\partial\phi\totimes\partial\w\right)\otimes\s{\H}{\bar{\partial}\w}\bigg)
	\end{align}
	Now, by the Liouville equation, we have (recall that $g=e^{2\alpha}g_0$)
	\begin{align*}
		-\Delta_g\alpha=K_{g}-K_{g_0}e^{-2\alpha},
	\end{align*}
	and as $|K_g|\leq \frac{1}{2}|\vec{\I}_g|^2=2|\H|^2-K_g$, we deduce that
	\begin{align}\label{e91}
		\left|-\int_{\Sigma}|\D^{\perp}\w|_g^2\,\Delta_g\alpha\,d\vg\right|&\leq \left(\int_{\Sigma}|K_g|d\vg+\int_{\Sigma}|K_{g_0}|d\mathrm{vol}_{g_0}\right)\np{|d\w|_g}{\infty}{\Sigma}^2\nonumber\\
		&\leq \left(2\int_{\Sigma}|\H|^2d\vg-\int_{\Sigma}K_gd\vg+|K_{g_0}|\right)\np{|d\w|_g}{\infty}{\Sigma}^2\nonumber\\
		&=\left(2\,W(\phi)+2\pi|\chi(\Sigma)|-2\pi\chi(\Sigma)\right)\np{|d\w|_g}{\infty}{\Sigma}^2
	\end{align}
	
	 Then we have by Cauchy-Schwarz and Minkowski's inequalities
	\begin{align}\label{e92}
		&\left|\mathrm{(IX)}+\int_{\Sigma}|\D^{\perp}\w|_g^2\,\Delta_g\,d\vg\right|\leq 4\left(\int_{\Sigma}|d\alpha|_g^2d\vg\right)^{\frac{1}{2}}\bigg\{\frac{1}{2}\np{|d\w|_g}{\infty}{\Sigma}\left(\int_{\Sigma}|\Delta_g^{\perp}\w|_g^2d\vg\right)^{\frac{1}{2}}
		\nonumber\\
		&+\frac{1}{\sqrt{2}}\left(\int_{\Sigma}|\h_0|^2_{WP}d\vg\right)^{\frac{1}{2}}\np{|d\w|_g}{\infty}{\Sigma}^2+\left(\left(\int_{\Sigma}|\H|^2d\vg\right)^{\frac{1}{2}}+\left(\int_{\Sigma}|\h_0|_{WP}^2d\vg\right)^{\frac{1}{2}}\right)\np{|d\w|_g}{\infty}{\Sigma}^2\nonumber\\
		&+\frac{1}{\sqrt{2}}\left(\int_{\Sigma}|\H|^2d\vg\right)^{\frac{1}{2}}\np{|d\w|_g}{\infty}{\Sigma}^2\bigg\}
		=\left(\int_{\Sigma}|d\alpha|_g^2d\vg\right)^{\frac{1}{2}}\bigg\{
		2\left(\int_{\Sigma}|\Delta_g^{\perp}\w|_g^2d\vg\right)^{\frac{1}{2}}
		\nonumber\\
		&+8\left(\sqrt{W(\phi)}+\sqrt{W(\phi)-2\pi\chi(\Sigma)}\right)\np{|d\w|_g}{\infty}{\Sigma}
		\bigg\}\np{|d\w|_g}{\infty}{\Sigma}
	\end{align}
	Therefore, we finally obtain by \eqref{e91} and \eqref{e92} the bound
	\begin{align}\label{e9}
		&|\mathrm{(IX)}|\leq \left(2\,W(\phi)+2\pi|\chi(\Sigma)|-2\pi\chi(\Sigma)\right)\np{|d\w|_g}{\infty}{\Sigma}^2
		+\left(\int_{\Sigma}|d\alpha|_g^2d\vg\right)^{\frac{1}{2}}\bigg\{
		2\left(\int_{\Sigma}|\Delta_g^{\perp}\w|_g^2d\vg\right)^{\frac{1}{2}}
		\nonumber\\
		&+8\left(\sqrt{W(\phi)}+\sqrt{W(\phi)-2\pi\chi(\Sigma)}\right)\np{|d\w|_g}{\infty}{\Sigma}
		\bigg\}\np{|d\w|_g}{\infty}{\Sigma}\nonumber\\
		&\leq 2\left(\,W(\phi)+2\pi|\chi(\Sigma)|+\left(2+5\sqrt{W(\phi)}+5\sqrt{W(\phi)-2\pi\chi(\Sigma)}\right)\left(\int_{\Sigma}|d\alpha|_g^2d\vg\right)^{\frac{1}{2}}\right)\norm{\w}_{\mathscr{E}_{\phi}(\Sigma,\R^n)}^2.
	\end{align}
	Now, by \eqref{b0}, \eqref{b1}, \eqref{b2},  \eqref{b3}, \eqref{b4}, \eqref{b5}, \eqref{b6}, \eqref{b7}, \eqref{b8}, and \eqref{b9}, we obtain
	\begin{align*}
		&D^2\widetilde{\mathscr{O}}(\phi)(\w,\w)=2\int_{\Sigma}\s{d\w\totimes d\w}{d\alpha\otimes d\alpha}_gd\vg-4\int_{\Sigma}\s{d\phi\totimes d\w+d\w\totimes d\phi}{d\alpha\otimes d\alpha}_g\s{d\phi}{d\w}_gd\vg\nonumber\\
		&+2\int_{\Sigma}|d\alpha|_g^2\left(\s{d\phi}{d\w}_g^2-16|\partial\phi\totimes\partial\w|_{WP}^2\right)d\vg
	    +\int_{\Sigma}\s{d\phi\totimes d\w+d\w\totimes d\phi}{d\alpha\otimes d\alpha_0'+d\alpha_0'\otimes d\alpha}_gd\vg\\
		&+2\int_{\Sigma}\s{d\phi\totimes d\w+d\w\totimes d\phi}{d\alpha\otimes d\alpha}_g\s{d\phi}{d\w}_gd\vg
		-2\int_{\Sigma}\s{d\alpha}{d\alpha_0'}_g\s{d\phi}{d\w}_gd\vg\\
		&-\int_{\Sigma}|d\alpha|_g^2\left(|d\w|_g^2-16|\partial\phi\totimes\partial\w|_{WP}^2\right)d\vg
		-2K_{g_0}\int_{\Sigma}\alpha_0'\s{d\phi}{d\w}d\mathrm{vol}_{g_0}\\		&+K_{g_0}\int_{\Sigma}\left(|d\w|_g^2-16|\partial\phi\totimes\partial\w|_{WP}^2\right)d\mathrm{vol}_{g_0}
		-4\int_{\Sigma}d\alpha_0'\wedge \Im\left(\s{\H}{\partial\w}-g^{-1}\otimes\left(\h_0\totimes\bar{\partial}\w\right)\right)	\nonumber\\
		&-4\int_{\Sigma}d\alpha\wedge\,\Im\bigg(\s{\Delta_{g}^{\perp}\w+4\,\Re\left(g^{-2}\otimes\left(\bar{\partial\phi}\totimes\bar{\partial}\w\right)\otimes \h_0\right)}{\partial\w}
		-\s{d\phi}{d\w}_g\left(\s{\H}{\partial\w}-g^{-1}\otimes\left(\h_0\totimes\bar{\partial}\w\right)\right)\nonumber\\
		&-4\,g^{-1}\otimes\left(\partial\phi\totimes\partial\w\right)\otimes\s{\H}{\bar{\partial}\w}\bigg)-\int_{\Sigma}|\D^{\perp}\w|_g^2\,\Delta_{g}\alpha\,d\mathrm{vol}_{g}\nonumber\\
		&=2\int_{\Sigma}\s{d\alpha\otimes d\alpha}{d\w\totimes d\w}_gd\vg-2\int_{\Sigma}\s{d\alpha\otimes d\alpha}{d\phi\totimes d\w+d\w\totimes d\phi}_g\s{d\phi}{d\w}_gd\vg\nonumber\\
		&-\int_{\Sigma}|d\alpha|_g^2\left(|d\w|_g^2-2\s{d\phi}{d\w}_g^2-16|\partial\phi\totimes\partial\w|_{WP}^2\right)d\vg
		+K_{g_0}\int_{\Sigma}\left(|d\w|_g^2-16|\partial\phi\totimes\partial\w|_{WP}^2\right)d\mathrm{vol}_{g_0}\\
		&-4\int_{\Sigma}d\alpha\wedge\,\Im\bigg(\s{\Delta_{g}^{\perp}\w+4\,\Re\left(g^{-2}\otimes\left(\bar{\partial\phi}\totimes\bar{\partial}\w\right)\otimes \h_0\right)}{\partial\w}
		-\s{d\phi}{d\w}_g\left(\s{\H}{\partial\w}-g^{-1}\otimes\left(\h_0\totimes\bar{\partial}\w\right)\right)\\
		&-4\,g^{-1}\otimes\left(\partial\phi\totimes\partial\w\right)\otimes\s{\H}{\bar{\partial}\w}\bigg)-\int_{\Sigma}|\D^{\perp}\w|_g^2\,\Delta_{g}\alpha\,d\mathrm{vol}_{g}\\
		&+\int_{\Sigma}\s{d\phi\totimes d\w+d\w\totimes d\phi}{d\alpha\otimes d\alpha_0'+d\alpha_0'\otimes d\alpha}_gd\vg-2\int_{\Sigma}\s{d\alpha}{d\alpha_0'}_g\s{d\phi}{d\w}_gd\vg\\
		&-4\int_{\Sigma}d\alpha_0'\wedge \Im\left(\s{\H}{\partial\w}-g^{-1}\otimes\left(\h_0\totimes\bar{\partial}\w\right)\right)-2K_{g_0}\int_{\Sigma}\alpha_0'\s{d\phi}{d\w}d\mathrm{vol}_{g_0}
	\end{align*}
	By gathering the estimates \eqref{e1}, \eqref{e2}, \eqref{e3}, \eqref{e4}, \eqref{e5}, \eqref{e6}, \eqref{e7}, \eqref{e8}, \eqref{e9}, we obtain
	\begin{align}\label{ee1}
		&|D^2\widetilde{\mathscr{O}}(\phi)(\w,\w)|\leq 
		22\left(\int_{\Sigma}|d\alpha|_g^2d\vg\right)\norm{\w}_{\mathscr{E}_{\phi}(\Sigma,\R^n)}^2
		+8\left(\int_{\Sigma}|d\alpha|_g^2d\vg\right)^{\frac{1}{2}}\left(\int_{\Sigma}|d\alpha_0'|_g^2d\vg\right)^{\frac{1}{2}}\norm{\w}_{\mathscr{E}_{\phi}(\Sigma,\R^n)}\nonumber\\
		&+8\left(\int_{\Sigma}|d\alpha|_g^2d\vg\right)\norm{\w}_{\mathscr{E}_{\phi}(\Sigma,\R^n)}^2
		+4\left(\int_{\Sigma}|d\alpha|_g^2d\vg\right)^{\frac{1}{2}}\left(\int_{\Sigma}|d\alpha_0'|_g^2d\vg\right)^{\frac{1}{2}}\norm{\w}_{\mathscr{E}_{\phi}(\Sigma,\R^n)}\nonumber\\
		&+2\left(\int_{\Sigma}|d\alpha|_g^2d\vg\right)\norm{\w}_{\mathscr{E}_{\phi}(\Sigma,\R^n)}^2
		+8\pi C_{PW}|\chi(\Sigma)|\left(\int_{\Sigma}|d\alpha_0'|^2d\mathrm{vol}_g\right)^{\frac{1}{2}}\norm{\w}_{\mathscr{E}_{\phi}(\Sigma,\R^n)}+4\pi|\chi(\Sigma)|\,\norm{\w}_{\mathscr{E}_{\phi}(\Sigma,\R^n)}^2\nonumber\\
		&+8\pi|\chi(\Sigma)|\norm{\w}_{\mathscr{E}_{\phi}(\Sigma,\R^n)}^2
		+2\left(\int_{\Sigma}|d\alpha_0'|_g^2d\mathrm{vol}_g\right)^{\frac{1}{2}}\left(\sqrt{W(\phi)}+\sqrt{W(\phi)-2\pi\chi(\Sigma)}\right)\norm{\w}_{\mathscr{E}_{\phi}(\Sigma,\R^n)}\nonumber\\
		&+2\left(\,W(\phi)+2\pi|\chi(\Sigma)|+\left(2+5\sqrt{W(\phi)}+5\sqrt{W(\phi)-2\pi\chi(\Sigma)}\right)\left(\int_{\Sigma}|d\alpha|_g^2d\vg\right)^{\frac{1}{2}}\right)\norm{\w}_{\mathscr{E}_{\phi}(\Sigma,\R^n)}^2\nonumber\\
		&=2\left(W(\phi)+8\pi|\chi(\Sigma)|+16\int_{\Sigma}|d\alpha|_g^2d\vg\right)\norm{\w}_{\mathscr{E}_{\phi}(\Sigma,\R^n)}^2\nonumber\\
		&+2\left(2+5\sqrt{W(\phi)}+5\sqrt{W(\phi)-2\pi\chi(\Sigma)}\right)\left(\int_{\Sigma}|d\alpha|_g^2d\vg\right)^{\frac{1}{2}}\norm{\w}_{\mathscr{E}_{\phi}(\Sigma,\R^n)}^2\nonumber\\
		&+2\left(4\pi C_{\mathrm{PW}}|\chi(\Sigma)|+\sqrt{W(\phi)}+\sqrt{W(\phi)-2\pi\chi(\Sigma)}+6\left(\int_{\Sigma}|d\alpha|_g^2d\vg\right)^{\frac{1}{2}}\right)\left(\int_{\Sigma}|d\alpha_0'|^2_g\vg\right)^{\frac{1}{2}}\norm{\w}_{\mathscr{E}_{\phi}(\Sigma,\R^n)}.
		\end{align}
	We compute immediately
	\begin{align*}
		&\frac{d^2}{dt^2}\left(\log\int_{\Sigma}d\mathrm{vol}_{g_t}\right)_{|t=0}=\frac{d}{dt}\left(\frac{1}{\int_{\Sigma}d\mathrm{vol}_{g_t}}\int_{\Sigma}\s{d\phi_t}{d\w_t}_{g_t}d\mathrm{vol}_{g_t}\right)_{|t=0}
		=-\left(\dfrac{1}{\int_{\Sigma}d\vg}\left(\int_{\Sigma}\s{d\phi}{d\w}_gd\vg\right)\right)^2\\
		&+\dfrac{1}{\int_{\Sigma}d\vg}\int_{\Sigma}\left(|d\w|_g+\s{d\phi}{d\w}_g^2-16|\Re(\partial\phi\totimes\bar{\partial}\w)|_{WP}^2-16|\partial\phi\totimes\partial\w|_{WP}^2\right)d\vg
	\end{align*}
	Therefore, we have
	\begin{align*}
		&D^2\left(\mathscr{O}-\frac{1}{2}\widetilde{\mathscr{O}}\right)(\phi)(\w,\w)=\frac{K_{g_0}}{2}\left(\frac{1}{\int_{\Sigma}d\vg}\left(\int_{\Sigma}\s{d\phi
	}{d\w}_gd\vg\right)^2\right)\\
&-\frac{K_{g_0}}{2}\frac{1}{\int_{\Sigma}d\vg}\int_{\Sigma}\left(|d\w|_g^2+\s{d\phi}{d\w}_g^2-16|\Re(\partial\phi\totimes\bar{\partial}\w)|_{WP}^2-16|\partial\phi\totimes{\partial}\w|_{WP}^2\right)d\vg.
	\end{align*}
	We estimate this term directly as
	\begin{align}\label{ee2}
		&\left|D^2\left(\mathscr{O}-\frac{1}{2}\widetilde{\mathscr{O}}\right)(\phi)(\w,\w)\right|\leq 2|K_{g_0}|\np{|d\w|_g}{\infty}{\Sigma}^2=4\pi|\chi(\Sigma)|\np{|d\w|_g}{\infty}{\Sigma}^2\leq 4\pi|\chi(\Sigma)|\norm{\w}_{\mathscr{E}_{\phi}(\Sigma,\R^n)}^2.
	\end{align}
	Finally, by \eqref{ee1} and \eqref{ee2}, we obtain
	\begin{align*}
		&|D^2\mathscr{O}(\phi)(\w,\w)|\leq \left(W(\phi)+12\pi|\chi(\Sigma)|+16\int_{\Sigma}|d\alpha|_g^2d\vg\right)\norm{\w}_{\mathscr{E}_{\phi}(\Sigma,\R^n)}^2\nonumber\\
		&+\left(2+5\sqrt{W(\phi)}+5\sqrt{W(\phi)-2\pi\chi(\Sigma)}\right)\left(\int_{\Sigma}|d\alpha|_g^2d\vg\right)^{\frac{1}{2}}\norm{\w}_{\mathscr{E}_{\phi}(\Sigma,\R^n)}^2\nonumber\\
		&+\left(4\pi C_{\mathrm{PW}}|\chi(\Sigma)|+\sqrt{W(\phi)}+\sqrt{W(\phi)-2\pi\chi(\Sigma)}+6\left(\int_{\Sigma}|d\alpha|_g^2d\vg\right)^{\frac{1}{2}}\right)\left(\int_{\Sigma}|d\alpha_0'|^2_g\vg\right)^{\frac{1}{2}}\norm{\w}_{\mathscr{E}_{\phi}(\Sigma,\R^n)}
	\end{align*}
	which concludes the proof of the theorem.
\end{proof}

\begin{cor}\label{endproof}
	Under the hypothesis of Theorem \ref{d2o}, suppose if $\Sigma=S^2$ and that we have fixed an Aubin gauge. Then $D^2\mathscr{O}(\phi)$ is a continuous linear map on $\mathscr{E}_{\phi}(\Sigma,\R^n)$ and we if $\Sigma=S^2$for have for a universal constant $C$ depending only on the genus of $\Sigma$ such that
	\begin{align*}
		\norm{D^2\mathscr{O}(\phi)}_{\mathscr{E}'_{\phi}(\Sigma,\R^n)}\leq C\left(1+C_{PW}+W(\phi)+\int_{\Sigma}|d\alpha|_g^2d\vg\right),
	\end{align*}
	where $C_{\mathrm{PW}}=C_{\mathrm{PW}}(g_0)$ is the Poincar\'{e}-Wirtinger constant for the injection $L^{2}(\Sigma,g_0)/\R\hookrightarrow \dot{W}^{1,2}(\Sigma,g_0)$.
\end{cor}
\begin{proof}
	First, by Theorem \ref{d2o} and as $C_{\mathrm{PW}}\left(\dfrac{g_{S^2}}{4\pi}\right)=\dfrac{1}{2}$, 
	\begin{align}\label{c1}
	&|D^2\mathscr{O}(\phi)(\w,\w)|\leq\left(W(\phi)+24\pi+16\int_{S^2}|d\alpha|_g^2d\vg\right)\norm{\w}_{\mathscr{E}_{\phi}(S^2,\R^n)}^2\nonumber\\
	&+\left(2+5\sqrt{W(\phi)}+5\sqrt{W(\phi)-4\pi}\right)\left(\int_{S^2}|d\alpha|_g^2d\vg\right)^{\frac{1}{2}}\norm{\w}_{\mathscr{E}_{\phi}(S^2,\R^n)}^2\nonumber\\
	&+\left(4\pi +\sqrt{W(\phi)}+\sqrt{W(\phi)-4\pi}+6\left(\int_{S^2}|d\alpha|_g^2d\vg\right)^{\frac{1}{2}}\right)\left(\int_{S^2}|d\alpha_0'|^2_g\vg\right)^{\frac{1}{2}}\norm{\w}_{\mathscr{E}_{\phi}(S^2,\R^n)}
\end{align}
	Thanks of Theorem \ref{mainestimate}, we have
	\begin{align}\label{c2}
		&\int_{S^2}|d\alpha'_0|_g^2d\vg\leq  \left(C+\sqrt{W(\phi)}+\sqrt{W(\phi)-4\pi}+4\left(\int_{S^2}|d\alpha|_g^2d\vg\right)^{\frac{1}{2}}\right)\np{|d\w|_g}{\infty}{S^2}.
	\end{align}
	 Combining \eqref{c1} and \eqref{c2}, we obtain by Cauchy's inequality
	 \begin{align}\label{nscal}
	 	&|D^2\mathscr{O}(\phi)(\w,\w)|\leq  \left(W(\phi)+24\pi+16\int_{S^2}|d\alpha|_g^2d\vg\right)\norm{\w}_{\mathscr{E}_{\phi}(S^2,\R^n)}^2\nonumber\\
	 	&+\left(2+5\sqrt{W(\phi)}+5\sqrt{W(\phi)-4\pi}\right)\left(\int_{S^2}|d\alpha|_g^2d\vg\right)^{\frac{1}{2}}\norm{\w}_{\mathscr{E}_{\phi}(S^2,\R^n)}^2\nonumber\\
	 	&+\left(4\pi +\sqrt{W(\phi)}+\sqrt{W(\phi)-4\pi}+6\left(\int_{S^2}|d\alpha|_g^2d\vg\right)^{\frac{1}{2}}\right)\left(\int_{S^2}|d\alpha_0'|^2_g\vg\right)^{\frac{1}{2}}\norm{\w}_{\mathscr{E}_{\phi}(S^2,\R^n)}\nonumber\\
	 	&\times \bigg\{
	 	\left(C+\sqrt{W(\phi)}+\sqrt{W(\phi)-4\pi}+4\left(\int_{S^2}|d\alpha|_g^2d\vg\right)^{\frac{1}{2}}\right)\norm{\w}_{\mathscr{E}_{\phi}(S^2,\R^n)}^2\bigg\}\nonumber\\
	 	&\leq C'\left(1+W(\phi)+\int_{S^2}|d\alpha|_g^2d\vg\right)\norm{\w}_{\mathscr{E}_{\phi}(S^2,\R^n)}^2.
	 \end{align}
	 The estimate for $\Sigma\neq S^2$ is the exact identical and this concludes the proof of the corollary.
\end{proof}

\section{The approximate tangent space at branched immersions}\label{bfinite}

In this section, we introduce the different possible definitions for the Morse index of a branched Willmore immersions. As unbranched Willmore surfaces are real-analytic, there is no difficulty in defining the Morse index as the dimension of the maximal subspace where the second derivative is negative definite. However,  the space of (branched) weak immersions cannot be endowed with a Banach space structure, there are \emph{a priori} three possible definitions for the Morse index depending on the smoothness of the variations considered.

\subsection{Branched weak immersions with finite total curvature}

\begin{defi}
	Let $\Sigma$ be a closed Riemann surface and fix a conformal map $\phi\in C^{l,\alpha}(\Sigma, \R^n)$ for some  $l\geq 1$ and $0\leq \alpha\leq 1$. We say that $\phi$ has a branch point of order $\theta_0\geq 2$ at some point $p\in U$ if there exists a neighbourhood $U$ of $p$ such that the restriction $\phi|_{U\setminus \ens{p}}:U\setminus \ens{p}\rightarrow \R^n$ is an immersion and if there exists a coordinate $z:U\rightarrow D^2\subset \C$ such that $z(p)=0$ and for some $\vec{A}_0\in \C^n\setminus\ens{0}$, 
	\begin{align*}
		\p{z}\phi(z)=\vec{A}_0z^{\theta_0-1}+o\left(z^{\theta_0-1}\right).
	\end{align*}
	We say that $\phi$ is a branched immersion if $\phi$ is an immersion outside finitely many points $p_1,\cdots,p_m\in \Sigma$ which are branch points of $\phi$. We say that a branched immersion $\phi:\Sigma\rightarrow \R^n$ has finite total curvature if
	\begin{align*}
		\int_{\Sigma}|\mathbb{\vec{I}}|_g^2d\mathrm{vol}_{g}=\int_{\Sigma}\left(4|\vec{H}_{g}|^2-2K_g\right)d\mathrm{vol}_g<\infty.
	\end{align*}
	if $g=\phi^{\ast}g_{\,\R^n}$ is the (branched) pull-back metric and $\vec{\I}_g$ is the second fundamental form of $\phi$. 
\end{defi}

	Thanks of the Gauss-Bonnet theorem for $C^{1,\alpha}$ (for some $0<\alpha\leq 1$) branched immersions $\phi:\Sigma\rightarrow \R^n$ of finite total curvature (see \cite{kusnerpacific}), we have
	\begin{align*}
		\int_{\Sigma}{K_g}d\vg=2\pi\,\chi(\Sigma)+2\pi\sum_{p\in \Sigma}^{}(\theta_0(p)-1)
	\end{align*}
	which is a finite quantity as there are finitely many branch points. In particular, finite total curvature and finite Willmore energy are equivalent for $C^{1,\alpha}$ (with $0<\alpha\leq 1$) branched immersions with finite total curvature .

The definition of branch point can be actually relaxed to the setting of weak $W^{2,2}$-immersions : this is a classical theorem of M\"{u}ller and \v{S}ver\'{a}k (\cite{muller}), using the previous contribution of Shiohama (\cite{shiohama}).  The version that we will use is given by the following proposition.
\begin{prop}[Rivi\`{e}re, \cite{rivierecrelle} \label{riv}]
	Let $\phi\in W^{1,2}(D^2,\R^n)\cap W^{2,2}_{loc}(D^2\setminus\ens{0},\R^n)$ be a conformal immersion of $D^2\setminus\ens{0}$ such that $\phi$ has finite total curvature and $\log|\D\phi|\in L^{\infty}_{\mathrm{loc}}(D^2\setminus\ens{0})$. Then $\phi\in W^{1,\infty}(D^2,\R^n)$, $\phi$ is conformal on $D^2$, and there exists an \emph{integer} $\theta_0\in \N\setminus \ens{0}$ and $\alpha>0$ such that
	\begin{align*}
		e^{2\lambda}=2|\p{z}\phi(z)|^2=\alpha|z|^{2\theta_0-2}\left(1+o(1)\right).
	\end{align*}
\end{prop}
This Proposition first allows us to relax the notion of branched immersions. Weak branched immersions are defined as follows
\begin{align*}
&\mathscr{E}_b(\Sigma,\R^n)=W^{2,2}\cap W^{1,\infty}(\Sigma,\R^n)\cap\bigg\{\phi:\;\,\phi\;\, \text{has finite total curvature and}\\
&\text{there exists}\; p_1,\cdots,p_m\in \Sigma\;\, \text{such that for all}\;  \epsilon>0,\;\, \inf_{\Sigma \setminus \cup_{j=1}^mD_{\epsilon}^2(p_j)}|d\phi\wedge d\phi|_{g_0}>0\bigg\}
\end{align*}
Here $D^2_{\epsilon}(p_j)$ designs a geodesic disk of radius $\epsilon>0$ around $p_j$ with respect to any fixed smooth metric.

We now define the order of a point of a weak branched immersion as follows.
\begin{prop}\label{branch}
	Let $\phi\in \mathscr{E}_b(\Sigma,\R^n)$ be a weak immersion of finite total curvature. Then for all $p\in \Sigma$, there exists an \emph{integer} $\theta_0=\theta_0(p)\geq 1$ such that for all complex chart $z$ defined on an neighbourhood of $p$ such that $z(p)=0$, there exists $\alpha=\alpha(z)>0$ such that
	\begin{align*}
		e^{2\lambda}=2|\p{z}\phi|^2=\alpha|z|^{2\theta_0-2}\left(1+o(1)\right).
	\end{align*}
	We say that $p$ is a branch point of order $\theta_0(p)\geq 1$.
\end{prop}
\begin{rem}
	This definition of order of branch point does not agree with the classical definition of branch point of algebraic curves (see \cite{griffiths} for example), as the order of the branch point is defined to be $\theta_0(p)-1>0$ (whenever this strict inequality holds), while for $\theta_0(p)=1$ we see that the curve is unbranched at $p$. However, in the general context where we set ourselves logarithmic singularities may occur for $\theta_0=1$ (the inversion of the catenoid for example exhibits this singular behaviour), which are of course excluded for algebraic curves, but need to be taken into account here. 
\end{rem}

\begin{lemme}\label{gb}
	Let $\phi\in \mathscr{E}_b(\Sigma,\R^n)$ be a branched immersion of finite total curvature. Then the following Gauss-Bonnet formula holds
	\begin{align*}
		\int_{\Sigma}K_gd\vg=2\pi\chi(\Sigma)+2\pi\sum_{p\in \Sigma}^{}(\theta_0(p)-1).
	\end{align*}
	In particular, $\phi$ has finite total curvature if and only if it has finite Willmore energy.
\end{lemme}
\begin{proof}
	Let $\alpha:\Sigma\rightarrow \R$ be the measurable function given by the uniformisation theorem such that 
	\begin{align*}
		g=e^{2\alpha}g_0,
	\end{align*}
	where $g_0$ is a metric of constant Gauss curvature $K_{g_0}$ and unit volume.
	Thanks of Proposition \ref{riv}, we obtain by the Liouville equation
	\begin{align}\label{sum}
	-\Delta_g\alpha=K_g-K_{g_0}e^{-2\alpha}-2\pi \sum_{p\in \Sigma}^{}(\theta_0(p)-1)\delta_{p},
	\end{align}
	where only finitely many terms are non-zero in the sum in \eqref{sum}.
	By integrating this equation with respect to the natural volume form we obtain the desired formula
	\begin{align*}
		\int_{\Sigma}K_gd\vg=2\pi \chi(\Sigma)+2\pi\sum_{p\in \Sigma}^{}(\theta_0(p)-1)
	\end{align*}
	and this concludes the proof of the lemma.
\end{proof}

Now, we have the following \emph{a priori} estimate.

\begin{prop}
	Let $\Sigma$ be a closed Riemann surface $2<p<\infty$ and $n\geq 3$ fixed. There exists a universal constant $C_1=C_1(n,p,\Sigma)<\infty$ such that for all weak immersion $\phi\in \mathscr{E}({\Sigma,\R^n})\cap W^{2,p}(\Sigma,\R^n)$, we have 
	\begin{align*}
	\int_{\Sigma}|\h_0|^p_{WP}d\vg\leq e^{C_1(1+W(\phi))}\int_{\Sigma}|\H|^pd\vg.
	\end{align*}
\end{prop}
\begin{rem}
	For $p=2$, thanks of Gauss-Bonnet formula, we have
	\begin{align*}
	\int_{\Sigma}^{}|\h_0|^2_{WP}d\vg=\int_{\Sigma}^{}\left(|\H|^2-K_g\right)d\vg=\int_{\Sigma}^{}|\H|^2d\vg-2\pi\chi(\Sigma)
	\end{align*}
	so we do not need the local estimate to obtain the boundedness of
	\begin{align*}
	\sigma_k^2\int_{\Sigma^2}|\vec{\I}|_{g_k}^2d\mathrm{vol}_{g_k}\conv{k\rightarrow\infty}0.
	\end{align*}
\end{rem}
\begin{proof}
	By compactness of $\Sigma$ and Besicovitch covering theorem, we can write
	\begin{align*}
		\Sigma=\bigcup_{j=1}^mB_j
	\end{align*}
	such that each open $B_j$ does not intersect more than $N(\Sigma)$ elements of the collection $\ens{B_1,\cdots,B_m}$, and for all  $1\leq j\leq m$, 
	\begin{align*}
		\int_{B_j}|d\n|^2_gd\vg<\frac{8\pi}{3}.
	\end{align*}
	Now, up to shrinking the $B_j$, we can assume that each $B_j$ is included in a chart domain, so we assume from now on that $\phi:D^2\rightarrow \R^n$ is a conformal weak immersion such that
	\begin{align}\label{est}
		\int_{D^2}|\D\n|^2dx<\frac{8\pi}{3}.
	\end{align}
	In particular, thanks of a theorem of Hélein (\cite{helein}), there exists a moving frame $(\tilde{\e_1},\tilde{\e_2})$ such that
	\begin{align*}
		-\Delta\lambda=\p{x_1}\tilde{\e_1}\cdot\p{x_2}\tilde{\e_2}-\p{x_2}\tilde{\e_1}\cdot\p{x_1}\tilde{\e_2}
	\end{align*}
	and for some universal constant $\Gamma_0>0$
	\begin{align}\label{ineq}
		\np{\D\tilde{\e_1}}{2}{D^2}^2+\np{\D\tilde{\e_2}}{2}{D^2}^2\leq \Gamma_0\int_{D^2}|\D\n|^2dx<\frac{8\pi \Gamma_0}{3}.
	\end{align}
	Let $\nu\in W^{1,2}_0(D^2)$ be the solution of the equation
    \begin{align*}
    	-\Delta\nu=\p{x_1}\tilde{\e_1}\cdot\p{x_2}\tilde{\e_2}-\p{x_2}\tilde{\e_1}\cdot\p{x_1}\tilde{\e_2}.
    \end{align*}
    Then by Wente's inequality, $\nu\in W^{1,2}_0(D^2)\cap C^0(D^2)$ and 
    \begin{align}
    	\np{\D\nu}{2                                                                                                                                                                        }{D^2}+\np{\nu}{\infty}{D^2}\leq \left(\frac{1}{2\pi}+\frac{1}{2}\sqrt{\frac{3}{\pi}}\right)
    	\np{\D\tilde{\e_1}}{2}{D^2}\np{\D\tilde{\e_2}}{2}{D^2}\leq  6\Gamma_0
    \end{align}
    by \eqref{est} and \eqref{ineq}. As the function $\lambda-\nu$ is harmonic, there exists $\Gamma_1,\Gamma_2>0$ such that
    \begin{align}\label{const}
         \np{\lambda-\nu}{\infty}{D^2(0,1/2)}\leq \Gamma_1\np{\D(\lambda-\nu)}{2,\infty}{D^2}\leq 6\,\Gamma_0\Gamma_1+\Gamma_1\np{\D\lambda}{2,\infty}{D^2}\leq 6\,\Gamma_0\Gamma_1+\Gamma_2\left(1+W(\phi)\right)
    \end{align}
    Therefore, there exists $\bar{\lambda}\in \R$ and a universal constant $C>0$ (depending only on the conformal class of $g=\phi^{\ast}g_{\R^n}$, see \cite{lauriv1} for more informations about this subtle dependence) such that
    \begin{align*}
    	\np{\lambda-\bar{\lambda}}{\infty}{D^2(0,1/2)}\leq C\left(1+W(\phi)\right).
    \end{align*}
    Furthermore, notice that
    \begin{align*}
    	\Delta\phi=2e^{2\lambda}\H.
    \end{align*}
    Therefore, we have by Calder\'{o}n-Zygmund estimates
    \begin{align*}
    	\np{\p{z}^2\phi}{p}{D^2(0,1/4)}\leq C\np{\Delta\phi}{p}{D^2(0,1/2)}\leq 2C                                                                                                                                                                                                                                                                                                                                                                                                                                                                                          \left(\int_{D^2(0,1/2)}|\H|^pe^{2\lambda p}dx\right)^{\frac{1}{p}}.
    \end{align*}
    which implies
    \begin{align*}
    	&\int_{D^2(0,1/4)}|\h_0|^p_{WP}d\vg=\int_{D^2(0,1/4)}\left|2(\p{z}^2\phi)^N\right|^2e^{-2\lambda}|dz|^2\leq 2^p\int_{D^2(0,1/4)}|\p{z}^2\phi|^pe^{-2(p-1)\lambda}|dz|^2\\
    	&\leq 2^p\exp\left(2(p-1)\np{\lambda-\bar{\lambda}}{\infty}{D^2(0,1/2)}\right)e^{-2(p-1)\bar{\lambda}}\int_{D^2(0,1/2)}|\H|^pe^{2\lambda p}|dz|^2\\
    	&\leq 2^p\exp\left(4\np{\lambda-\bar{\lambda}}{\infty}{D^2(0,1/2)}\right)\int_{D^2(0,1/2)}|\H|^pe^{2\lambda}|dz|^2\\
    	&\leq 2^p\exp\left(4(p-1)\np{\lambda-\bar{\lambda}}{\infty}{D^2(0,1/2)}\right)\int_{D^2(0,1/2)}|\H|^pd\vg.
   \end{align*}
   Finally, by \eqref{const}, we see that there exists a universal constant $C=C(p)$ such that
   \begin{align*}
   	\int_{D^2(0,1/4)}|\h_0|^2_{WP}d\vg\leq e^{C(1+W(\phi))}\int_{D^2(0,1/2)}|\H|^pd\vg.
   \end{align*}
   Finally, by the Besicovitch covering theorem we obtain the final claim.
\end{proof}

\begin{rem}\label{energy}
	Now, consider the following relaxation of the area functional
	\begin{align*}
		A_{\sigma}(\phi)=\mathrm{Area}(\phi(\Sigma))+\sigma^2\int_{\Sigma}\left(1+|\H|^2\right)^2d\vg+\frac{1}{\log\left(\frac{1}{\sigma}\right)}\mathscr{O}(\phi)+\frac{1}{\log\log\log\left(\frac{1}{\sigma}\right)}W(\phi),
	\end{align*}
	 and the refined entropy condition for a critical point $\phi$ of $A_{\sigma}$
	\begin{align}\label{germain}
		\frac{d}{d\sigma}A_{\sigma}(\phi)\leq \frac{1}{\sigma\log\left(\frac{1}{\sigma}\right)\log\log\left(\frac{1}{\sigma}\right)\log\log\log\left(\frac{1}{\sigma}\right)\log\log\log\log\left(\frac{1}{\sigma}\right)}
	\end{align}
	Notice that we can always construct a sequence critical points satisfying this inequality together with the expected Morse index bound as long as the right-hand side of \eqref{germain} is replace by a positive function which does not belong to $L^1([0,1])$ (see \cite{index2}).
	
	The estimate \eqref{germain} implies that (up to choosing an Aubin gauge for $\Sigma=S^2$)
	\begin{align*}
		&\sigma^2\int_{\Sigma}\left(1+|\H|^2\right)d\vg\leq \frac{1}{\log\left(\frac{1}{\sigma}\right)\log\log\left(\frac{1}{\sigma}\right)}\\
		&\frac{1}{\log\left(\frac{1}{\sigma}\right)}\mathscr{O}(\phi)\leq \frac{1}{\log\log\left(\frac{1}{\sigma}\right)}\\
		&\frac{1}{\log\left(\frac{1}{\sigma}\right)}\int_{\Sigma}|d\alpha|_g^2d\vg\leq \frac{6}{\log\log\left(\frac{1}{\sigma}\right)}\\
		&\frac{1}{\log\log\log\left(\frac{1}{\sigma}\right)}W(\phi)\leq \frac{1}{\log\log\log\log\left(\frac{1}{\sigma}\right)}.
	\end{align*}
	In particular, we have as $\mathscr{A}(\w)\leq 2|d\w|_g|\vec{\I}|_g$
	\begin{align*}
		&|DW(\phi)(\w)|\leq W(\phi)^{\frac{1}{2}}\left(\int_{\Sigma}|\Delta_g^{\perp}\w|_g^2d\vg\right)^{\frac{1}{2}}+2\np{|d\w|_g}{\infty}{\Sigma}\int_{\Sigma}|\vec{\I}|_g||\H|d\vg\\
		&\leq W(\phi)^{\frac{1}{2}}\left(\int_{\Sigma}|\Delta_g^{\perp}\w|_g^2d\vg\right)^{\frac{1}{2}}+2\np{|d\w|_g}{\infty}{\Sigma}
	\end{align*}
	and
	\begin{align*}
		&\left|DF(\phi)(\w)\right|\leq 2\np{|d\w|_g}{\infty}{\Sigma}\int_{\Sigma}\left(1+|\H|^2\right)^2d\vg+2\left(\int_{\Sigma}\left(1+|\H|^2\right)^2d\vg\right)^{\frac{3}{4}}\left(\int_{\Sigma}|\Delta_g^{\perp}\w|^4d\vg\right)^{\frac{1}{4}}\\
		&+2\np{|d\w|_g}{\infty}{\Sigma}\left(\int_{\Sigma}|\vec{\I}|_g^4d\vg\right)^{\frac{1}{4}}\left(\int_{\Sigma}\left(1+|\H|^2\right)^2d\vg\right)^{\frac{3}{4}}\\
		&\leq 2\np{|d\w|_g}{\infty}{\Sigma}\int_{\Sigma}\left(1+|\H|^2\right)^2d\vg+2\left(\int_{\Sigma}\left(1+|\H|^2\right)^2d\vg\right)^{\frac{3}{4}}\left(\int_{\Sigma}|\Delta_g^{\perp}\w|^4d\vg\right)^{\frac{1}{4}}\\
		&+2\np{|d\w|_g}{\infty}{\Sigma}e^{\frac{C}{4}(1+W(\phi))}\int_{\Sigma}\left(1+|\H|^2\right)^2d\vg.
	\end{align*}
	Finally,
	\begin{align*}
		|D\mathscr{O}(\phi)(\w)|\leq 4\pi|\chi(\Sigma)|+3\int_{\Sigma}|d\alpha|_g^2d\vg+\left(\sqrt{W(\phi)}+\sqrt{W(\phi)-2\pi\chi(\Sigma)}\right)\left(\int_{\Sigma}|d\alpha|_g^2d\vg\right)^{\frac{1}{2}}
	\end{align*}
	Therefore, $A_{\sigma}$ satisfies the \textbf{Energy condition} of \cite{index2} (which says that the norm operator of $DA_{\sigma}$ stays bounded as long as long as $A_{\sigma}$ is bounded). Furthermore, observe that
	\begin{align*}
	&\sigma^2\left(\int_{\Sigma}|\vec{\I}|_g^4d\vg\right)^{\frac{1}{4}}\left(\int_{\Sigma}\left(1+|\H|^2\right)^2d\vg\right)^{\frac{3}{4}}\leq e^{\frac{C}{4}(1+W(\phi))}\sigma^2\int_{\Sigma}\left(1+|\H|^2\right)^2d\vg\\
	&\leq e^{\frac{C}{4}}\left(\log\log\left(\frac{1}{\sigma}\right)\right)^{\frac{C}{4}}\frac{1}{\log\left(\frac{1}{\sigma}\right)\log\log\left(\frac{1}{\sigma}\right)}
	=e^{\frac{C}{4}}\frac{\left(\log\log\left(\frac{1}{\sigma}\right)\right)^{\frac{C}{4}-1}}{\log\left(\frac{1}{\sigma}\right)}\conv{\sigma\rightarrow 0}0.
    \end{align*}
    Likewise, the second derivative satisfies the same estimate, so we see that the proof of the lower semi-continuity of the index for minimal surfaces obtained by the viscosity method in \cite{hierarchies} could be adapted to this relaxation of the area, if the hard analysis for $\sigma\rightarrow 0$ could be carried in a similar fashion. 
    
    Finally, one should notice that for immersions with valued into a closed manifold the second derivative of the relaxation term $A_{\sigma}-A_{0}$ would only change by curvature terms which are of $0$-th order at the PDE level, so they would not change the general bounds obtained previously.
\end{rem}

\subsection{The second derivative of the Willmore energy and viscous approximations}

We first recall that in \cite{eversion} the following functional is introduced for all $\sigma>0$
\begin{align*}
W_{\sigma}(\phi)=W(\phi)+\sigma^2\int_{\Sigma}\left(1+|\H|^2\right)^2d\vg+\frac{1}{\log\left(\frac{1}{\sigma}\right)}\mathscr{O}(\phi).
\end{align*}
To understand why we might have different definitions for the index, we recall the formulas (see \cite{indexS3})
\begin{align*}
&2\D_{\frac{d}{dt}}^\perp \H_{g_t}
=\Delta^{\perp}_g\w-\frac{1}{2}\sum_{i,j=1}^2\left(\s{\D_{\e_i}\w}{\e_j}+\s{\D_{\e_j}\w}{\e_i}\right)\vec{\I}(\e_i,\e_j)=\Delta_g^{\perp}\w+\mathscr{A}(\w),\\
&2\D_{\frac{d}{dt}}^\perp\D_{\frac{d}{dt}}^\perp \vec{H}_{g_t}=
\sum_{i,j=1}^2\bigg\{\left(-2\s{\D_{\e_i}\w}{\D_{\e_j}\w}+4\s{d\phi}{d\w}_g\left(\s{\D_{\e_i}\w}{\e_j}+\s{\D_{\e_j}\w}{\e_i}\right)\right)\vec{\I}(\e_i,\e_j)\\
&-2\left(\s{\D_{\e_i}\w}{\e_j}+\s{\D_{\e_j}\w}{\e_i}\right)\left((\D_{\e_i}\D_{\e_j}-\D_{(\D_{\e_i}\e_j)^\top})\w\right)^\perp\bigg\}\\
&-\sum_{i=1}^22\left(4\s{\D_{\e_1}\w}{\e_1}\s{\D_{\e_2}\w}{\e_2}-\left(\s{\D_{\e_1}\w}{\e_2}+\s{\D_{\e_1}\w}{\e_2}\right)^2\right)\vec{\I}(\e_i,\e_i)\\
&-\sum_{i,j=1}^2\s{(\D_{\e_i}\D_{\e_i}-\D_{(\D_{\e_i}\e_i)^\top})\w}{\e_j}\D_{\e_j}\w\bigg)^\perp\\
&=-2\,\vec{\I}\res_g \left(d\w\totimes d\w\right)+4\s{d\phi}{d\w}_g\,\vec{\I}\res_g \left(d\phi\totimes d\w+d\w\totimes d\phi\right)-4\left(\s{d\phi}{d\w}_g^2-16|\partial\phi\totimes \partial\w|_{WP}^2\right)\H\\
&-\s{\Delta_g\w}{d\phi}\res_g d\w-2(\D\,d\w)^{\perp}\res_g \left(d\phi\totimes d\w+d\w\totimes d\phi\right).
\end{align*}
where $\mathscr{A}(\w)$ is the Simons operator. In particular, we obtain
\begin{align*}
&\frac{d}{dt}|\H_{g_t}|^2_{|t=0}=2\s{\D_{\frac{d}{dt}}^{\perp}\H}{\H}=\s{\Delta_g^{\perp}\w+\mathscr{A}(\w)}{\H}\\
&\frac{d^2}{dt^2}|\H_{g_t}|^2_{|t=0}=2\s{\D_{\frac{d}{dt}}^{\perp}\D_{\frac{d}{dt}}^{\perp}\H}{\H}+2|\D_{\frac{d}{dt}}^{\perp}\H|^2\\
&=-2\bs{\s{\H}{\vec{\I}}}{d\w\totimes d\w}_g+4\s{d\phi}{d\w}_g\bs{\s{\H}{\vec{\I}}}{d\phi\totimes d\w+d\w\totimes d\phi}_g-4\left(\s{d\phi}{d\w}_g^2-16|\partial\phi\totimes \partial\w|_{WP}^2\right)|\H|^2\\
&-\bs{\s{\Delta_g\w}{d\phi}}{\s{\H}{d\w}}_g-2\bs{\s{\H}{(\D d\w)^{\perp}}}{d\phi\totimes d\w+d\w\totimes d\phi}_g+\frac{1}{2}\left|\Delta^{\perp}_g\w+\mathscr{A}(\w)\right|^2
\end{align*}
Therefore, we obtain
\begin{align}\label{index0}
DW(\phi)(\w)=\int_{\Sigma}\s{\Delta_g^{\perp}\w+\mathscr{A}(\w)}{\H}d\vg+\int_{\Sigma}|\H|^2\s{d\phi}{d\w}_gd\vg
\end{align}
Furthermore, as $\mathscr{A}(\w)\leq 2|d\w|_g|\vec{\I}|_g$, we have
\begin{align}\label{ineq0}
|DW(\phi)(\w)|\leq W(\phi)^{\frac{1}{2}}\left(\int_{\Sigma}|\Delta_g^{\perp}\w|^2d\vg\right)^{\frac{1}{2}}+\left(6\,W(\phi)-2\pi\chi(\Sigma)\right)\np{|d\w|_g}{\infty}{\Sigma}.
\end{align}
As $\H\in L^2(\Sigma,d\vg)$, we see that by Cauchy-Schwarz inequality that \eqref{index0} is well-defined. Actually, we see that \eqref{index0} is well-defined if and only if $\s{\Delta_g^{\perp}\w}{\H}\in L^1(\Sigma,d\vg)$.
Now, we compute
\begin{align}\label{d2w}
&D^2W(\phi)(\w,\w)=\int_{\Sigma}\frac{d^2}{dt^2}|\H_{g_t}|^2_{|t=0}d\vg+2\int_{\Sigma}\frac{d}{dt}|\H_{g_t}|^2\frac{d}{dt}\left(d\mathrm{vol}_{g_t}\right)_{|t=0}+\int_{\Sigma}|\H|^2\frac{d}{dt}\left(d\mathrm{vol}_{g_t}\right)_{|t=0}\nonumber\\
&=\int_{\Sigma}\bigg\{-2\bs{\s{\H}{\vec{\I}}}{d\w\totimes d\w}_g+4\s{d\phi}{d\w}_g\bs{\s{\H}{\vec{\I}}}{d\phi\totimes d\w+d\w\totimes d\phi}_g\nonumber\\
&-4\left(\s{d\phi}{d\w}_g^2-16|\partial\phi\totimes \partial\w|_{WP}^2\right)|\H|^2
-\bs{\s{\Delta_g\w}{d\phi}}{\s{\H}{d\w}}_g-2\bs{\s{\H}{(\D d\w)^{\perp}}}{d\phi\totimes d\w+d\w\totimes d\phi}_g\nonumber\\
&+\frac{1}{2}\left|\Delta^{\perp}_g\w+\mathscr{A}(\w)\right|^2+2\s{\Delta_g^{\perp}\w+\mathscr{A}(\w)}{\H}\s{d\phi}{d\w}_g+|\H|^2\left(|d\w|_g^2-16|\partial\phi\totimes \partial\w|_{WP}^2\right) \bigg\}d\vg.
\end{align}
We will show that $D^2W(\phi)(\w,\w)$ is well-defined if and only $\w$ is such that 
\begin{align}\label{key2}
\np{|d\w|_g}{\infty}{\Sigma}+\left(\int_{\Sigma}|\Delta_g^{\perp}\w|^2d\vg\right)^{\frac{1}{2}}<\infty,
\end{align}
a condition which is not satisfied for all variations $\w\in W^{2,2}_{\phi}\cap W^{1,\infty}(\Sigma,\R^n)$ if $\phi$ is a branched immersion. Indeed, at a branch point of multiplicity $\theta_0\geq 2$, the metric admits the a development (in some complex chart $(z,U)$ such that $z:U\rightarrow D^2\subset \C$ and $z(p)=0$ around the branch point) $e^{2\lambda}=\gamma|z|^{2\theta_0-2}\left(1+O(|z|)\right)$ for some $\gamma>0$ and
\eqref{key2} is equivalent to 
\begin{align*}
\frac{|\D\w|}{|z|^{\theta_0-1}}\in L^{\infty}(D^2),\quad \int_{D^2}\frac{|\Delta^{\perp}\w|^2}{|z|^{2\theta_0-2}}|dz|^2<\infty.
\end{align*}
In particular, if $\w$ is a smooth variation, then (by the analysis of \cite{beriviere} or \cite{classification}) $\w$ admits the following expansion for some $\vec{A}\in \C^n$
\begin{align*}
\w=\w(0)+\Re\left(\vec{A}_0z^{\theta_0}\right)+O(|z|^{\theta_0+1}).
\end{align*}
In particular, there are no conditions if the branch point is of multiplicity $1$.

Denoting
\begin{align*}
F(\phi)(\w,\w)=\int_{\Sigma}\left(1+|\H|^2\right)^2d\vg,
\end{align*}
we also have
\begin{align*}
DF(\phi)(\w,\w)=2\int_{\Sigma}\s{\H}{\Delta_g^{\perp}\w+\mathscr{A}(\w)}\left(1+|\H|^2\right)d\vg+\int_{\Sigma}\left(1+|\H|^2\right)^2\s{d\phi}{d\w}_gd\vg.
\end{align*}
Then, we compute (where the derivatives are meant to be taken at $t=0$)
\begin{align*}
&D^2F(\phi)(\w,\w)=\int_{\Sigma}\left(2\frac{d^2}{dt^2}|\H_{g_t}|^2\left(1+|\H|^2\right)+2\left(\frac{d}{dt}|\H_{g_t}|^2\right)^2\right)d\vg\\
&+4\int_{\Sigma}\frac{d}{dt}|\H_{g_t}|^2\left(1+|\H|^2\right)\frac{d}{dt}\left(d\mathrm{vol}_{g_t}\right)+\int_{\Sigma}\left(1+|\H|^2\right)^2\frac{d^2}{dt^2}\left(d\mathrm{vol}_{g_t}\right)\\
&=2\int_{\Sigma}\bigg\{-2\bs{\s{\H}{\vec{\I}}}{d\w\totimes d\w}_g+4\s{d\phi}{d\w}_g\bs{\s{\H}{\vec{\I}}}{d\phi\totimes d\w+d\w\totimes d\phi}_g\\
&-4\left(\s{d\phi}{d\w}_g^2-16|\partial\phi\totimes \partial\w|_{WP}^2\right)|\H|^2
-\bs{\s{\Delta_g\w}{d\phi}}{\s{\H}{d\w}}_g-2\bs{\s{\H}{(\D d\w)^{\perp}}}{d\phi\totimes d\w+d\w\totimes d\phi}_g\\
&+\frac{1}{2}\left|\Delta^{\perp}_g\w+\mathscr{A}(\w)\right|^2+2\s{d\phi}{d\w}_g\s{\H}{\Delta_g^{\perp}\w+\mathscr{A}(\w)}+|\H|^2\left(|d\w|_g^2-16|\partial\phi\totimes \partial\w|_{WP}^2\right) \bigg\}\left(1+|\H|^2\right)d\vg\\
&+2\int_{\Sigma}\s{\Delta_g^{\perp}\w+\mathscr{A}(\w)}{\H}^2d\vg+4\int_{\Sigma}\s{d\phi}{d\w}_g\s{\H}{\Delta_{g}^{\perp}\w+\mathscr{A}(\w)}\left(1+|\H|^2\right)d\vg\\
&+\int_{\Sigma}\left(|d\w|_g^2-16|\partial\phi\totimes \partial\w|_{WP}^2\right)\left(1+|\H|^2\right)^2d\vg
\end{align*}
Finally, we have
\begin{align}\label{ofu}
&D^2W_{\sigma}(\phi)(\w,\w)=D^2W(\phi)(\w,\w)+\sigma^2D^2F(\phi)(\w,\w)+\frac{1}{\log\left(\frac{1}{\sigma}\right)}D^2\mathscr{O}(\phi)(\w,\w)\nonumber\\
&=\int_{\Sigma}\bigg\{-2\bs{\s{\H}{\vec{\I}}}{d\w\totimes d\w}_g+4\s{d\phi}{d\w}_g\bs{\s{\H}{\vec{\I}}}{d\phi\totimes d\w+d\w\totimes d\phi}_g\nonumber\\
&-4\left(\s{d\phi}{d\w}_g^2-16|\partial\phi\totimes \partial\w|_{WP}^2\right)|\H|^2
-\bs{\s{\Delta_g\w}{d\phi}}{\s{\H}{d\w}}_g-2\bs{\s{\H}{(\D d\w)^{\perp}}}{d\phi\totimes d\w+d\w\totimes d\phi}_g\nonumber\\
&+\frac{1}{2}\left|\Delta^{\perp}_g\w+\mathscr{A}(\w)\right|^2+2\s{\Delta_g^{\perp}\w+\mathscr{A}(\w)}{\H}\s{d\phi}{d\w}_g+|\H|^2\left(|d\w|_g^2-16|\partial\phi\totimes \partial\w|_{WP}^2\right) \bigg\}d\vg\nonumber\\
&+\sigma^2\bigg\{2\int_{\Sigma}\bigg\{-2\bs{\s{\H}{\vec{\I}}}{d\w\totimes d\w}_g+4\s{d\phi}{d\w}_g\bs{\s{\H}{\vec{\I}}}{d\phi\totimes d\w+d\w\totimes d\phi}_g\nonumber\\
&-4\left(\s{d\phi}{d\w}_g^2-16|\partial\phi\totimes \partial\w|_{WP}^2\right)|\H|^2
-\bs{\s{\Delta_g\w}{d\phi}}{\s{\H}{d\w}}_g-2\bs{\s{\H}{(\D d\w)^{\perp}}}{d\phi\totimes d\w+d\w\totimes d\phi}_g\nonumber\\
&+\frac{1}{2}\left|\Delta^{\perp}_g\w+\mathscr{A}(\w)\right|^2+2\s{d\phi}{d\w}_g\s{\H}{\Delta_g^{\perp}\w+\mathscr{A}(\w)}+|\H|^2\left(|d\w|_g^2-16|\partial\phi\totimes \partial\w|_{WP}^2\right) \bigg\}\left(1+|\H|^2\right)d\vg\nonumber\\
&+2\int_{\Sigma}\s{\Delta_g^{\perp}\w+\mathscr{A}(\w)}{\H}^2d\vg+4\int_{\Sigma}\s{d\phi}{d\w}_g\s{\H}{\Delta_{g}^{\perp}\w+\mathscr{A}(\w)}\left(1+|\H|^2\right)d\vg\nonumber\\
&+\int_{\Sigma}\left(|d\w|_g^2-16|\partial\phi\totimes \partial\w|_{WP}^2\right)\left(1+|\H|^2\right)^2d\vg\bigg\}\nonumber\\
&+\frac{1}{\log\left(\frac{1}{\sigma}\right)}\bigg\{\int_{\Sigma}\s{d\alpha\otimes d\alpha}{d\w\totimes d\w}_gd\vg-\int_{\Sigma}\s{d\alpha\otimes d\alpha}{d\phi\totimes d\w+d\w\totimes d\phi}_g\s{d\phi}{d\w}_gd\vg\nonumber\\
&-\frac{1}{2}\int_{\Sigma}|d\alpha|_g^2\left(|d\w|_g^2-2\s{d\phi}{d\w}_g^2-16|\partial\phi\totimes\partial\w|_{WP}^2\right)d\vg
+\frac{K_{g_0}}{2}\int_{\Sigma}\left(|d\w|_g^2-16|\partial\phi\totimes\partial\w|_{WP}^2\right)d\mathrm{vol}_{g_0}\nonumber\\
&+\frac{1}{2}\int_{\Sigma}\s{d\phi\totimes d\w+d\w\totimes d\phi}{d\alpha\otimes d\alpha_0'+d\alpha_0'\otimes d\alpha}_gd\vg-\int_{\Sigma}\s{d\alpha}{d\alpha_0'}_g\s{d\phi}{d\w}_gd\vg\nonumber\\
&-2\int_{\Sigma}d\alpha_0'\wedge \Im\left(\s{\H}{\partial\w}-g^{-1}\otimes\left(\h_0\totimes\bar{\partial}\w\right)\right)-K_{g_0}\int_{\Sigma}\alpha_0'\s{d\phi}{d\w}d\mathrm{vol}_{g_0}\nonumber\\
&-2\int_{\Sigma}d\alpha\wedge\,\Im\bigg(\s{\Delta_{g}^{\perp}\w+4\,\Re\left(g^{-2}\otimes\left(\bar{\partial\phi}\totimes\bar{\partial}\w\right)\otimes \h_0\right)}{\partial\w}
-\s{d\phi}{d\w}_g\left(\s{\H}{\partial\w}-g^{-1}\otimes\left(\h_0\totimes\bar{\partial}\w\right)\right)\nonumber\\
&-4\,g^{-1}\otimes\left(\partial\phi\totimes\partial\w\right)\otimes\s{\H}{\bar{\partial}\w}\bigg)-\frac{1}{2}\int_{\Sigma}|\D^{\perp}\w|_g^2\,\Delta_{g}\alpha\,d\mathrm{vol}_{g}
+\frac{K_{g_0}}{2}\left(\frac{1}{\int_{\Sigma}d\vg}\int_{\Sigma}\s{d\phi}{d\w}_gd\vg\right)^2\nonumber\\
&-\frac{K_{g_0}}{2}\frac{1}{\int_{\Sigma}d\vg}\int_{\Sigma}\left(|d\w|_g^2-16|\partial\phi\totimes\partial\w|_{WP}^2\right)d\vg \bigg\}
\end{align}
The goal of the next sections is to give a reasonable definition of the index though the second derivative, as this is an index that one can hope to compute explicitly (see \cite{indexS3} for a first result in this direction), contrary to a $C^0$ index which seems elusive to us.

\subsection{On the definition of Morse index of branched Willmore immersions}

The vector spaces of admissible variations of $\phi$, denoted by $\mathrm{Var}_l(\phi)$, are defined for $l=0,1,2$ by
\begin{align*}
	\mathrm{Var}_l(\phi)=W^{2,2}_{\phi}\cap W^{1,\infty}(\Sigma,\R^n)\,\cap\,&\Big\{\w=\left(\frac{d}{dt}\phi_t\right)_{|t=0}\;\, \text{for some}\;\, C^{2}\;\, \text{family}\;\, \{\phi_t\}_{t\in I}\in \mathscr{E}(\Sigma,\R^n)\\
	&\text{for some fixed}\;\, 0<\alpha< 1,\;\,  \text{such that}\;\, \phi_0=\phi\;\, \text{and}\;\, t\mapsto W(\phi_t)\;\,\text{is}\;\, C^l\Big\}.
\end{align*}
Notice that imposing $\{\phi_t\}_{t\in I}$ to be $C^2$ could be relaxed to $C^1$ or $C^0$ (respectively for $\mathrm{Var}_0$ and $\mathrm{Var}_1$), but as variations may be taken as $\phi_t=\phi+t\w$ for some $\w\in W^{2,2}_{\phi}\cap W^{1,\infty}(\Sigma,\R^n)$, the condition is not restrictive. 

To have a differential notion of the index for a functional as the Willmore functions defined on the space of branch Sobolev immersions, which cannot be equipped with a Banach space structure, we must restrict to variations $\{\phi_t\}_{t\in I}$ such that the map $I\rightarrow \R, t\mapsto W(\phi_t)$ is $C^2$. Therefore, we define the restricted notion of index as follows.

\begin{defi}
	Let $\Sigma$ be a closed Riemann surface and let $\phi:\Sigma\rightarrow \R^n$ be a branched Willmore immersion. Then the index of $\phi$, denoted by $\mathrm{Ind}_{\mathscr{W}}(\phi)$ is defined as the dimension of the maximal sub-vector space $V\subset \mathrm{Var}_2(\phi)$ on which the second variation $D^2\mathscr{W}(\phi)$ is negative definite. Likewise we define the $W$ index denoted by $\mathrm{Ind}_W(\phi)$ where $\mathscr{W}$ is replaced by $W$.
\end{defi}

\begin{rem}
	In general, we could define the index as the dimension of the maximal sub-vector space $V\subset \mathrm{Var}_0(\phi)$, such that for all $\w\in V$, there exists a path $\{\phi\}_{t\in I}:\Sigma\rightarrow \R^n$ of branched immersions such that $\mathscr{W}(\phi_t)<\mathscr{W}(\phi)$ for all $t\in I\setminus \ens{0}$. However, it seems difficult to estimate this index, and this is why we need a differential index.
\end{rem}

The next lemma shows that that the the weakest possible notion of index preserves the branch points.

\begin{lemme}\label{gaussbonnet}
	Let $\phi:\Sigma\rightarrow \R^n$ be a  branched Willmore surface, $\epsilon>0$, $I=(-\epsilon,\epsilon)$ and $\{\phi_t\}_{t\in I}$ be family of branched immersions of finite total curvature such that $\phi_0=\phi$,  and assume that the map $I:t\mapsto {W}(\phi_t)$ is continuous and
	\begin{align*}
		\mathscr{W}(\phi_t)<\mathscr{W}(\phi)\;\,\text{for all}\;\, t\in I\setminus \ens{0}.
	\end{align*}
	 Then for $|t|$ small enough, the map $\phi_t:\Sigma\rightarrow \R^n$ has branch points located exactly at the branch points of $\phi$ and with the same multiplicity. Furthermore, if $\ens{\phi_t}_{t\in I}$ is differentiable, and
	 \begin{align*}
	 	w=\frac{d}{dt}\left(\phi_t\right)_{|t=0},
	 \end{align*}
	 then we have the estimate $|d\w|_g\in L^{\infty}(\Sigma)$.
\end{lemme}

\begin{proof}
	Without loss of generality, we can assume that $\phi=\phi+t\w$ for some $\w\in W^{2,2}_{\phi}\cap W^{1,\infty}(\Sigma,\R^n)$. Let $p\in \Sigma$ be a branch point of  $\phi$ and $z$ a complex chart at $p$ such that $z(p)=0$. Thanks of Proposition \ref{branch}, there exists $\theta_0=\theta_0(p)\in \N\setminus\ens{0}$, and $\vec{A}_0\in \C^n\setminus\ens{0}$ (normalised such that $2|\vec{A}_0|^2=1$) such that
	\begin{align}\label{dev00}
	    &\p{z}\phi=\vec{A}_0z^{\theta_0-1}+O(|z|^{\theta_0})\nonumber\\
	    &e^{2\lambda}=2|\p{z}\phi|^2=|z|^{2\theta_0-2}\left(1+O(|z|)\right),
	\end{align}
	As the normal $\n:\Sigma\rightarrow \mathscr{G}_{n-2}(\R^n)$ admits the following Taylor type expansion for all thanks of the Corollary $4.24$ of \cite{classification}:
	\begin{align}\label{ndev0}
	\n=\Re\left(2i\ast \bar{\vec{A}_0}\wedge \vec{A}_0+\sum_{k,l,p}{\vec{A}_{k,l,p}}z^k\z^l\log^p|z|\right)+O(|z|^{\theta_0+1}\log^{p_0}|z|),
	\end{align}
	where $1\leq k+l\leq \theta_0$ for all the \emph{finitely many} $(k,l,p)\in \Z\times \Z\times\N$ such that $\vec{A}_{k,l,p}\neq 0$ and $1\leq k+l\leq \theta_0$ (we  take $\vec{A}_{k,l,p}$ to be $0$ outside of this range, and we also know that for all fixed $k,l$, there exists only finitely many $p\in \N$ such that $\vec{A}_{k,l,p}\neq 0$).
	Furthermore, up to shrinking the domain of $z$, we can find a local trivialisation $\n_1,\cdots,\n_{n-2}:\Sigma\rightarrow \R^n$ such that $|\n_j|=1$ for all $1\leq j\leq n$ and 
	\begin{align*}
		\n=\n_1\wedge\cdots\wedge \n_{n-2}.
	\end{align*}
	In particular, a normal variation $\w\in \mathscr{E}_{\phi}(\Sigma,\R^n)$ is given locally by
	\begin{align*}
		\w=\sum_{j=1}^{n-2}w_j\,\n_j,
	\end{align*}
	for some functions $w_j:W^{2,2}\cap W^{1,\infty}(\Sigma,\R)$. In particular, by standard regularisation, we can assume that $w_j\in C^{\infty}(\Sigma,\R)$. Therefore, there exists $\alpha\in \N^{\ast}$ and $\beta\in \N$  and a measurable function $\vec{T}$ such that
	\begin{align}\label{devw}
		&\p{z}\w=\vec{T}(z)+o(|z|^{\alpha-1}\log^{\beta}|z|),\quad C^{-1}|z|^{\alpha-1}\log^{\beta}|z|\leq \vec{T}(z)|\leq C|z|^{\alpha-1}\log^{\beta}|z|
	\end{align}
	so we obtain if $\alpha>\theta_0$ the estimate $\p{z}\w=O(|z|^{\theta_0})$, so
	\begin{align*}
		\p{z}\phi_t=\p{z}\phi+t\p{z}\w=\vec{A}_0z^{\theta_0-1}+O(|z|^{\theta_0}),
	\end{align*}
	so for all $t\in I$, $\phi_t:\Sigma\rightarrow \R^n$ has a branch point at $p$ of order $\theta_0$. Now, thanks of Proposition \ref{branch}, there exists $m(t)\in \N$ and $\gamma_t>0$ such that (up to the composition of a local conformal diffeomorphism) 
	\begin{align}\label{beta1}
	e^{2\lambda_t}=\sqrt{\det(g_t)}=\gamma_t|z|^{2m(t)-2}\left(1+o(1)\right),
	\end{align}
	If $\alpha\leq \theta_0-1$, we compute
	\begin{align}\label{exp}
		\frac{1}{4}\mathrm{det}(g_t)&=|\p{z}\phi|^4+4t|\p{z}\phi|^2\,\Re\left(\s{\p{z}\phi}{\p{\z}\w}\right)+t^2\left(4\,\Re\left(\s{\p{z}\phi}{\p{\z}\w}\right)^2+2|\p{z}\phi|^2|\p{z}\w|^2-4|\s{\p{z}\phi}{\p{z}\w}|^2\right)\nonumber\\
		&+4\,t^3\left(|\p{z}\w|^2\Re\left(\s{\p{z}\phi}{\p{\z}\w}\right)-\Re\left(\s{\p{z}\phi}{\p{z}\w}\s{\p{\z}\w}{\p{\z}\w}\right)\right)+t^4\left(|\p{z}\w|^4-|\s{\p{z}\w}{\p{z}\w}|^2\right)\nonumber\\
		&=\frac{1}{4}\gamma_t^2|z|^{4m(t)-4}\left(1+o(1)\right).
	\end{align}
	Furthermore, thanks of this expansion \eqref{exp}, \eqref{dev00} and \eqref{ndev0} (also recall that $\phi$ enjoys the same type of expansion as $\n$ in \eqref{ndev0}), there exists a sequence of $\C$-valued polynomials ${P}_{k,l,p}(t)$ of degree at most most $4$ and some coefficients $a_1,a_2,a_3,a_4\in \C$ and $p_0\in \N$ such that
	\begin{align*}
		\mathrm{det}(g_t)&=\Re\left(\sum_{m=0}^{4\theta_0-3}\sum_{p\geq 0}\sum_{k+l=m}^{}{P}_{k,l,p}(t)z^k\z^l\log^p|z|\right)+\left(1+a_1t+a_2t^2+a_3t^3+a_4t^4\right)|z|^{4\theta_0-4}\\
		&+\Re\left(\sum_{p\geq 1}^{}\sum_{k+l=4\theta_0-4}^{}{P}_{k,l,p}(t)z^k\z^l\log^p|z|\right)+O(|z|^{4\theta_0-5}\log^{p_0}|z|)\\
		&=\gamma_t^2|z|^{4m(t)-4}\left(1+o(1)\right),
	\end{align*}
	where $k,l\in \Z$ and the polynomials ${P}_{k,l,p}$ are almost all zero, that is, all but finitely many. In particular, if $|t|$ is small enough such that $|a_1t+a_2t^2+a_3t^3+a_4t^4|<1/2$, we must have the inequality $m(t)\leq \theta_0$.
	
	Finally, we deduce that for all $t\in I\setminus\ens{0}$ small enough, $\phi_t$ has a branch point at $p$ of constant order $\theta_0(p,t)\leq \theta_0(p)$ independent of $t$. In particular, if $\phi_t$ cannot have branch points at a point $p\in \Sigma$ where $\phi$ is unbranched.

	 Now, assume by contradiction that $p\in \Sigma$ is such that for all $|t|$ sufficiently small, $\theta_0(p,t)<\theta_0(p)$. By the Gauss-Bonnet theorem for branched immersions with finite total curvature (see Lemma \ref{gb}), we obtain
	\begin{align}\label{local1}
		\int_{\Sigma}K_{g_t}d\mathrm{vol}_{g_t}=2\pi \chi(\Sigma)+2\pi\sum_{j=1}^{m}\left(\theta_0(p_j,t )-1\right)\leq 2\pi \chi(\Sigma)+2\pi\sum_{j=1}^{m}(\theta_0(p_j)-1)-2\pi=\int_{\Sigma}K_gd\vg-2\pi.
	\end{align}
	In particular, as $t\mapsto W(\phi_t)$ is continuous, for all $0<\epsilon<2\pi$,  there exists  $\delta>0$ such that for all $|t|<\delta$
	\begin{align}\label{local2}
		\left|\int_{\Sigma}|\H_{g_t}|^2d\mathrm{vol}_{g_t}-\int_{\Sigma}|\H|^2d\vg\right|<\epsilon
	\end{align}
	so we obtain by \eqref{local1} and \eqref{local2} for all $|t|<\delta$
	\begin{align*}
		\mathscr{W}(\phi_t)=\int_{\Sigma}|\H_{g_t}|^2d\mathrm{vol}_{g_t}-\int_{\Sigma}K_{g_t}d\mathrm{vol}_{g_t}\geq\int_{\Sigma}|\H_g|^2d\vg-\int_{\Sigma}K_gd\vg+2\pi-\epsilon=\mathscr{W}(\phi)+2\pi-\epsilon>\mathscr{W}(\phi),
	\end{align*}
	and in particular, $\{\phi_t\}_{t\in I}$ cannot be a negative variation and we deduce the Lemma from this stronger observation.
\end{proof}

\begin{rem}
	This argument would not give any information if we had chosen the $W$-index, that is the index with respect to 
	\begin{align*}
		W(\phi)=\int_{\Sigma}|\H_g|^2d\mathrm{vol}_g,
	\end{align*}
	but as we want to define an index conformally invariant we chose the $\mathscr{W}$-index. The conformal invariance is needed as we can only hope to compute the $\mathscr{W}$-index of Willmore spheres, looking at the explicit expression in \cite{indexS3}.
\end{rem}

\begin{lemme}\label{normaldelta}
	Let $\Sigma$ be a closed Riemann surfaces and let $\phi:\Sigma\rightarrow \R^n$ be a smooth immersion. Then we have for all $\w\in \mathrm{Var}(\phi)$
	\begin{align*}
		\Delta_g^\perp\w=(\Delta_g\w)^{\perp}+4\,g^{-1}\otimes\Re\bigg\{
		g^{-1}\otimes\left(\bar{\partial}\phi\totimes\bar{\partial}\w\right)\otimes\h_0+\left(\partial\phi\totimes\bar{\partial}\w\right)\H \bigg\}.
	\end{align*}
\end{lemme}
\begin{proof}
	Taking conformal coordinates, if $e^{2\lambda}$ is the conformal parameter of $\phi$, this is easy to see (\cite{classification}) that
	\begin{align*}
		\Delta_g^\perp=2\left(\D_{\p{z}}^\perp\D_{\p{\z}}^\perp+\D_{\p{\z}}^\perp\D_{\p{z}}^\perp\right).
  	\end{align*}
	As for all vector $\vec{X}\in \R^n$, we have
	\begin{align*}
		\vec{X}^{\top}=-2\,g^{-1}\otimes \s{\vec{X}}{\bar{\partial}\phi}\otimes\partial\phi-2\,g^{-1}\otimes\s{\vec{X}}{\partial\phi}\otimes\bar{\partial}\phi,
	\end{align*}
	we obtain (writing $\e_z=\p{z}\phi$)
	\begin{align}\label{tang1}
		\D_{\p{z}}^\perp\D_{\p{\z}}^\perp\vec{w}&=\D_{\p{z}}^\perp\left(\D_{\p{\z}}\w-\D_{\p{\z}}^{\top}\w\right)=\D_{\p{z}}^\perp\D_{\p{\z}}\w+\D_{\p{z}}^\perp\Big(2e^{-2\lambda}\s{\D_{\p{\z}}\vec{w}}{\e_{\z}}\e_z+2e^{-2\lambda}\s{\D_{\p{\z}}\w}{\e_z}\e_{\z}\Big)\nonumber\\
		&=\D_{\p{z}}^\perp\D_{\p{\z}}\w+2e^{-2\lambda}\s{\D_{\p{\z}}\w}{\e_{\z}}\D_{\p{z}}^\perp\e_z+2e^{-2\lambda}\s{\D_{\p{\z}}\w}{\e_z}\D_{\p{z}}^\perp\e_{\z}\nonumber\\
		&=\D_{\p{z}}^\perp\D_{\p{\z}}\w+\s{\D_{\p{\z}}\w}{\e_{\z}}\H_0+\s{\D_{\p{\z}}\w}{\e_z}\H,
	\end{align}
	where we used the notations
	\begin{align*}
		&\H_0=2\,e^{-2\lambda}\vec{\I}(\e_z,\e_z)=2\,e^{-2\lambda}\D_{\p{z}}^\perp\e_z\\
		&\H=2\,e^{-2\lambda}\vec{\I}(\e_z,\e_{\z})=2\,e^{-2\lambda}\D_{\p{z}}^\perp\e_{\z}.
	\end{align*}
	Likewise, we have
	\begin{align*}
		\D_{\p{\z}}^\perp\D_{\p{z}}^\perp\w=\D_{\p{\z}}^\perp\D_{\p{z}}\w+\s{\D_{\p{z}}\w}{\e_{\z}}\H+\s{\D_{\p{z}}\w}{\e_{z}}\bar{\H_0}.
	\end{align*}
	As
	\begin{align}\label{tang2}
		\D_{\p{z}}^\perp\D_{\p{\z}}\w=\D_{\p{z}}\D_{\p{\z}}\w-(\D_{\p{z}}\D_{\p{\z}}\w)^{\top}=\D_{\p{z}}\D_{\p{\z}}\w+2e^{-2\lambda}\s{\D_{\p{z}}\D_{\p{\z}}\w}{\e_{\z}}\e_z+2e^{-2\lambda}\s{\D_{\p{z}}\D_{\p{\z}}\w}{\e_z}\e_{\z},
	\end{align}
	and
	\begin{align}\label{tang3}
		\Delta_g\w=2e^{-2\lambda}\left(\D_{\p{z}}\D_{\p{\z}}+\D_{\p{\z}}\D_{\p{z}}\right)\w,
	\end{align}
	we obtain by \eqref{tang1}, \eqref{tang2} and \eqref{tang3} the identity
	\begin{align*}
		\Delta_g^\perp\w&=\Delta_g\w+2e^{-2\lambda}\s{\Delta_g\w}{\e_{\z}}\e_z+2e^{-2\lambda}\s{\Delta_g\w}{\e_z}\e_{\z}+2e^{-2\lambda}\Big(\s{\D_{\p{\z}}\w}{\e_z}+\s{\D_{\p{z}}\w}{\e_{\z}}\e_z\Big)\H\\
		&+2e^{-2\lambda}\s{\D_{\p{\z}}\w}{\e_{\z}}\H_0+2e^{-2\lambda}\s{\D_{\p{z}}\w}{\e_z}\bar{\H_0}\\
		&=\Delta_g\w+4\,g^{-1}\otimes\Re\bigg\{(\Delta_g\w\totimes\partial\phi)\otimes\bar{\partial}\phi+g^{-1}\otimes\left(\bar{\partial}\w\totimes\bar{\partial}\phi\right)\otimes\h_0+\left(\partial\w\totimes\bar{\partial}\phi\right)\H \bigg\}\\
		&=\left(\Delta_g\w\right)^\perp+4\,g^{-1}\otimes\Re\bigg\{g^{-1}\otimes\left(\bar{\partial}\w\totimes\bar{\partial}\phi\right)\otimes\h_0+\left(\partial\w\totimes\bar{\partial}\phi\right)\H \bigg\}
	\end{align*}
	and this concludes the proof of the lemma.
\end{proof}

\begin{lemme}\label{est2}
    Let $\phi\in \mathscr{E}_{\phi}(D^2,\R^n)$ be a weak branched immersion with finite total curvature assume that $\phi$ has a unique branch point at $z=0$ of multiplicity $\theta_0\geq 2$ and let $\w\in \mathscr{E}_{\phi}(\Sigma,\R^n)$ be an admissible normal variation. Then for all $\alpha<\theta_0-1$, there exists a  constant $C=C(\phi)>0$ depending only on $\phi$ such that 
    \begin{align*}
    	\int_{D^2(0,\frac{1}{2})}\frac{|\D^2\w|^2}{|z|^{2\alpha}}|dz|^2\leq C\int_{D^2}|\Delta_g^{\perp}\w|^2d\vg+C\left(\frac{\alpha^2}{(\theta_0-1)-\alpha}+\int_{D^2}|\h_0|^2_{WP}d\vg\right)\np{|d\w|_g}{\infty}{D^2}.
    \end{align*}
    where $\D^2$ is the normal gradient defined in \eqref{opd2}.
\end{lemme}
\begin{proof}
	We first compute thanks of the Ricci identity
	\begin{align*}
		\D_{\e_{\z}}^{\perp}\D_{\e_z}^{\perp}\D^{\perp}_{\e_z}\w&=\D_{\e_{z}}^{\perp}\D_{\e_{\z}}^{\perp}\D_{\e_z}^\perp \w+R^{\perp}(\e_{\z},\e_z)\D_{\e_z}^{\perp}\w\\
		&=\D_{\e_{z}}^{\perp}\D_{\e_{\z}}^{\perp}\D_{\e_z}^{\perp}\w+\vec{\I}(\bar{\D}_{\e_{\z}}\D_{\e_{z}}^{\perp}\w,\e_{z})-\vec{\I}(\e_{\z},\bar{\D}_{\e_z}\D_{\e_z}^{\perp}\w).
	\end{align*}
	Furthermore, we have directly
	\begin{align*}
		\bar{\D}_{\e_{\z}}\D_{\e_z}^{\perp}\w&=2e^{-2\lambda}\s{\D_{\e_{\z}}\D_{\e_z}^{\perp}\w}{\e_{z}}\e_z+2e^{-2\lambda}\s{\D_{\e_{\z}}\D_{\e_z}^{\perp}\w}{\e_z}\e_{\z}\\
		&=-\s{\D_{\e_z}^{\perp}\w}{\bar{\H_0}}\e_z-\s{\D_{\e_{\z}}^{\perp}\w}{\H}\e_{\z}\\
		\bar{\D}_{\e_z}^{\perp}\D_{\e_z}^{\perp}\w&=-\s{\D_{\e_z}^{\perp}\w}{\H}\e_z-\s{\D_{\e_{z}}^{\perp}\w}{\H_0}\e_{\z}.
	\end{align*}
	Therefore, we finally obtain
	\begin{align*}
		\D_{\e_{\z}}^{\perp}\D_{\e_z}^{\perp}\D_{\e_z}^{\perp}\w&=\D_{\e_z}^{\perp}\D_{\e_{\z}}^{\perp}\D_{\e_z}^{\perp}\w+\frac{1}{2}e^{-2\lambda}\s{\D_{\e_z}^{\perp}\w}{\bar{\H_0}}\H_0+\frac{1}{2}e^{-2\lambda}\s{\D_{\e_z}^{\perp}\w}{\H}\H-\frac{1}{2}\s{\D_{\e_z}^{\perp}\w}{\H}\H-\frac{1}{2}\s{\D_{\e_z}^{\perp}\w}{\H_0}\bar{\H_0}\\
		&=\D_{\e_z}^{\perp}\D_{\e_{\z}}^{\perp}\D_{\e_z}^{\perp}\w+\frac{1}{2}e^{-2\lambda}\s{\D_{\e_z}^{\perp}\w}{\bar{\H_0}}\H_0-\frac{1}{2}\s{\D_{\e_z}^{\perp}\w}{\H_0}\bar{\H_0}
	\end{align*}
	Furthermore, we have
	\begin{align}\label{d1}
		\Delta^{\perp}=2\left(\D_{\e_z}^{\perp}\D_{\e_{\z}}^{\perp}+\D_{\e_{\z}}^{\perp}\D_{\e_z}^{\perp}\right)
	\end{align}
	while by Ricci's identity
	\begin{align}\label{d2}
		&\D_{\e_z}^{\perp}\D_{\e_{\z}}^{\perp}\w=\D_{\e_{\z}}^{\perp}\D_{\e_z}^{\perp}\w+R^{\perp}(\e_z,\e_{\z})\w\nonumber\\
		&=\D_{\e_{\z}}^{\perp}\D_{\e_z}^{\perp}\w+\vec{\I}(\bar{\D}_{\e_z}\w,\e_{\z})-\vec{\I}(\e_z,\bar{\D}_{\e_{\z}}\w)\nonumber\\
		&=\D_{\e_{\z}}^{\perp}\D_{\e_z}^{\perp}\w+2e^{-2\lambda}\left(\s{\D_{\e_z}\w}{\e_{\z}}\vec{\I}(\e_z,\e_{\z})+\s{\D_{\e_z}\w}{\e_z}\vec{\I}(\e_{\z},\e_{\z})-\s{\D_{\e_{\z}}\w}{\e_{\z}}\vec{\I}(\e_z,\e_z)-\s{\D_{\e_{\z}}\w}{\e_z}\vec{\I}(\e_z,\e_{\z})\right)\nonumber\\
		&=\D_{\e_{\z}}^{\perp}\D_{\e_z}^{\perp}\w+\s{\D_{\e_z}\w}{\e_{z}}\bar{\H_0}-\s{\D_{\e_{\z}}\w}{\e_{\z}}\vec{H}_0
	\end{align}
	as by normality
	\begin{align*}
		2e^{-2\lambda}\s{\D_{\e_z}\w}{\e_{\z}}\vec{\I}(\e_z,\e_{\z})=-\frac{e^{2\lambda}}{2}\s{\w}{\H}\H=2e^{-2\lambda}\s{\D_{\e_{\z}}\w}{\e_z}\vec{\I}(\e_z,\e_{\z}).
	\end{align*}
	Therefore, we obtain from \eqref{d1} and \eqref{d2} the identity
	\begin{align}\label{d3}
		\Delta^{\perp}\w=4\D_{\e_{\z}}^{\perp}\D_{\e_z}^{\perp}\w+2\left(\s{\D_{\e_z}\w}{\e_z}\bar{\H_0}-\s{\D_{\e_{\z}}\w}{\e_{\z}}\H_0\right)=4\D_{\e_{\z}}^{\perp}\D_{\e_z}^{\perp}\w+4i\,\Im\left(\s{\D_{\e_z}\w}{\e_z}\bar{\H_0}\right).
	\end{align} 
	In particular, observe as $\Delta^{\perp}\w$ is real and $\s{\D_{\e_z}\w}{\e_z}\bar{\H_0}-\s{\D_{\e_{\z}}\w}{\e_{\z}}\H_0$ is imaginary pure
	\begin{align*}
		\left|\D_{\e_{\z}}^{\perp}\D_{\e_z}^{\perp}\w\right|^{2}&=\left|\frac{1}{4}\Delta^{\perp}\w+\frac{1}{2}\left(\s{\D_{\e_{\z}}\w}{\e_{\z}}\H_0-\s{\D_{\e_z}\w}{\e_z}\bar{\H_0}\right)\right|^2\\
		&=\frac{1}{16}|\Delta^{\perp}\w|^2+\frac{1}{2}\left|\s{\D_{\e_{z}}\w}{\e_{z}}\right|^2|{\H_0}|^2-\frac{1}{2}\Re\left(\s{\D_{\e_{\z}}\w}{\e_{\z}}^2\s{\H_0}{\H_0}\right)
	\end{align*}
	Now, we have by integrating by parts
	\begin{align*}
		&\int_{D^2}\frac{|\D_{\e_z}^{\perp}\D_{\e_z}^{\perp}\w|^2}{|z|^{2\alpha}}\rho^2|dz|^2=\int_{D^2} {\s{\D_{\e_z}^{\perp}\D_{\e_z}^{\perp}\w}{\D_{\e_{\z}}^{\perp}\D_{\e_{\z}}^{\perp}\w}}\frac{\rho^2|dz|^2}{|z|^{2\alpha}}\\
		&=-\int_{D^2}\s{\D_{\e_{\z}}^{\perp}\D_{\e_z}^{\perp}\D_{\e_z}^{\perp}\w}{\D_{\e_{\z}}^{\perp}\w}\frac{\rho^2|dz|^2}{|z|^{2\alpha}}+\alpha\int_{D^2}z\s{\D_{\e_z}^{\perp}\D_{\e_z}^{\perp}\w}{\D_{\e_{\z}}^{\perp}\w}\frac{\rho^2|dz|^2}{|z|^{2\alpha+2}}-2\int_{D^2}{\s{\D_{\e_z}^{\perp}\D_{\e_z}^{\perp}\w}{\D_{\e_{\z}}^{\perp}\w}}\frac{(\p{\z}\rho)\, \rho |dz|^2}{|z|^{2\alpha}}.
	\end{align*}
	By \eqref{d2}, we also have
	\begin{align*}
		&-\int_{D^2}\s{\D_{\e_{\z}}^{\perp}\D_{\e_z}^{\perp}\D_{\e_z}^{\perp}\w}{\D_{\e_{\z}}^{\perp}\w}\frac{\rho^2|dz|^2}{|z|^{2\alpha}}=-\int_{D^2}\s{\D_{\e_z}^{\perp}\D_{\e_{\z}}^{\perp}\D_{\e_z}^{\perp}\w}{\D_{\e_{\z}}^{\perp}\w}\frac{\rho^2|dz|^2}{|z|^{2\alpha}}\\
		&+\frac{1}{2}\int_{D^2}e^{-2\lambda}\s{\D_{\e_z}^{\perp}\w}{\bar{\H_0}}\s{\H_0}{\D_{\e_{\z}}^{\perp}\w}\frac{|dz|^2}{|z|^{2\alpha}}
		-\frac{1}{2}\int_{D^2}e^{-2\lambda}\s{\D_{\e_z}^{\perp}\w}{\H_0}\s{\D_{\e_{\z}}^{\perp}\w}{\bar{\H_0}}\frac{\rho^2|dz|^2}{|z|^{2\alpha}}\\
		&=\int_{D^2}\s{\D_{\e_{\z}}^{\perp}\D_{\e_z}^{\perp}\w}{\D_{\e_z}^{\perp}\D_{\e_{\z}}^{\perp}\w}\frac{\rho^2|dz|^2}{|z|^{2\alpha}}-\alpha\int_{D^2}\bar{z}\s{\D_{\e_{\z}}^{\perp}\D_{\e_z}^{\perp}\w}{\D_{\e_{\z}}^{\perp}\w}\frac{\rho^2|dz|^2}{|z|^{2\alpha+2}}+2\int_{D^2}\s{\D_{\e_{\z}}^{\perp}\D_{\e_z}^{\perp}\w}{\D_{\e_{\z}}^{\perp}\w}\frac{(\p{z}\rho)\,\rho|dz|^2}{|z|^{2\alpha}}\\
		&+\frac{1}{2}\int_{D^2}e^{-2\lambda}\s{\D_{\e_z}^{\perp}\w}{\bar{\H_0}}\s{\H_0}{\D_{\e_{\z}}^{\perp}\w}\frac{\rho^2|dz|^2}{|z|^{2\alpha}}
		-\frac{1}{2}\int_{D^2}e^{-2\lambda}\s{\D_{\e_z}^{\perp}\w}{\H_0}\s{\D_{\e_{\z}}^{\perp}\w}{\bar{\H_0}}\frac{\rho^2|dz|^2}{|z|^{2\alpha}}.
	\end{align*}
	Thanks of \eqref{d3}, we have
	\begin{align*}
		&\s{\D^{\perp}_{\e_{\z}}\D_{\e_z}^{\perp}\w}{\D_{\e_z}^{\perp}\D_{\e_{\z}}^{\perp}\w}=|\D_{\e_{\z}}^{\perp}\D_{\e_z}^{\perp}\w|^2+\s{\D_{\e_z}\w}{\e_z}\s{\bar{\H_0}}{\D_{\e_{\z}}^{\perp}\D_{\e_z}^{\perp}\w}-\s{\D_{\e_{\z}}\w}{\e_{\z}}\s{\H_0}{\D_{\e_{\z}}\D_{\e_z}^{\perp}\w}\\
		&=\left|\frac{1}{4}\Delta^{\perp}\w+\frac{1}{2}\left(\left(\s{\D_{\e_z}\w}{\e_z}\bar{\H_0}-\s{\D_{\e_{\z}}\w}{\e_{\z}}\H_0\right)\right)\right|^2
		+\s{\D_{\e_z}\w}{\e_z}\s{\bar{\H_0}}{\left(\s{\D_{\e_z}\w}{\e_z}\bar{\H_0}-\s{\D_{\e_{\z}}\w}{\e_{\z}}\H_0\right)}\\
		&-\s{\D_{\e_{\z}}\w}{\e_{\z}}\s{\H_0}{\left(\s{\D_{\e_z}\w}{\e_z}\bar{\H_0}-\s{\D_{\e_{\z}}\w}{\e_{\z}}\H_0\right)}\\
		&=\frac{1}{16}|\Delta^{\perp}\w|^2+\frac{1}{4}\left|\s{\D_{\e_z}\w}{\e_z}\bar{\H_0}-\s{\D_{\e_{\z}}\w}{\e_{\z}}\H_0\right|^2+2\,\Re\left(\s{\D_{\e_{\z}}\w}{\e_{\z}}\s{\H_0}{\H_0}\right)-2|\s{\D_{\e_z}\w}{\e_z}|^2|\H_0|^2\\
		&=\frac{1}{16}|\Delta^{\perp}\w|^2+\frac{3}{2}\,\Re\left(\s{\D_{\e_{\z}}\w}{\e_{\z}}\s{\H_0}{\H_0}\right)-\frac{3}{2}|\s{\D_{\e_z}\w}{\e_z}|^2|\H_0|^2.
	\end{align*}
	Finally, we obtain
	\begin{align}\label{cancelcod}
		&\int_{D^2}\frac{|\D_{\e_z}^{\perp}\D_{\e_{z}}^{\perp}\w|^2}{|z|^{2\alpha}}\rho^2|dz|^2=\frac{1}{16}\int_{D^2}\frac{|\Delta^{\perp}\w|^2}{|z|^{2\alpha}}\rho^2|dz|^2-2\int_{D^2}\s{\D_{\e_z}^{\perp}\D_{\e_z}^{\perp}\w}{\D_{\e_{\z}}^{\perp}\w}\frac{(\p{\z}\rho)\rho|dz|^2}{|z|^{2\alpha}}\nonumber\\
		&+2\int_{D^2}\s{\D_{\e_{\z}}^{\perp}\D_{\e_z}^{\perp}\w}{\D_{\e_{\z}}^{\perp}\w}\frac{(\p{z}\rho)\,\rho|dz|^2}{|z|^{2\alpha}}+\alpha\int_{D^2}z\s{\D_{\e_z}^{\perp}\D_{\e_z}^{\perp}\w}{\D_{\e_{\z}}^{\perp}\w}\frac{\rho^2|dz|^2}{|z|^{2\alpha+2}}\nonumber\\
		&-\alpha\int_{D^2}\z\s{\D_{\e_{\z}}^{\perp}\D_{\e_z}^{\perp}\w}{\D_{\e_{\z}}^{\perp}\w}\frac{\rho^2|dz|^2}{|z|^{2\alpha+2}}
		+\frac{1}{2}\int_{D^2}e^{-2\lambda}\left|\s{\H_0}{\D_{\e_{\z}}^{\perp}\w}\right|^2\frac{\rho^2|dz|^2}{|z|^{2\alpha}}
		-\frac{1}{2}\int_{D^2}e^{-2\lambda}\left|\s{\H_0}{\D_{\e_z}^{\perp}\w}\right|^2\frac{\rho^2|dz|^2}{|z|^{2\alpha}}\nonumber\\
		&+\frac{3}{2}\int_{D^2}\Re\left(\s{\D_{\e_{\z}}\w}{\e_{\z}}\s{\H_0}{\H_0}\right)\frac{\rho^2|dz|^2}{|z|^{2\alpha}}-\frac{3}{2}\int_{D^2}|\s{\D_{\e_z}\w}{\e_z}|^2|\H_0|^2\frac{\rho^2|dz|^2}{|z|^{2\alpha}}\nonumber\\
		&\leq \frac{1}{16}\int_{D^2}\frac{|\Delta^{\perp}\w|^2}{|z|^{2\alpha}}\rho^2|dz|^2-2\int_{D^2}\s{\D_{\e_z}^{\perp}\D_{\e_z}^{\perp}\w}{\D_{\e_{\z}}^{\perp}\w}\frac{(\p{\z}\rho)\rho|dz|^2}{|z|^{2\alpha}}
		+2\int_{D^2}\s{\D_{\e_{\z}}^{\perp}\D_{\e_z}^{\perp}\w}{\D_{\e_{\z}}^{\perp}\w}\frac{(\p{z}\rho)\,\rho|dz|^2}{|z|^{2\alpha}}\nonumber\\
		&+\alpha\int_{D^2}z\s{\D_{\e_z}^{\perp}\D_{\e_z}^{\perp}\w}{\D_{\e_{\z}}^{\perp}\w}\frac{\rho^2|dz|^2}{|z|^{2\alpha+2}}
		-\alpha\int_{D^2}\z\s{\D_{\e_{\z}}^{\perp}\D_{\e_z}^{\perp}\w}{\D_{\e_{\z}}^{\perp}\w}\frac{\rho^2|dz|^2}{|z|^{2\alpha+2}}
		+\frac{1}{2}\int_{D^2}e^{-2\lambda}\left|\s{\H_0}{\D_{\e_{\z}}^{\perp}\w}\right|^2\frac{\rho^2|dz|^2}{|z|^{2\alpha}}\nonumber\\
		&+\frac{3}{2}\int_{D^2}\Re\left(\s{\D_{\e_{\z}}\w}{\e_{\z}}\s{\H_0}{\H_0}\right)\frac{\rho^2|dz|^2}{|z|^{2\alpha}}\nonumber\\
		&=\frac{1}{16}\int_{D^2}\frac{|\Delta^{\perp}\w|^2}{|z|^{2\alpha}}\rho^2|dz|^2+\mathrm{(I)}+\mathrm{(II)}+\mathrm{(III)}+\mathrm{(IV)}+\mathrm{(V)}+\mathrm{(VI)}.
	\end{align}
	Now, we estimate as $|\D\rho(x)||x|\leq 2$ for $x\in \mathrm{supp}(\D\rho)$
	\begin{align}\label{z1}
		\mathrm{(I)}&=\left|2\int_{D^2}\s{\D_{\e_z}^{\perp}\D_{\e_z}^{\perp}\w}{\D_{\e_{\z}}^{\perp}\w}\frac{(\p{\z}\rho)\rho|dz|^2}{|z|^{2\alpha}}\right|\leq\epsilon \int_{D^2}|\D_{\e_z}^{\perp}\D_{\e_z}^{\perp}\w|^2\frac{\rho^2|dz|^2}{|z|^{2\alpha}}+\frac{1}{\epsilon}\int_{D^2}|\D_{\e_{\z}}^{\perp}\w|^2\frac{|\p{\z}\rho|^2|dz|^2}{|z|^{2\alpha}}\nonumber\\
		&\leq \epsilon\int_{D^2}|\D_{\e_z}^{\perp}\D_{\e_z}^{\perp}\w|^2\frac{\rho^2|dz|^2}{|z|^{2\alpha}}+\frac{1}{2\epsilon}\int_{D^2}\frac{|\D\w|^2}{|z|^{2\alpha+2}}|dz|^2.
	\end{align}
	Likewise, we have
	\begin{align}\label{z2}
		\mathrm{(II)}&=\left|2\int_{D^2}\s{\D_{\e_{\z}}^{\perp}\D_{\e_z}^{\perp}\w}{\D_{\e_{\z}}^{\perp}\w}\frac{(\p{z}\rho)\,\rho|dz|^2}{|z|^{2\alpha }}\right|\leq \epsilon\int_{D^2}|\D^{\perp}_{\e_{\z}}\D^{\perp}_{\e_z}\w|^2\frac{\rho^2|dz|^2}{|z|^{2\alpha}}+\frac{1}{2\epsilon}\int_{D^2}\frac{|\D\w|^2}{|z|^{2\alpha+2}}|dz|^2\nonumber\\
		&=\frac{\epsilon}{16}\int_{D^2}\frac{|\Delta^{\perp}\w|^2}{|z|^{2\alpha}}\rho^2|dz|^2+\frac{\epsilon}{2}\int_{D^2}|\s{\D_{\e_z}\w}{\e_z}|^2|\H_0|^2\frac{\rho^2|dz|^2}{|z|^{2\alpha}}-\frac{\epsilon}{2}\int_{D^2}\Re\left(\s{\D_{\e_{\z}}\w}{\e_{\z}}^2\s{\H_0}{\H_0}\right)\frac{\rho^2|dz|^2}{|z|^{2\alpha}}\nonumber\\
		&+\frac{1}{2\epsilon}\int_{D^2}\frac{|\D\w|^2}{|z|^{2\alpha+2}}|dz|^2\nonumber\\
		&\leq \frac{\epsilon}{16}\int_{D^2}\frac{|\Delta^{\perp}\w|^2}{|z|^{2\alpha}}\rho^2|dz|^2+\epsilon\int_{D^2}|\s{\D_{\e_z}\w}{\e_z}|^2|\H_0|^2\frac{\rho^2|dz|^2}{|z|^{2\alpha}}+\frac{1}{2\epsilon}\int_{D^2}\frac{|\D\w|^2}{|z|^{2\alpha+2}}|dz|^2
	\end{align}
	Now we estimate
	\begin{align}\label{z3}
	\mathrm{(III)}&=\left|\alpha\int_{D^2}z\s{\D_{\e_z}^{\perp}\D_{\e_z}^{\perp}\w}{\D_{\e_{\z}}^{\perp}\w}\frac{\rho^2|dz|^2}{|z|^{2\alpha+2}}\right|\leq \epsilon\int_{D^2}|\D_{\e_z}^{\perp}\D_{\e_z}^{\perp}\w|^2\frac{\rho^2|dz|^2}{|z|^{2\alpha}}+\frac{\alpha^2}{4\epsilon}\int_{D^2}\frac{|\D_{\e_{\z}}^{\perp}\w|^2}{|z|^{2\alpha+2}}|dz|^2\nonumber\\
	&\leq \epsilon\int_{D^2}|\D_{\e_z}^{\perp}\D_{\e_z}^{\perp}\w|^2\frac{\rho^2|dz|^2}{|z|^{2\alpha}}+\frac{\alpha^2}{8\epsilon}\int_{D^2}\frac{|\D\w|^2}{|z|^{2\alpha+2}}|dz|^2.
	\end{align}
	The next inequality is
	\begin{align}\label{z4}
	\mathrm{(IV)}&=\left|\alpha\int_{D^2}\z\s{\D_{\e_{\z}}^{\perp}\D_{\e_z}^{\perp}\w}{\D_{\e_{\z}}^{\perp}\w}\frac{\rho^2|dz|^2}{|z|^{2\alpha+2}}\right|\leq \frac{\epsilon}{2}\int_{D^2}\frac{|\Delta^{\perp}\w|^2}{|z|^{2\alpha}}\rho^2|dz|^2+\frac{\epsilon}{2}\int_{D^2}|\s{\D_{\e_z}\w}{\e_z}|^2|\H_0|^2\frac{\rho^2|dz|^2}{|z|^{2\alpha}}\nonumber\\
	&-\frac{\epsilon}{2}\int_{D^2}\Re\left(\s{\D_{\e_{\z}}\w}{\e_{\z}}^2\s{\H_0}{\H_0}\right)\frac{\rho^2|dz|^2}{|z|^{2\alpha}}
	+\frac{\alpha^2}{8\epsilon}\int_{D^2}\frac{|\D\w|^2}{|z|^{2\alpha+2}}|dz|^2.
	\end{align}
	Now, as $z=0$ is a branched point of order $\theta_0$, there exists $\gamma>0$ such that
	\begin{align*}
		g=e^{2\lambda}|dz|^2=\gamma|z|^{2\theta_0-2}(1+O(|z|))|dz|^2,
	\end{align*}
	In particular, we have for some constant $C_0=C_0(\phi)$ the estimate
	\begin{align}\label{non-explicit}
		\frac{1}{C_0(\phi)}|z|^{2\theta_0-2}\leq e^{2\lambda}\leq C_0(\phi)|z|^{2\theta_0-2}
	\end{align}
	and as $\alpha<\theta_0-1$, we obtain
	\begin{align}\label{z5}
	\mathrm{(V)}&=\frac{1}{2}\int_{D^2}e^{-2\lambda}\left|\s{\H_0}{\D_{\e_z}\w}\right|^2\frac{\rho^2|dz|^2}{|z|^{2\alpha}}\leq \frac{1}{4}\np{|d\w|_g}{\infty}{D^2}^2\int_{D^2}|\h_0|^2_{WP}\frac{e^{4\lambda}|dz|^2}{|z|^{2\alpha}}\nonumber\\
	&\leq \frac{C_0}{4}\np{|d\w|_g}{\infty}{D^2}^2\int_{D^2}|\h_0|^2_{WP}d\vg.
	\end{align}
	We also estimate directly
    \begin{align}\label{z6}
	|\mathrm{(VI)}|&=\frac{3}{2} \left|\int_{D^2}\Re\left(\s{\D_{\e_{\z}}\w}{\e_{\z}}\s{\H_0}{\H_0}\right)\frac{\rho^2|dz|^2}{|z|^{2\alpha}}\right|\leq\frac{3}{2}\int_{D^2}|\s{\D_{\e_z}\w}{\e_z}|^2|\H_0|^2\frac{\rho^2|dz|^2}{|z|^{2\alpha}}\nonumber\\
	&\leq  \frac{3C_0}{8}\np{|d\w|_g}{\infty}{D^2}^2\int_{D^2}|\h_0|^2_{WP}d\vg.
	\end{align}
    Finally, by \eqref{cancelcod}, we obtain the estimate
    \begin{align*}
    	\int_{D^2}\frac{|\D_{\e_z}^{\perp}\D_{\e_z}^{\perp}\w|^2}{|z|^{2\alpha}}\rho^2|dz|^2&\leq 2\epsilon\int_{D^2}\frac{|\D_{\e_z}^{\perp}\D_{\e_z}^{\perp}\w|^2}{|z|^{2\alpha}}\rho^2|dz|^2+\frac{1}{16}\left(1+2\epsilon\right)\int_{D^2}\frac{|\Delta^{\perp}\w|^2}{|z|^{2\alpha}}\rho^2|dz|^2\\
    	&+\frac{1}{\epsilon}\left(1+\frac{\alpha^2}{4}\right)\int_{D^2}\frac{|\D\w|^2}{|z|^{2\alpha+2}}|dz|^2+C_0\np{|d\w|_g}{\infty}{D^2}^2\int_{D^2}|\h_0|^2_{WP}d\vg
    \end{align*}
    Now, choosing $\epsilon=1/4$ and marking $\rho$ converge to the indicator function of $D^2(0,1/2)$, we obtain
    \begin{align}\label{key}
    	\int_{D^2(0,1/2)}\frac{|\D_{\e_z}^{\perp}\D_{\e_z}^{\perp}\w|^2}{|z|^{2\alpha}}|dz|^2&\leq \frac{3}{16}\int_{D^2}\frac{|\Delta^{\perp}\w|^2}{|z|^{2\alpha}}+2(4+\alpha^2)\int_{D^2}\frac{|\D\w|^2}{|z|^{2\alpha+2}}|dz|^2\nonumber\\
    	&+2C_0\np{|d\w|_g}{\infty}{D^2}^2\int_{D^2}|\h_0|^2_{WP}d\vg.  
    \end{align}
    Now, we have
    \begin{align*}
    	\D^2_{\e_z,\e_z}=\D^{\perp}_{\e_z}\D_{\e_z}^{\perp}-\D_{\bar{\D}_{\e_z}\e_z}^{\perp}=\D_{\e_z}^{\perp}\D_{\e_z}^{\perp}-2(\p{z}\lambda)\D_{\e_z}^{\perp}
    \end{align*}
    Now, recall the expansion from \cite{classification}
    \begin{align*}
    	&\lambda=(\theta_0-1)\log|z|+\frac{1}{2}\log(\gamma)+O(|z|)\\
    	&2(\p{z}\lambda)=\frac{(\theta_0-1)}{z}+O(1),
    \end{align*}
    where the rest $2(\p{z}\lambda)-(\theta_0-1)/z$ is en element of $C^{0,\alpha}$ for all $0<\alpha<1$. In particular, there exists $C_1=C_1(\phi)>0$ such that
    \begin{align*}
    	4|\D\lambda|^2=8|\p{z}\lambda|^2\leq \frac{C_1}{|z|^2}
    \end{align*}
     This implies that
    \begin{align*}
    	|\D^{2}_{\e_z,\e_z}\w|^2&= |\D_{\e_z}^{\perp}\D_{\e_z}^{\perp}\w|^2-4|\p{z}\lambda|^2\Re\left(\s{\D_{\e_z}^{\perp}\D_{\e_z}^{\perp}\w}{\D_{\e_{\z}}^{\perp}\w}\right)+4|\p{z}\lambda|^2|\D_{\e_z}^{\perp}\w|^2\\
    	&\leq 2|\D_{\e_z}^{\perp}\D_{\e_z}^{\perp}\w|^2+\frac{C_1}{|z|^2}|\D_{\e_z}^{\perp}\w|^2.
    \end{align*}
    Now, recall that by the Ricci equation
    \begin{align*}
    	\left(\D^2_{\e_z,\e_{\z}}-\D^2_{\e_{\z},\e_z}\right)\w&=R^{\perp}(\e_z,\e_{\z})\w=\vec{\I}(\bar{\D}_{\e_z}\w,\e_{\z})-\vec{\I}(\e_z,\bar{\D}_{\e_{\z}}\w)\\
    	&=\frac{1}{2}\left(\s{\D_{\e_{z}}\w}{\e_{z}}\bar{\H_0}-\s{\D_{\e_{\z}}\w}{\e_{\z}}\H_0\right)=i\,\Im\left(\s{\D_{\e_z}\w}{\e_z}\bar{\H_0}\right).
    \end{align*}
    Furthermore, we have by the parallelogram identity
    \begin{align*}
    	&|\D^2\w|^2=\sum_{i,j=1}^{2}|\D^2_{\e_i,\e_j}\w|^2
    	=|\D^{2}_{\e_z+\e_{\z},\e_z+\e_{\z}}\w|^2+|\D^2_{\e_z+\e_{\z},i(\e_z-\e_{\z})}\w|^2+|\D^2_{i(\e_z-\e_{\z}),\e_z+\e_{\z}}\w|^2\nonumber\\
    	&+|\D^2_{\e_z+\e_{\z},i(\e_z-\e_{\z})}\w|^2=\left|2\,\Re(\D^{2}_{\e_z,\e_z}\w)+\frac{1}{2}\Delta^{\perp}\w\right|^2+\left|-2\,\Im(\D_{\e_z,\e_z}^2\w)-i(\D^2_{\e_z,\e_{\z}}-\D^2_{\e_{\z},\e_z})\w\right|^2\nonumber\\
    	&+\left|-2\,\Im\left(\D_{\e_z,\e_z}^2\w\right)+i(\D^2_{\e_z,\e_{\z}}-\D^2_{\e_{\z},\e_z})\w\right|^2
    	+\left|2\,\Re(\D^2_{\e_z,\e_z}\w)+\frac{1}{2}\Delta^{\perp}\w\right|^2\\
    	&=4\left|\Re\left(\D_{\e_z,\e_z}^2\w\right)\right|^2+4\left|\Im(\D_{\e_z,\e_z}^2)\right|+\frac{1}{2}|\Delta^{\perp}\w|^2+2\left|(\D_{\e_z,\e_z}^{2}-\D_{\e_{\z},\e_z})\w\right|^2\\
    	&=4|\D_{\e_z,\e_z}^2\w|^2+\frac{1}{2}|\Delta^{\perp}\w|^2+2\left|\Im\left(\s{\D_{\e_z}\w}{\e_z}\bar{\H_0}\right)\right|^2.
    \end{align*}
    Then by \eqref{non-explicit}
    \begin{align*}
    	2\left|\Im\left(\s{\D_{\e_z}\w}{\e_z}\bar{\H_0}\right)\right|^2\leq \frac{1}{2}e^{4\lambda}|d\w|_g^2|\H_0|^2=\frac{1}{2}e^{4\lambda}|d\w|_g^2\,|\h_0|_{WP}^2.
    \end{align*}
    Finally, we have
    \begin{align*}
    	|\D^2\w|^2\leq 8|\D^{\perp}_{\e_z}\D^{\perp}_{\e_z}\w|^2+\frac{1}{2}|\Delta^{\perp}\w|^2+\frac{C_1}{|z|^2}|\D_{\e_z}^{\perp}\w|^2+\frac{1}{2}e^{4\lambda}|d\w|_g^2|\h_0|^2_{WP},
    \end{align*}
    so we obtain the estimate
    \begin{align*}
    	\int_{D^2(0,1/2)}\frac{|\D^2\w|^2}{|z|^{2\alpha}}|dz|^2&\leq 2\int_{D^2}\frac{|\Delta^{\perp}\w|^2}{|z|^{2\alpha}}|dz|^2+\left(16(4+\alpha^2)+C_1\right)\int_{D^2}\frac{|\D\w|^2}{|z|^{2\alpha+2}}|dz|^2\\
    	&+(2C_0+1)\np{|d\w|_g}{\infty}{D^2}^2\int_{D^2}|\h_0|^2_{WP}\,d\vg.
    \end{align*}
    Finally, as $\alpha<\theta_0-1$, we obtain by \eqref{non-explicit}
    \begin{align*}
    	\int_{D^2}\frac{|\D\w|^2}{|z|^{2\alpha+2}}|dz|^2\leq C_0\np{|d\w|_g}{\infty}{D^2}^2\int_{D^2}|z|^{2(\theta_0-\alpha)-4}|dz|^2=\frac{\pi C_0}{(\theta_0-1)-\alpha}\np{|d\w|_g}{\infty}{D^2}^2,
    \end{align*}
    while
    \begin{align*}
    	\int_{D^2}\frac{|\Delta^{\perp}\w|^2}{|z|^{2\alpha}}|dz|^2\leq C_0\int_{D^2}|\Delta_g^{\perp}\w|^2d\vg.
    \end{align*}
    Gathering the different estimates, we obtain
    \begin{align*}
    	\int_{D^2(0,1/2)}\frac{|\D^2\w|^2}{|z|^{2\alpha}}|dz|^2&\leq 2C_0\int_{D^2}|\Delta_g^{\perp}\w|^2d\vg+\bigg(\frac{\pi C_0(16(4+\alpha^2)+C_1)}{(\theta_0-1)-\alpha}\nonumber\\
    	&+(2C_0+1)\int_{D^2}|\h_0|^2_{Wp}d\vg\bigg)\np{|d\w|_g}{\infty}{D^2}^2\\
    	&\leq C_2\left(\int_{D^2}|\Delta_g^{\perp}\w|^2d\vg+\left(\frac{\alpha^2}{(\theta_0-1)-\alpha}+\int_{D^2}|\h_0|^2_{WP}d\vg\right)\np{|d\w|_g}{\infty}{D^2}^2\right)
    \end{align*}
\end{proof}
\begin{rem}
	One can obtain better estimates in codimension $1$, as the four last terms of \eqref{cancelcod} cancel. 
\end{rem}
\begin{prop}\label{variations}
	Let $\Sigma$ be a closed Riemann surface and let $\phi:\Sigma\rightarrow\R^n$ be a branched Willmore immersion and $g=\phi^{\ast}g_{\,\R^n}$. Then the set of admissible variations $\w\in W^{2,2}_{\phi}\cap W^{1,\infty}(\Sigma,\R^n)$ coincides with
	\begin{align}\label{tangentspace}
	\mathrm{Var}_2(\phi)=W^{2,2}_{\phi}\cap W^{1,\infty}(\Sigma,\R^n)\cap\ens{\w: \mathscr{L}_{g}\w\in L^2(\Sigma,d\vg), |d\w|_{g}\in L^{\infty}(\Sigma,d\mathrm{vol}_{g_0})},
	\end{align}
	where 
	\begin{align*}
	\mathscr{L}_g\,\w=\Delta_g^\perp\w+\mathscr{A}(\w).
	\end{align*}
	is the Jacobi operator of $\phi$, and $g_0$ is a conformal metric to $\phi$ of constant Gauss curvature, and $\mathscr{A}(\w)$ is the Simons operator.
\end{prop}
\begin{rem}
	As a branch point of multiplicity $\theta_0\geq 1$, by parametrising $\phi$ locally by the unit disk $D^2$ such that
	\begin{align*}
	e^{2\lambda}=2|\p{z}\phi(z)|^{2}=|z|^{2\theta_0-2}\left(1+O(|z|)\right)
	\end{align*}
	The conditions of \eqref{tangentspace} read
	\begin{align*}
	\int_{D^2}\frac{|\Delta^{\perp}\w|^2}{|z|^{2\theta_0-2}}|dz|^2<\infty,\;\,\text{and}\;\,\frac{\D\w}{|z|^{\theta_0-1}}\in L^{\infty}(D^2)
	\end{align*}
	if $\Delta$ is the flat Laplacian. In particular, if $\theta_0=1$, there is not additional conditions, as expected. This implies if $\w$ is smooth or $\dfrac{\Delta\w}{|z|^{\theta_0-1}}\in L^{\infty}(D^2)$ (\emph{e.g.} by Proposition $C.2$ of \cite{beriviere}) that for some $\vec{C}_1\in \C^n$
	\begin{align*}
		\w(z)=\w(0)+\Re\left(\vec{C}_1z^{\theta_0}\right)+O(|z|^{\theta_0+1}\log^2|z|).
	\end{align*}
	thanks of the appendix of \cite{classification} (see also the appendix of \cite{beriviere} for a weaker form). This also shows that there are no conditions for $\theta_0=1$, as expected.
\end{rem}

\begin{proof}
	As $|d\w|_g\in L^{\infty}(\Sigma)$, we deduce that $D^2W(\phi)(\w,\w)$ is finite if and only
	\begin{align}\label{conv1}
	\left|\int_{\Sigma}\left(-\bs{\s{\Delta_g\w}{d\phi}}{\s{\H}{d\w}}_g-2\bs{\s{\H}{(\D d\w)^{\perp}}}{d\phi\totimes d\w+d\w\totimes d\phi}_g
	+\frac{1}{2}\left|\Delta^{\perp}_g\w+\mathscr{A}(\w)\right|^2\right)d\vg\right|<\infty
	\end{align}
	As away from the branch points the conditions are trivially satisfied (recall that the immersion is real analytic outside of branch points), it suffices to restrict to the case of immersions $\phi:D^2\rightarrow\R^n$ from the unit disk such that for some $\theta_0\geq 2$ (for $\theta_0=1$, there is nothing to prove), and for some $\vec{A}_0\in \C^n\setminus\ens{0}$, we have
	\begin{align*}
	\phi(z)=\Re\left(\vec{A}_0z^{\theta_0}\right)+O\left(|z|^{\theta_0+1-\epsilon}\right)
	\end{align*}
	for all $\epsilon>0$ (see \cite{beriviere} and \cite{classification} for more details about this expansion). In particular, up to a suitable normalisation of $\vec{A}_0$, we obtain
	\begin{align}\label{devconf}
	e^{2\lambda}=|z|^{2(\theta_0-1)}\left(1+O(|z|)\right)\quad (\text{the}\;\,\epsilon\;\,\text{disappears}).
	\end{align}
	Now, let $\{\phi_t\}_{|t|<\epsilon}:D^2\rightarrow\R^n$ a sequence of branched immersions such that $\phi_0=\phi$, we must determine under which minimal conditions the map
	\begin{align*}
	t\mapsto \mathscr{W}(\phi_t)
	\end{align*}
	is an element of $C^2((-\epsilon,\epsilon),\R)$.
	
	Now, let $(U,z)$ a complex chart around a branch point $p\in \Sigma$ of multiplicity $\theta_0\geq 2$ of $\phi$ such that $z(p)=0$, and $z:U\rightarrow D^2\subset \C$ and consider the restriction $\phi:D^2\rightarrow \R^n$ and let  $\n:D^2\rightarrow \mathscr{G}_{n-2}(\R^n)$ is the unit normal. Up to shrinking $D^2$ we have locally $\n=\n_1\wedge \cdots \wedge \n_{n-2}$ for some $\n_1,\cdots,\n_{n-2}:D^2\rightarrow \R$ satisfying $|\n_j|=1$, and 
	\begin{align*}
	\w=\sum_{j=1}^{n}w_j\n_j,
	\end{align*}
	and we can assume that $w_j\in C^{\infty}(\Sigma)$ by standard regularisation. Furthermore, as $\w$ is an admissible variation, we have
	\begin{align*}
	|d\w|_g=e^{-\lambda}|\D\w|\in L^{\infty}(D^2),
	\end{align*}
	and there exists  $\gamma>0$, and integer $\alpha\leq \theta_0-1$ and $\vec{C}_1\in \C^n$ such that 
	\begin{align*}
	&e^{2\lambda}=\gamma|z|^{2\theta_0-2}\left(1+O(|z|)\right),\\
	&\H=O\left(\frac{\vec{C}_1}{z^\alpha}\right)+O(|z|^{1-\alpha})
	\end{align*}
	so 
	\begin{align}\label{est9}
	\frac{|\D\w|}{|z|^{\theta_0-1}}\in L^{\infty}(D^2).
	\end{align}
	Now, recall as  $\n=e^{-2\lambda}\,2i\,\ast\left(\p{\z}\phi\wedge \p{z}\phi\right)$ and as $\phi$ for all $k\in \N$ there exists $\alpha_k\in \N$  such that (see \cite{classification})
	\begin{align*}
	\D^k\phi=O(|z|^{\theta_0-k}\log^{\alpha_k}|z|)
	\end{align*} 
	we have 
	\begin{align}\label{est10}
	\D\n=O(\log|z|),\qquad
	\D^2\n=O\left(\frac{1}{|z|}\right).
	\end{align}
	If $\theta_0=1$, then we easily see that all integrals are finite. Otherwise if $\theta\geq 2$, then we obtain by \eqref{est9} and \eqref{est10}
	\begin{align*}
	\D^2\w=O(|z|^{\theta_0-2})
	\end{align*}
	In particular, we have
	\begin{align*}
	|\D^2\w||\H|=O\left(\frac{1}{|z|}\right),
	\end{align*}
	so that for some constant $C>0$
	\begin{align}\label{conv2}
	&\left|\int_{D^2}\left(-\bs{\s{\Delta_g\w}{d\phi}}{\s{\H}{d\w}}_g-2\bs{\s{\H}{(\D d\w)^{\perp}}}{d\phi\totimes d\w+d\w\totimes d\phi}_g\right)d\vg\right|\nonumber\\
	&\leq 8\np{|d\w|_g}{\infty}{D^2}\int_{D^2}|\D^2\w||\H||dz|^2\leq C\np{|d\w|_g}{\infty}{D^2}\int_{D^2}\frac{|dz|^2}{|z|}=2\pi C\np{|d\w|_g}{\infty}{D^2}<\infty.
	\end{align}
	
	Furthermore, as $|\mathscr{A}(\w)|\leq 2|d\w|_g|\vec{\I}|_g$, we deduce that
	\begin{align}\label{conv3}
	\int_{\Sigma}|\mathscr{A}(\w)|^2d\vg\leq 4\np{|d\w|_g}{\infty}{\Sigma}^2\int_{\Sigma}|\vec{\I}|_g^2d\vg=16\left(W(\phi)-\pi\chi(\Sigma)\right)\np{|d\w|_g}{\infty}{\Sigma}^2<\infty.
	\end{align}
	Finally, we see by \eqref{conv1}, \eqref{conv2} and \eqref{conv3} that $D^2W(\phi)(\w,\w)$ is well-defined if and only if $|d\w|_g\in L^{\infty}(\Sigma)$ and
	\begin{align*}
	\int_{\Sigma}|\Delta_g^{\perp}\w|^2d\vg<\infty,
	\end{align*}
	which was the claim to prove.
\end{proof}

\textbf{In the following, we will write $\mathrm{Var}(\phi)$ instead of $\mathrm{Var}_2(\phi)$.}

We first recall the following proposition from \cite{classification} (this is a small improvement from a lemma of \cite{beriviere}).
\begin{prop}\label{integrating}
	Let $u\in C^1(\bar{D^2}\setminus\ens{0})$ be such that
	\begin{align*}
	{\partial}_{\z}\,u(z)=\mu(z)f(z),\quad z\in D^2\setminus\ens{0} 
	\end{align*}
	where $f\in L^p(D^2)$ for some $2<p\leq \infty$, and $|\mu(z)|\simeq |z|^a\log^b|z|$ at $0$ for some $a\in \N$, and $b\geq 0$. Then
	\begin{align*}
	u(z)=P(z)+|\mu(z)|T(z)
	\end{align*}
	for some polynomial $P$ of degree less or equal than $a$, and a function $T$ such that
	\begin{align*}
	T(z)=O(|z|^{1-\frac{2}{p}}\log^{\frac{2}{p'}}|z|).
	\end{align*}
	In particular, if $f\in L^{\infty}(D^2)$, we have
	\begin{align*}
	u(z)=P(z)+O(|z|^{a+1}\log^{b+2}|z|).
	\end{align*}
\end{prop}

\begin{cor}\label{absurd}
	Let $\Sigma$ be a closed Riemann surface and let $\phi:\Sigma\rightarrow \R^n$ be a branched Willmore immersion. 
	Then $\mathrm{Ind}_W(\phi)=\mathrm{Ind}_{\mathscr{W}}(\phi)$.
\end{cor}
\begin{proof}
	Now, recall that by Lemma \eqref{d2k}
	\begin{align*}
		\frac{d^2}{dt^2}\left(K_{g_t}d\mathrm{vol}_{g_t}\right)_{|t=0}&=d\,\Im\,\bigg(2\s{\Delta_g^{\perp}\w+4\,\Re\left(g^{-2}\otimes\left(\bar{\partial}\phi\totimes\bar{\partial}\w\right)\otimes \h_0\right)}{\partial\w}-\partial\left(|\D^{\perp}\w|^2_g\right)\nonumber\\
		&-2\s{d\phi}{d\w}_g\left(\s{\H}{\partial \w}-g^{-1}\otimes\left(\h_0\totimes\bar{\partial}\w\right)\right)
		-8\,g^{-1}\otimes \left(\partial\phi\totimes\partial\w\right)\otimes \s{\H}{\bar{\partial}\w}\bigg)\\
		&=d\,\omega(\w)
  	\end{align*}
	for a $1$-form $\omega(\w,\w)$ bilinear in $\w$. From the explicit formula we obtain for some universal constant
	\begin{align*}
		|\omega(\w)+\Im(\partial|\D^{\perp}\w|_g^2)|\leq \left(|\Delta_g^{\perp}\w|e^{\lambda}\np{|d\w|_g}{\infty}{\Sigma}+C(|\H|+|\h_0|_{WP})e^{\lambda}\np{|d\w|_g}{\infty}{\Sigma}^2\right) 
	\end{align*}
	Let $p\in \Sigma$ a branch point of $\phi$ and define for all $j\in \N$ $r_j=2^{-j}$, and for all $r>0$, let $B_r(p)$ denote the open ball for the geodesic distance with respect to any smooth fixed metric and let $\partial B_r(p)$ denotes its frontier. Thanks of the Courant-Lebesgue lemma, for all $j\in \N$, there exists $\rho_j\in [r_{j+1},r_j]$ such that
	\begin{align}\label{chart}
		&\int_{\partial B_{\rho_j}(p)}\left({|\Delta^{\perp}_g\w|}+(|\H|+|\h_0|_{WP})\right)e^{\lambda}d\mathscr{H}^1\leq 2\sqrt{\pi}\left(\int_{B_{r_{j}}\setminus B_{r_{j+1}}(p)}|\Delta_g^{\perp}\w|e^{2\lambda}|dz|^2\right)^{\frac{1}{2}}
		\nonumber\\
		&+2\sqrt{\pi}\left(\int_{B_{r_j}\setminus B_{r_{j+1}}(0)}2\left(|\H|^2+|\h_0|^2_{WP}\right)d\vg\right)\conv{j\rightarrow \infty}0
	\end{align}
	as $\Delta_g^{\perp}\w\in {L}^2(\Sigma,d\vg)$. Therefore, if ${Q}_{\phi}$ is the quadratic form such that
	\begin{align*}
		D^2W(\phi)[\w,\w]=\int_{\Sigma}Q_{\phi}(\w)d\vg,
	\end{align*}
	we have
	\begin{align}\label{stokes1}
		D^2\mathscr{W}(\phi)[\w,\w]=\int_{\Sigma}Q_{\phi}(\w)d\vg-\int_{\Sigma}d\left(\omega(\w)\right).
	\end{align}
	Now, if $\phi$ has branch points $p_1,\cdots,p_m$, taking any complex chart $(U_i,z_i)$ with $z_i(U)=D^2\subset \C$, and defining $D_i(j)=z^{-1}_i\left(\bar{B}_{\rho_j(p_i)}\right)$
	\begin{align*}
		\Sigma_j=\Sigma\setminus \bigcup_{j=1}^mD_i(j),
	\end{align*}
	we obtain by Stokes theorem
	\begin{align}\label{stokes2}
		D^2\mathscr{W}(\phi)(\w,\w)=D^2W(\phi)(\w,\w)-\sum_{i=1}^{m}\int_{\partial D_i(j)}\omega(\w) 
	\end{align}
	where the $D_i$ are positively oriented (this explains the minus sign). Now, thanks of \eqref{chart}, we obtain
	\begin{align}\label{chart2}
		\left|\sum_{i=1}^m\int_{\partial D_i(j)}\omega(\w)+\Im\left(\partial|\D^{\perp}\w|_g^2\right)\right| \conv{j\rightarrow \infty}0.
	\end{align}
	Now, assume that $n=3$. Then $\w=w\n$ for some $w:\Sigma\rightarrow \R$ and $|\D^{\perp}\w|_g^2=|dw|_g^2$, while $\Delta_g^{\perp}\w=\Delta_g w$. Therefore, by regularisation, we can assume that $w$ is smooth. The condition $\Delta_gw\in L^2(D^2,d\vg)$ in a disk around a branch point of multiplicity $\theta_0\geq 2$ implies that $g=e^{2\lambda}|dz|^2=|z|^{2\theta_0-2}(1+O(|z|))$ (up to a positive multiplicative constant)
	\begin{align*}
		\int_{\Sigma}\frac{|\Delta w|^2}{|z|^{2\theta_0-2}}|dz|^2<\infty,
	\end{align*}
	and as $w$ is smooth, we have 
	\begin{align*}
		\frac{\Delta w}{|z|^{\theta_0-1}}\in L^{\infty}(D^2),
	\end{align*}
	so by Proposition \ref{integrating}, we have for some polynomial $P$ of degree less or equal than $\theta_0-1$ the expansion
	\begin{align*}
		\p{z}w=P(z)+O(|z|^{\theta_0}\log^2|z|),
	\end{align*}
	which as $w$ is smooth implies in fact
	\begin{align*}
		\p{z}w=P(z)+O(|z|^{\theta_0}).
	\end{align*}
	Furthermore, as $\dfrac{|\p{z}w|}{|z|^{\theta_0-1}}\in L^{\infty}(D^2)$, the polynomial $P$ must be for some $a\in \C$ of the form
	\begin{align*}
		P(z)=az^{\theta_0-1}.
	\end{align*}
	Therefore, we have $\p{z}w=az^{\theta_0-1}+O(|z|^{\theta_0})$
	so that
	\begin{align*}
		\frac{|\D w|^2}{|z|^{2\theta_0-2}}=4|a|^2+O(|z|),
	\end{align*}
	and as $g\in C^{2,1}(D^2)$ and $w$ is smooth, we obtain
	\begin{align*}
		\partial\left(|dw|_g^2\right)=\partial\left(|\D^{\perp}\w|^2_g\right)=O(1),
	\end{align*}
	so we deduce that for all $i=1,\cdots,m$
	\begin{align}\label{chart3}
		\lim\limits_{j\rightarrow \infty}\int_{\partial D_i(j)}\Im\left(\partial|\D^{\perp}\w|^2_g\right)=0,
	\end{align}
	Therefore, by \eqref{chart2} and \eqref{chart3} we have $D^2\mathscr{W}(\phi)[\w,\w]=D^2W(\phi)[\w,\w]$ for all $\w\in \mathrm{Var}(\phi)$ and this concludes the proof of the corollary for the case $n=3$. 
	
	We now come back to the general case where $n\geq 3$ is arbitrary.
	As the estimate we need is local, we assume from now on that $\phi:D^2\rightarrow \R^n$ is a branched Willmore disk with a unique branch point at $0$ of multiplicity $\theta_0\geq 1$.
	Let $\n:\Sigma\rightarrow \mathscr{G}_{n-2}(\R^n)$ be the unit normal, then up to shrinking the domain of $\phi$
	we can write locally $\n=\n_1\wedge\cdots\wedge \n_{n-2}$ for some $\n_1,\cdots,\n_{n-2}:\Sigma\rightarrow \R$ such that $|\n_j|=1$ for all $1\leq j\leq n-2$. Furthermore, for all $\w\in \mathscr{E}_{\phi}(D^2,\R^n)$, there exists functions $w_1,\cdots,w_{n-2}:D^2\rightarrow \R$ such that
	 and
	\begin{align*}
		\w=\sum_{j=1}^{n-2}w_j\n_j,
	\end{align*}
	In particular, by regularisation, we can assume that $w_j\in C^{\infty}(D^2)$. However, notice in general the best possible estimate for $\n$ is  $\D\n\in \mathrm{BMO}(D^2)$.
	However, recall (see \cite{classification}) that for some $\vec{A}_0\in \C^n\setminus\ens{0}$ we have the expansions (up to the normalisation $2|\vec{A}_0|^2=1$)
	\begin{align}\label{dev}
		&\p{z}\phi=\vec{A}_0z^{\theta_0-1}\left(1+O(|z|)\right)\nonumber\\
		&e^{2\lambda}=|z|^{2\theta_0-2}\left(1+O(|z|)\right)\nonumber\\
		&\vec{H}=O(|z|^{1-\theta_0})\nonumber\\
		&\h_0=O(|z|^{\theta_0-1}).
	\end{align}	
	The key point (as in the proof of Lemma \ref{gaussbonnet}) is that the unit normal $\n$ can be developed to any finite order by the corollary $4.24$ of \cite{classification}, in the form
	\begin{align}\label{ndev}
		\n=\Re\left(2i\ast \bar{\vec{A}_0}\wedge \vec{A}_0+\sum_{k,l,p}{\vec{A}_{k,l,p}}z^k\z^l\log^p|z|\right)+O(|z|^{\theta_0+1}\log^{p_0}|z|),
	\end{align}
	where $1\leq k+l\leq \theta_0$ for all the \emph{finitely many} $(k,l,p)\in \Z\times \Z\times\N$ such that $\vec{A}_{k,l,p}\neq 0$ and $1\leq k+l\leq \theta_0$ (we  take $\vec{A}_{k,l,p}$ to be $0$ outside of this range, and we also know that for all fixed $k,l$, there exists only finitely many $p\in \N$ such that $\vec{A}_{k,l,p}\neq 0$). In particular, we deduce that  there exists two integers $\alpha\in \N$ and $\beta\in \N$ such that
	\begin{align}\label{part0}
		\Delta^N\w=O(|z|^{\alpha}\log^{\beta}|z|),
	\end{align}
	and as
	\begin{align*}
		\frac{\Delta^N\w}{|z|^{\theta_0-1}}\in L^2(D^2),
	\end{align*}
	while $1/z\notin L^2(D^2)$, we deduce that $\alpha\geq \theta_0-1$.  Also, notice that by Lemma \ref{normaldelta}, \eqref{d3} and \eqref{dev}
	\begin{align}\label{part1}
		\Delta^N\w&=4\D_{\e_{\z}}^{\perp}\D_{\e_z}^{\perp}\w+4i\,\Im\left(\s{\D_{\e_z}\w}{\e_z}\bar{\H_0}\right)\nonumber\\
		&=4\D_{\p{\z}}^N\D_{\p{z}}^N\w+O(|z|^{\theta_0-1}).
	\end{align}
	Furthermore, we have by \eqref{dev}
	\begin{align}\label{part2}
		\D_{\p{\z}}^{\top}(\D_{\p{z}}^{N}\w)&=2e^{-2\lambda}\s{\D_{\p{\z}}\D^{N}_{\p{z}}\w}{\e_{\z}}\e_z+2e^{-2\lambda}\s{\D_{\p{\z}}\D^N_{\p{z}}\w}{\e_z}\e_{\z}\nonumber\\
		&=-\s{\D_{\p{z}}^N\w}{\bar{\H_0}}\e_z-\s{\D_{\p{\z}}^N\w}{\H}\e_{\z}\nonumber\\
		&=O(|z|^{\theta_0-1}|z|^{1-\theta_0}\cdot|z|^{\theta_0-1})=O(|z|^{\theta_0-1}).
	\end{align}
	Finally, by \eqref{part0}, \eqref{part1} and \eqref{part2}, we obtain
   \begin{align*}
   	\D_{\p{\z}}\left(\D_{\p{z}}^N\w\right)=O(|z|^{\theta_0-1}\log^{\beta}|z|)
   \end{align*}
	which implies that
	\begin{align}\label{nend}
		\frac{\D_{\p{\z}}\left(\D_{\p{z}}^N\w\right)}{|z|^{\theta_0-1}}\in \bigcap_{p<\infty}L^p(D^2).
	\end{align}
	 By a Proposition \ref{integrating} the estimate \eqref{nend} implies that the exists a $\C^n$-valued polynomial $\vec{P}$ of degree at most $\theta_0-1$ such that
	\begin{align*}
		\D_{\p{z}}^N\w=\vec{P}(z)+O(|z|^{\theta_0-\epsilon})\;\,\text{for all}\;\, \epsilon>0.
	\end{align*}
	Furthermore, as $\D\w=O(|z|^{\theta_0-1})$, and by decomposing this estimate into tangent and normal parts, we deduce that in fact $P$ must be a monomial, so we have
	\begin{align*}
		P(z)=\vec{C}_1z^{\theta_0-1}
	\end{align*}
	for some $\vec{C}_1\in \C^n$, and
	\begin{align}\label{part3}
		\D_{\p{z}}^N\w=\vec{C}_1z^{\theta_0-1}+O(|z|^{\theta_0-\epsilon})\;\, \text{for all}\;\, \epsilon>0.
	\end{align}
	 Finally, we obtain by \eqref{part3}
	\begin{align*}
		|\D^{\perp}\w|_g^2=4e^{-2\lambda}|(\D_{\p{z}}\w)^N|^2=4|\vec{C}_1|^2+O(|z|^{1-\epsilon})
	\end{align*}
	and we can differentiate this estimate (as the development can be performed to any arbitrary order) to obtain
	\begin{align*}
		\partial\left(|\D^{\perp}\w|^2_g\right)=O(|z|^{-\epsilon})\;\,\text{for all}\;\, \epsilon>0,
	\end{align*}
	which implies that
	\begin{align*}
		\partial\left(|\D^{\perp}\w|_g^2\right)\in \bigcap_{p<\infty}L^p(D^2).
	\end{align*}
	Therefore, the previous argument based on the Courant-Lebesgue lemma can be used to show that we have for some $\rho_j\in [r_j,r_{j+1}]$ the estimate
	\begin{align*}
		&\int_{\partial B_{\rho_j}(p)}\left(|\Delta_g^{N}\w|e^{\lambda}+\left(|\H|+|\h_0|_{WP}\right)e^{\lambda}+|\partial\left(|\D^{\perp}\w|_g^2\right)|\right)d\mathscr{H}^1\leq 2\sqrt{\pi}\left(\int_{B_{r_{j}}\setminus B_{r_{j+1}}(p)}|\Delta_g^{\perp}\w|e^{2\lambda}|dz|^2\right)^{\frac{1}{2}}
			\nonumber\\
			&+2\sqrt{\pi}\left(\int_{B_{r_j}\setminus B_{r_{j+1}}(0)}2\left(|\H|^2+|\h_0|^2_{WP}\right)d\vg\right)+2\sqrt{\pi}\int_{B_{r_j}\setminus B_{r_{j+1}}(0)}|\partial(|\D^{\perp}\w|_g^2)|^2d\mathscr{L}^2\conv{j\rightarrow \infty}0.
	\end{align*}
	Therefore, \eqref{chart2} shows that we have
	\begin{align*}
		\lim\limits_{j\rightarrow \infty}\left|\sum_{i=1}^{m}\int_{\partial D_i(j)}\omega(\w)\right|\conv{j\rightarrow \infty}0,
	\end{align*}
	or equivalently that by \eqref{stokes1} and \eqref{stokes2}, we have the equality $D^2W(\phi)(\w,\w)=D^2\mathscr{W}(\phi)(\w,\w)$. QED.
\end{proof}

\section{Proof of the main result}\label{admissible}

\subsection{Admissible classes of min-max}

Here we recall the subset of the natural classes of min-max presented by Palais in (\cite{palais2}) for which our results apply.

In the following, denote by $X$ the complete Hilbert manifold (see \cite{hierarchies}, \cite{lower})
\begin{align*}
X=\widetilde{\mathrm{Imm}}_{3,2}(\Sigma,\R^n)=\mathrm{Imm}_{3,2}(\Sigma,\R^n)/\mathrm{Diff}_+^{\ast}(\Sigma),
\end{align*}
where $\mathrm{Diff}_+^{\ast}(\Sigma)$ is the group of positive $W^{3,2}$ diffeomorphisms of $\Sigma$, fixing three distinct points if $\Sigma$ has genus $0$, fixing one point if $\Sigma$ has genus $1$.

\begin{defi}[Min-max families]\label{defminmax}

	\textbf{(1)} \textbf{ Admissible family.}	We say that $\mathscr{A}\subset \mathscr{P}(X)\setminus\ens{\varnothing}$ is an admissible min-max family of dimension $d\in\N$ with boundary $(B^{d-1},h)$ (possibly empty)
	for $X$ if
	\begin{enumerate}
		\itemsep-0.2em 
		\item[(A1)] For all $A\in\A$, $A$ is compact in $X$,
		\item[(A2)] There exists a $d$-dimensional compact \emph{Lipschitz} manifold $M^d$ with boundary $B^{d-1}$, (possibly empty)
		and a continuous map $h:B\rightarrow X$ such that for all $A\in\mathscr{A}$, there exists a continuous map $f:M^d\rightarrow X$ that $A=f(M^d)$, and $f=h$ on $B$. 
		\item[(A3)] For every homeomorphism $\varphi$ of $X$ isotopic to the identity map such that $\varphi|_{B}=\mathrm{Id}|_{h(B)}$, and for all $A\in\A$, we have $\varphi(A)\in A$.
	\end{enumerate}
	
	More generally, one can relax the notions of uniqueness of the compact manifold $M^d$ as follows. Let $I$ a set of indices and a family $\ens{M_i^d}_{i\in I}$ of compact Lipschitz manifold with boundary $(B^{d-1}_i,h_i)$. Then we define
	\begin{align*}
	\mathscr{A}=\mathscr{P}(X)\cap \ens{A: \text{there exists}\;\,i\in I\;\,\text{and}\;\, f\in C^0(M_i^d,X)\;\, \text{such that}\;\, A=f(M_i^d)\;\,\text{and}\;\, f_{B_i^{d-1}}=h_i}
	\end{align*}
	Clearly, these classes is stable under homeomorphisms of $X$ isotopic to the identity preserving the boundary $h(B)$ (resp. $h(B_i)$ for all $i\in I$).
	
	Finally, we define the following boundary values of an admissible family $\mathscr{A}$ by
	\begin{align*}
	&\widehat{\beta}(F,\mathscr{A})=\sup_{i\in I}\,\sup F(h_i(B_i^{d-1})),\quad (\text{where}\;\, C_i=h(B^{d-1}_i)\;\, \text{for all}\;\, i\in I).
	\end{align*}
\end{defi}

\begin{defi}\label{nt}
	Let  $\mathscr{A}^{\ast}$ be a $d$-dimensional admissible  min-max family  with boundary $\ens{C_i}_{i\in I}$ of $X$. We say that $\mathscr{A}$ (resp. $\mathscr{A}^{\ast}$, resp. $\widetilde{\mathscr{A}}$) is \emph{non-trivial} with respect to a continuous map $F:X\rightarrow \R$ if 
	\begin{align}\label{adapted}
	\beta(F,\mathscr{A})=\inf_{A\in \mathscr{A}}\sup F(A)>\sup \sup F(h_i(B_i^{d-1}))=\widehat{\beta}(F,\mathscr{A}).
	\end{align}
	Whenever this does not yield confusion, we shall write more simply $\beta(\mathscr{A})$ and $\widehat{\beta}(\mathscr{A})$.
\end{defi}

The second class of mappings are based on (co)-homology type properties.

\begin{defi}\label{homology}
	Let $R$ be an arbitrary ring,  $G$ be an abelian group, and $d\in \N$ be a fixed integer.
	
	\textbf{(2)} \textbf{Homological family.} Let ${\alpha}_{\ast}\in H_d(X,B,R)\setminus \ens{0}$ be a \emph{non-trivial} $d$-dimensional relative (singular) homology class of $X$ with respect to $B$ with $R$ coefficients. We say that $\underline{\mathscr{A}}=\underline{\mathscr{A}}(\alpha_{\ast})$ is a $d$-dimensional homological family with respect to $\alpha_{\ast}\in H_d(X,B,R)$ and boundary $B$ if
	\begin{align*}
	\underline{\mathscr{A}}({\alpha}_{\ast})=\mathscr{P}(X)\cap \ens{A: A\;\, \text{compact}, B\subset A\;\, \text{and}\;\, \alpha\in \mathrm{Im}(\iota^{A}_{\ast})},
	\end{align*} 
	where for all $A \supset B$, the application $\iota_{\ast}^A: H_d(A,B,R)\rightarrow H_d(X,A,R)$ is the induced map in homology from the injection $\iota^{A}:A\rightarrow X$.
	
	\textbf{(3)} \textbf{Cohomological family.} Let $\alpha^{\ast}\in H^{d}(X,G)\setminus\ens{0}$ be a \emph{non-trivial} $d$-dimensional (singular) cohomology class of $X$ with $G$ coefficients. We say that $\bar{\mathscr{A}}=\bar{\mathscr{A}}(\alpha^{\ast})$ is a $d$-dimensional cohomological family with respect to $\alpha^{\ast}\in H^d(X,G)$ if
	\begin{align*}
	\bar{\mathscr{A}}(\alpha^{\ast})=\mathscr{P}(X)\cap\ens{A: A \text{ compact and}\;\, \alpha^{\ast}\notin \mathrm{Ker}(\iota^{\ast}_{A})},
	\end{align*}
	where for all $A\subset X$, the application $\iota^{\ast}_{ A}:H^d(X,G)\rightarrow H^d(A,G)$ is the induced map in cohomology from the injection $\iota_{A}:A\rightarrow X$. In other word, the non-zero class $\alpha^{\ast}$ is not annihilated by the restriction map in cohomology $\iota^{\ast}_{ A}:H^d(X,G)\rightarrow H^d(A,G)$.
	
\end{defi}

\subsection{Statement and proof of the lower semi-continuity of the index}

We now state the main theorem for the case of admissible families, although it also holds for homological and cohomological families as presented in Definition \ref{homology}.

\begin{theorem}\label{liminf}
	Let $\mathscr{A}$ $d$-dimensional admissible family of $W^{3,2}$ immersions that we assume to be non-trivial (in the sense of Definition \ref{nt}). There exists 
	 a sequence $\ens{\sigma_k}_{k\in\N}$ of positive numbers converging to zero and $\{\phi_k\}_{k\in\N}$ a sequence of critical immersions $S^2\rightarrow\R^n$ associated to $\{{W}_{\sigma_k}\}_{k\in\N}$ such that we have
	\begin{align}\label{p0}
		\lim\limits_{k\rightarrow\infty}W_{\sigma_k}(\phi)=\beta(\mathscr{A}),\quad \partial_{\sigma}W_{\sigma_k}(\phi_k)=o\left(\frac{1}{\sigma_k\log\frac{1}{\sigma_k}}\right),\quad \mathrm{Ind}_{W_{\sigma_k}}(\phi_k)\leq d,
	\end{align}
	and there exists $\phi_{\infty}\in W^{1,\infty}(S^2,\R^n)$ and a sequence of Lipschitz homeomorphisms $\psi_k:S^2\rightarrow S^2$ such that
	\begin{align*}
		\phi_k\circ \psi_k\conv{k\rightarrow\infty}\phi_{\infty}\quad \text{strongly in}\;\, C^0(S^2,\R^n).
	\end{align*}
	then up to translations, there exists branched Willmore spheres $\vec{\xi}_j:S^2\rightarrow\R^n$ ($j=1,\cdots,m$ for some $m\geq 1$) such that
	for all $j=1,\cdots,m$, there exists finitely many points $\{p_1^j,\cdots,p_{m_j}^j\}\in S^2$ and a sequence of positive conformal diffeomorphism $\{f_k^j\}_{k\in\N}\subset \mathrm{Diff}_+(S^2)$ such that
	\begin{align}
		\phi_k\circ f_k^j\conv{k\rightarrow\infty}\vec{\xi}_j\quad\text{strongly in}\;\, C^l_{\mathrm{loc}}(S^2\setminus\{p_1^1,\cdots,p_{m_j}^j\}),
	\end{align}
	and
	\begin{align}\label{indexbound2}
		&{(\phi_{\infty})}_\ast[S^2]=\sum_{j=1}^{m}{(\vec{\xi}_j)}_{\ast}[S^2],\;\,\text{and}\quad 
		\sum_{j=1}^{m}\mathrm{Ind}_{W}(\vec{\xi}_j)\leq d.
	\end{align}
\end{theorem}

\begin{rem}
	The proof of the next theorem is long but not difficult. It amounts at checking precisely how the negative variations can be "pull-backed" to the sequence of critical points of the viscous energies. 
\end{rem}

\begin{proof}
	\textbf{Step 0 : Index bound for $\sigma>0$.} This is  direct consequence of the main result of \cite{index2} that there exists a sequence $\ens{\sigma_k}_{k\in \N}\subset (0,\infty)$  converging to zero and a sequence of critical points $\phi_{\sigma}$ of $W_{\sigma}$ such that
	\begin{align}\label{ub}
		W_{\sigma_k}(\phi_{k})=\beta(\sigma_k),\quad  W_{\sigma_k}(\phi_k)-W(\phi_k)\leq \frac{1}{\log\left(\frac{1}{\sigma_k}\right)\log\log\left(\frac{1}{\sigma_k}\right)},\quad\text{and}\quad \mathrm{Ind}_{W_{\sigma_k}}(\phi_{k})\leq d.
	\end{align}
	The hypothesis of the main result of \cite{index2} are easily seen to be satisfied. The ellipticity of the operator induced by the second derivative of $W_{\sigma}$ is easy to check as in \cite{hierarchies}, and it proved that for all critical point $\phi$ of $W_{\sigma}$ (which is smooth by \cite{eversion}),  the restriction of $D^2W_{\sigma}(\phi)$ on $W^{3,2}_{\phi}(\Sigma,T\R^n)$ is a Fredholm operator. The \textbf{Energy bound} was already checked in Remark \ref{energy}. Finally, the Palais-Smale condition is a basic ingredient needed even when we do not prescribe the index of the viscous critical points and it was proved in \cite{eversion}.
	
	Now, the whole point of this paper is to show that the upper-bound \eqref{ub} for the index is preserved as $\sigma_k\conv{k\rightarrow \infty}0$.
	
	\textbf{Step 1 : Definitions and properties of the Willmore bubble tree.}
	
	Fix some $0<\delta<\dfrac{8\pi}{3}$. Recall that by the constructions in \cite{quanta} (see Proposition $\mathrm{III.1}$) and \cite{eversion}, there exists $a^1,\cdots,a^n\in S^2$ such that for some sequence of positive diffeomorphisms $\ens{\psi_k}_{k\in \N}\subset \mathrm{Diff}_+(S^2)$
	\begin{align}\label{strong}
		\phi_k\circ \psi_k\conv{k\rightarrow\infty}\phi_{\infty}\;\, \text{strongly in}\;\, C^l_{\mathrm{loc}}(S^2\setminus\ens{a_1,\cdots,a_n})
	\end{align}
	and furthermore, there exists $Q^1,\cdots, Q^m\in \N$ such for all $1\leq i\leq m$ and for all $1\leq j\leq Q^i$, there exists sequences of points $\{x_k^{i,j}\}_{k\in \N}\subset S^2$ such that $x_k^{i,j}\conv{k\rightarrow\infty}a_i$ and sequences of radii $\{\rho^{i,j}_k\}_{k\in \N}\subset (0,\infty)$ such that $\rho_k^{i,j}\conv{k\rightarrow\infty}0$ satisfying
	\begin{align*}
		\int_{B_{\rho^{i,j}_k}(x_k^{i,j})}|d\n_k|^2_{g_k}d\mathrm{vol}_{g_k}>\frac{8\pi}{3}-\delta.
	\end{align*}
	Furthermore, for all $i\in \ens{1,\cdots,m}$ and $j\in \ens{1,\cdots,Q^i}$, the set of indices
	\begin{align*}
		I^{i,j}=\ens{1,\cdots, Q^i}\cap\ens{j': x_k^{i,j'}\in B_{\rho^{i,j}_k}(x_k^{i,j})\;\,\text{and}\;\, \dfrac{\rho_k^{i,j}}{\rho^{i,j'}_k}\conv{k\rightarrow\infty}\infty} 
	\end{align*}
	is independent of $k$. Finally, for all $0<\alpha<1$, we define the neck region $\Omega_k(\alpha)\subset S^2$ as
	\begin{align}\label{defneck}
		\Omega_k(\alpha)=\bigcup_{i=1}^m\left\{\left(B_{\alpha}(a^i)\setminus \bigcup_{j=1}^{Q_i}B_{\alpha^{-1}\rho_{k}^{i,j}}(x_k^{i,j})\right)\cup \bigcup_{j=1}^{Q^i}\bigcup_{j'\in I^{i,j}}\left(B_{\alpha \rho_k^{i,j}}(x_k^{i,j'})\setminus\bigcup_{j''\in I^{i,j}}B_{\alpha^{-1}\rho_k^{i,j''}}(x_k^{i,j''})\right)\right\}.
	\end{align}
	Notice that by the definition of $I^{i,j}$, the right part of \eqref{defneck} is not empty if $I^{i,j}\neq \varnothing$ and $k$ is large enough. One of the main result of \cite{quanta} and \cite{eversion} is to show that there is no concentration of energy in these neck regions, \emph{i.e.}
	\begin{align*}
		\lim\limits_{\alpha\rightarrow 0}\lim\limits_{k\rightarrow\infty}\np{\D\n_{k}}{2}{\Omega_k(\alpha)}=0.
	\end{align*}
	From this step one easily infer the quantization of energy. Let 
	\begin{align*}
		B(i,j,\alpha,k)=B_{\alpha^{-1}\rho_{k}^{i,j}}(x_k^{i,j})\setminus\bigcup_{j'\in I^{i,j}}B_{\alpha \rho_k^{i,j}}(x_k^{i,j'})
	\end{align*}
	and notice that by \eqref{defneck}, we have
	\begin{align}\label{trivialobs}
		\Omega_k(\alpha)\cup\bigcup_{i=1}^m\bigcup_{j=1}^{Q^i}B(i,j,\alpha,k)=\bigcup_{i=1}^mB_{\alpha}(a^i).
	\end{align}
	
	Now, if $\vec{\Psi}:S^2\rightarrow \R^n$, we let $Q_{\Psi}$ the quadratic function in the variation $\w\in \mathrm{Var}(\vec{\Psi})$ such that 
	\begin{align*}
	D^2W(\Psi)[\w,\w]=\int_{S^2}Q_{\Psi}(\w)d\mathrm{vol}_{g_{\Psi}}
	\end{align*}
	is well-defined.
	We remark thanks of the strong convergence \eqref{strong} and thanks of the invariance by re-parametrisation of the Willmore energy that for all $0<\alpha<1$, we have
	\begin{align}\label{t1}
		\lim\limits_{k\rightarrow\infty}\int_{S^2\setminus \cup_{i=1}^mB_{\alpha}(a^i)}Q_{\phi_k}(P(\phi_k)\w)d\mathrm{vol}_{g_{\phi_k}}=\int_{S^2\setminus \cup_{i=1}^mB_{\alpha}(a^i)}Q_{\phi_{\infty}}(\w)d\mathrm{vol}_{g_{\phi_{\infty}}}
	\end{align}
	if $P(\phi_k)$ is the projection on $\mathrm{Var}(\phi_k)$.
	
	\textbf{We write in the following $\w_k=P(\phi_k)\w$ for all $\w\in \mathrm{Var}(\phi_{\infty})$}.
	
	Furthermore, as $\phi_{\infty}$ is a branched Willmore sphere, we have for all $\w\in \mathrm{Var}(\phi_{\infty})$
	\begin{align}\label{t2}
		\lim_{\alpha\rightarrow 0}\int_{S^2\setminus  \cup_{i=1}^mB_{\alpha}(a^i)}Q_{\phi_{\infty}}(\w)d\mathrm{vol}_{g_{\phi_{\infty}}}=D^2W(\phi_{\infty})[\w,\w].
	\end{align}
	Taking together \eqref{t1} and \eqref{t2} yields
	\begin{align}\label{est1}
		\lim_{\alpha\rightarrow 0}\lim_{k\rightarrow\infty}\int_{S^2\setminus \cup_{i=1}^mB_{\alpha}(a^i)}Q_{\phi_k}(\w_k)d\mathrm{vol}_{g_{\phi_k}}=D^2W(\phi_{\infty})[\w,\w].
	\end{align}
	
	Finally, for all $1\leq i\leq n$, for all $1\leq j\leq Q^i$, there exists $z_k^{i,j}\in B(i,j,\alpha,k)$, there exists a branched Willmore sphere $\phi_{\infty}^{i,j}:\C\rightarrow \R^3$, and for all $j'\in I^{i,j}$ there exists and $a_i^{j,j'}(k)\conv{k\rightarrow\infty}a_i^{j,j'}\in D_{\C}^2(0,1)$, such that the renormalised immersion
	\begin{align*}
		&\phi_k^{i,j}(\alpha):A_{\C}^{i,j}(\alpha)=D_{\C}^2(0,\alpha^{-1})\setminus\bigcup_{j''\in I^{i,j}}D_{\C}^2(a^{j,j'}_i(k),\alpha)\rightarrow \R^3\\
		&\phi_k^{i,j}(\alpha)(y)=e^{-\lambda_k(z_k^{i,j})}\left(\phi_k(\rho_k^{i,j}y+x_{k}^{i,j})-\phi_k(x_k^{i,j})\right)
	\end{align*}
	satisfies the following conditions:
	\begin{enumerate}
		\item[(1)] For all $j'\in I^{i,j}$, $\lim_{k\rightarrow \infty}a_i^{j,j'}(k)=\lim\limits_{k\rightarrow\infty}\dfrac{x_k^{i,j}-z_k^{i,j}}{\rho_k^{i,j}}=a_i^{j,j'}$.
		\item[(2)] $\phi_k^{i,j}(\alpha)\conv{k\rightarrow\infty} \phi_{\infty}^{i,j}$ in $C^l(A_{\C}^{i,j}(\alpha))$ for all $l\in \N$.
	\end{enumerate}
    Now fix some $0<\alpha<1$.
    We obtain thanks of the strong convergence in all $C^l$ that for any variation $\w\in W^{3,2}(S^2,\R^3)$ and the explicit expression, we have
    \begin{align*}
    	\lim\limits_{k\rightarrow\infty}\int_{A_{\C}^{i,j}(\alpha)}Q_{\phi_k^{i,j}}(\w_k)d\mathrm{vol}_{g_{\phi_k^{i,j}}}=\int_{A_{\C}^{i,j}(\alpha)}Q_{\phi^{i,j}_{\infty}}(\w)d\mathrm{vol}_{g_{\phi_{\infty}^{i,j}}}
    \end{align*}
    so that
    \begin{align}\label{est3}
    	\lim\limits_{\alpha\rightarrow 0}\lim_{k\rightarrow\infty}\int_{A_{\C}(\alpha)}Q_{\phi_k^{i,j}}(\w_k)d\mathrm{vol}_{g_{\phi_k^{i,j}}}=D^2W(\phi_{\infty}^{i,j})[\w,\w].
    \end{align}
    Notice that for all $0<\alpha<1$, for all $(i_1,j_1)\neq (i_2,j_2)$ and for all large enough $k\geq1$ the three subset of $S^2$
    \begin{align*}
    	S^2\setminus \bigcup_{i=1}^{n}B_{\alpha}(a^i),\;\, \Omega_k(\alpha),\;\, B(i_1,j_1,\alpha,k)\;\, \text{and}\;\, B(i_2,j_2,\alpha,k)
    \end{align*}
    are pair-wise disjoint. 
    
    \textbf{Step 2 : Cut-off function.}
    
    Now, we introduce a continuous function $f:(0,1)\rightarrow (0,1)$ such that for all $0<\alpha<1$, we have $f(\alpha)<\alpha$ and
    \begin{align}\label{falpha}
    \frac{\alpha}{f(\alpha)}\conv{\alpha\rightarrow 0}\infty
    \end{align}
    and a cut-off function $\eta:S^2\rightarrow [0,1]$ such that (this is the classical logarithm cut-off trick, see \cite{coldingminicozzi1} or the proof of theorem $\mathrm{III.3}$ in \cite{hierarchies})
    \begin{align*}
    \left\{
    \begin{alignedat}{2}
    &\eta=1\quad \text{in}\;\, S^2\setminus\bigcup_{i=1}^m B_{\alpha}(a^i)\\
    &\mathrm{supp}(\eta)\subset S^2\setminus \bigcup_{i=1}^m B_{f(\alpha)}(a^i) 
    \end{alignedat}\right.
    \end{align*}
    defined on $B_{\alpha}(a^i)\setminus B_{f(\alpha)}(a^i)$ by
    \begin{align*}
    \eta(z)=-\frac{\log(f(\alpha))}{\log\left(\frac{\alpha}{f(\alpha)}\right)}+\frac{\log|z-a^i|}{\log\left(\frac{\alpha}{f(\alpha)}\right)}.
    \end{align*}

    \textbf{Step 3 : Passage to the limit in the second derivative of the Willmore energy of the macroscopic surface.}
    
    Now, suppose that the index of $\phi_{\infty}$ is equal to $d\geq 1$ (if $\phi_{\infty}$ is stable, there is nothing to prove). Then there exists $\w^1,\cdots,\w^d\in \mathrm{Var}(\phi_{\infty})$ orthogonal in $L^2$ such that for all $1\leq l\leq d$, we have for some $0<\delta_l<\infty$
    \begin{align*}
    	D^2W(\phi_{\infty})(\w^l,\w^l)<-\delta_l.
    \end{align*}
    In particular, thanks of the property \eqref{t2} there exists $0<\alpha<1$ small enough such that
    \begin{align*}
    	\int_{S^2\setminus\cup_{i=1}^m{B_{\alpha}}(a^i)}Q_{\phi_{\infty}}(\w^{l})d\mathrm{vol}_{g_{\phi_{\infty}}}<-\frac{\delta_l}{2}\;\,\text{for all}\;\, 1\leq l\leq d.
    \end{align*}
    Now, thanks of \eqref{t1}, we obtain for $k\geq 1$ large enough that
    \begin{align}\label{ineq8}
    	\int_{S^2\setminus \cup_{i=1}^mB_{\alpha}(a^i)}Q_{\phi_k}(\w^l_k)d\mathrm{vol}_{g_{\phi_k}}<-\frac{\delta_l}{4}.
    \end{align}
    Taking $k$ large enough, we can suppose as 
    \begin{align*}
    	\Omega_k(\alpha)\conv{k\rightarrow\infty}\bigcup_{i=1}^mB_{\alpha}(a^i)\setminus\ens{a^i}\;\, \text{uniformly}
    \end{align*}
    that $\mathrm{supp}(\eta)\subset \Omega_k(\alpha)$. Then, we have 
    \begin{align}\label{ineq9}
    	\left|D^2W(\phi_k)[\eta\,\w_k^l,\eta\,\w_k^l]-\int_{S^2\setminus\cup_{i=1}^mB_{\alpha}(a^i)}Q_{\phi_k}(\w_k^l)d\mathrm{vol}_{g_{\phi_k}}\right|\leq \left|\int_{\supp(\D\eta)}Q_{\phi_k}(\eta\,\w^l_k)d\mathrm{vol}_{g_{\phi_k}}\right|
    \end{align}
    so if we prove that for an appropriate choice of $f(\alpha)$, we have
    \begin{align*}
    	\lim_{\alpha\rightarrow 0}\limsup_{k\rightarrow\infty}\left|\int_{\supp(\D\eta)}Q_{\phi_k}(\eta\,\w^l_k)d\mathrm{vol}_{g_{\phi_k}}\right|=0
    \end{align*}
    this will imply by \eqref{ineq8} and \eqref{ineq9} that for $\alpha$ small enough and $k$ large enough, we have
    \begin{align*}
    	D^2W(\phi_k)[\eta\,\w_k^l,\eta\,\w_k^l]<-\frac{\delta_l}{8}.
    \end{align*}
    Therefore, we will have proved as the $\eta\,\w^l_k$ are linearly independent for $0<\alpha<1$ small enough that
    \begin{align*}
    	\mathrm{Ind}_{W}(\phi_{\infty})\leq \liminf_{k\rightarrow\infty}\mathrm{Ind}_{W_{\sigma_k}}(\phi_k).
    \end{align*}
    Now, recall the explicit formula
    \begin{align}\label{d2w1}
    &D^2W(\phi)(\w,\w)=\int_{\Sigma}\frac{d^2}{dt^2}|\H_{g_t}|^2_{|t=0}d\vg+2\int_{\Sigma}\frac{d}{dt}|\H_{g_t}|^2\frac{d}{dt}\left(d\mathrm{vol}_{g_t}\right)_{|t=0}+\int_{\Sigma}|\H|^2\frac{d}{dt}\left(d\mathrm{vol}_{g_t}\right)\nonumber\\
    &=\int_{\Sigma}\bigg\{-2\bs{\s{\H}{\vec{\I}}}{d\w\totimes d\w}_g+4\s{d\phi}{d\w}_g\bs{\s{\H}{\vec{\I}}}{d\phi\totimes d\w+d\w\totimes d\phi}_g\nonumber\\
    &-4\left(\s{d\phi}{d\w}_g^2-16|\partial\phi\totimes \partial\w|_{WP}^2\right)|\H|^2
    -\bs{\s{\Delta_g\w}{d\phi}}{\s{\H}{d\w}}_g-2\bs{\s{\H}{(\D d\w)^{\perp}}}{d\phi\totimes d\w+d\w\totimes d\phi}_g\nonumber\\
    &+\frac{1}{2}\left|\Delta^{\perp}_g\w+\mathscr{A}(\w)\right|^2+2\s{\Delta_g^{\perp}\w+\mathscr{A}(\w)}{\H}\s{d\phi}{d\w}_g+|\H|^2\left(|d\w|_g^2-16|\partial\phi\totimes \partial\w|_{WP}^2\right) \bigg\}d\vg.
    \end{align}
    Thanks of the strong convergence, we have
    \begin{align*}
    	\int_{\supp(\D\eta)}Q_{\phi_k}(\eta\,\w_k^l)d\mathrm{vol}_{g_{\phi_k}}\conv{k\rightarrow\infty}\int_{\supp(\D\eta)}Q_{\phi_{\infty}}(\eta\,\w_\infty^l)d\mathrm{vol}_{g_{\phi_\infty}}.
    \end{align*}
    As $\eta$ is a harmonic function and thanks of \eqref{d2w1}, we see that it suffices to prove that the following quantities
    \begin{align*}
        &J^1_k(\alpha)=\int_{B_{\alpha}(a^i)\setminus B_{f(\alpha)}(a^i)}|\D d(\eta\,\w_k^l)|_{{g_{\phi_k}}}\,|\H_{g_{\phi_k}}|d\mathrm{vol}_{g_{\phi_k}}\\
    	&J^2_k(\alpha)=\int_{B_{\alpha}(a^i)\setminus B_{f(\alpha)}(a^i)}|d(\eta\,\w_k^l)|_{g_{\phi_k}}^2|\H_{g_{\phi_k}}|||\vec{\I}_{g_{\phi_k}}|d\mathrm{vol}_{g_{\phi_k}}\\
    	&J^3_k(\alpha)=\int_{B_{\alpha}(a^i)\setminus B_{f(\alpha)}(a^i)}|\Delta_{g_{\phi_k}}^{\perp}(\eta\,\w_k^l)|^2d\mathrm{vol}_{g_{\phi_k}}
    \end{align*}
    verify
    \begin{align}\label{prop}
    	\lim\limits_{\alpha\rightarrow 0}\limsup_{k\rightarrow \infty}J^{l}_k(\alpha)=0,\quad \text{for}\; l=1,2,3.
    \end{align}
	Now, as there is no concentration of energy in $B_{\alpha}(a^i)\setminus B_{f(\alpha)}(a^i)$, we deduce with obvious notations that
	\begin{align*}
		J_{k}^l(\alpha)\conv{k\rightarrow \infty}J^l_{\infty}(\alpha)\quad \text{for}\;\, l=1,2,3.
	\end{align*}
	Now write to simplify $g=g_{\phi_{\infty}}$ and $e^{2\lambda}$ the conformal parameter of $\phi_{\infty}$. Then if $\phi_{\infty}$ does not have a branch point at $a^i$ (or has a branch point of multiplicity $1$, though this cannot happen as we would not have a \emph{true} Willmore), the estimates are trivial. Otherwise, we have, up to a normalisation constant which we take equal to $1$, the estimate
	\begin{align*}
		e^{2\lambda}=|z|^{2\theta_0-2}\left(1+O(|z|)\right)
	\end{align*}
	for some integer $\theta_0\geq 2$. As $\w\in \mathrm{Var}(\phi_{\infty})$, we have 
	\begin{align}\label{boundvar}
		\int_{S^2}|\Delta_{g}^{\perp}\w|^2d\vg+\np{|d\w|_g}{\infty}{S^2}^2\leq C<\infty.
	\end{align}
	for some positive constant $C<\infty$. Now, as $\eta$ is harmonic on $B_{\alpha}(a^i)\setminus B_{f(\alpha)}(a^i)$, we have 
	\begin{align*}
		\Delta_g^{\perp} (\eta\,\w)=\eta\,\Delta_g^{\perp}\w+2\s{d\eta}{\D^{\perp}\w}_g.
	\end{align*}
	Using \eqref{boundvar}, we deduce that
	\begin{align}\label{boundvar2}
		\int_{B_{\alpha}(a^i)\setminus B_{f(\alpha)}(a^i)}|\Delta_g^{\perp}(\eta\,\w)|^2d\vg&\leq 2\int_{B_{\alpha}(a^i)\setminus B_{f(\alpha)}(a^i)}|\Delta_g^{\perp}\w|^2d\vg\nonumber\\
		&+8\np{|d\w|_g}{\infty}{S^2}^2\int_{B_{\alpha}(a^i)\setminus B_{f(\alpha)}(a^i)}|\D\eta|^2|dz|^2
	\end{align}
	As
	\begin{align}\label{boundvar3}
		\int_{B_{\alpha}(a^i)\setminus B_{f(\alpha)}(a^i)}|\D\eta|^2|dz|^2=\int_{B_{\alpha}\setminus B_{f(\alpha)}(0)}\frac{|dz|^2}{\log^2\left(\frac{\alpha}{f(\alpha)}\right)|z|^2}=\frac{2\pi}{\log\left(\frac{\alpha}{f(\alpha)}\right)}
	\end{align}
	we deduce from \eqref{boundvar}, \eqref{boundvar2} and \eqref{boundvar3} that
	\begin{align*}
		J_{\infty}^3(\alpha)=\int_{B_{\alpha}(a^i)\setminus B_{f(\alpha)}(a^i)}|\Delta_g^{\perp}(\eta\,\w)|^2d\vg\leq 2\int_{B_{\alpha}(a^i)\setminus B_{f(\alpha)}(a^i)}|\Delta_g^{\perp}\w|^2d\vg+16\pi\,C\frac{1}{\log\left(\frac{\alpha}{f(\alpha)}\right)}\conv{\alpha\rightarrow 0}0
	\end{align*}
	thanks of \eqref{falpha}.
	
	Now, let $\nu\leq \theta_0-2$ such that
	\begin{align}\label{improved}
		\H=\Re\left(\frac{\vec{C}_1}{z^{\nu}}\right)+O(|z|^{1-\alpha-\epsilon})
	\end{align}
	for all $\epsilon>0$. We also easily have $|\h_0|=O(|z|^{\theta_0-1})$, so that for some constant $C$ depending only on $\phi_{\infty}$
	\begin{align*}
		\int_{B_{\alpha}\setminus B_{f(\alpha)}(0)}|\D d{\eta}\,\w|_g\,|\H|d\vg&\leq C\np{\w}{\infty}{S^2}\int_{B_\alpha\setminus B_{f(\alpha)}(0)}\frac{1}{\log\left(\frac{\alpha}{f(\alpha)}\right)|z|^2}\cdot |z|^{1-\theta_0}\cdot |z|^{-\nu}\cdot  |z|^{2\theta_0-2}|dz|^2\\
		&=C\np{\w}{\infty}{S^2}\frac{1}{\log\left(\frac{\alpha}{f(\alpha)}\right)}\frac{1}{\theta_0-1-\nu}\left(\alpha^{\theta_0-1-\nu}-f(\alpha)^{\theta_0-1-\nu}\right)
	\end{align*}
	As $\nu\leq \theta_0-2$, we have $\theta_0-1-\nu\geq 1$, so we have
	\begin{align*}
		\int_{B_{\alpha}\setminus B_{f(\alpha)}(0)}|\D\,d{\eta}\,\w|_g		
		|\H|d\vg&\leq {C}\np{\w}{\infty}{S^2}\frac{\alpha}{\log\left(\frac{\alpha}{f(\alpha)}\right)}\conv{\alpha\rightarrow 0}0.
	\end{align*}
	Now, the other terms are treated similarly by virtue of Lemma \ref{est2} for the other components of $J_{\infty}^1(\alpha)$ so that
	\begin{align*}
		J_{\infty}^1(\alpha)\conv{\alpha\rightarrow 0}0.
	\end{align*}
	Finally, we have
	\begin{align*}
		J_{\infty}^2(\alpha)\leq 2\np{|d\w|_g}{\infty}{S^2}^2\int_{B_{\alpha}\setminus B_{f(\alpha)}(a^i)}|\vec{\I}|_g^2d\vg+2\np{\w}{\infty}{S^2}^2\int_{B_{\alpha}\setminus B_{f(\alpha)}(a^i)}|\D\eta|^2|\H||\vec{\I}|_gd\vg
	\end{align*}
	We have
	\begin{align*}
		\int_{B_{\alpha}\setminus B_{f(\alpha)}(a^i)}|\D\eta|^2|\H||\vec{\I}|_gd\vg\leq C\int_{B_\alpha\setminus B_{f(\alpha)}(0)}\frac{1}{\log^2\left(\frac{\alpha}{f(\alpha)}\right)|z|^2}|z|^{-\nu}|z|^{1-\theta_0}|z|^{2\theta_0-2}|dz|^2\leq C\frac{\alpha}{\log^2\left(\frac{\alpha}{f(\alpha)}\right)}\conv{\alpha\rightarrow 0}0
	\end{align*}
	so that
	\begin{align*}
		\lim\limits_{\alpha\rightarrow 0}J_{\infty}^2(\alpha)=0.
	\end{align*}
	Therefore, we obtain
	\begin{align*}
		\mathrm{Ind}_W(\phi_{\infty})\leq \liminf_{k\rightarrow\infty}\mathrm{Ind}_{W_{\sigma_k}}(\phi_k).
	\end{align*}
	
	\textbf{Step 4 : Passage to the limit in the viscous energy.}
	
	Recall that
	\begin{align*}
	&D^2F(\phi)(\w,\w)=2\int_{S^2}\bigg\{-2\bs{\s{\H}{\vec{\I}}}{d\w\totimes d\w}_g+4\s{d\phi}{d\w}_g\bs{\s{\H}{\vec{\I}}}{d\phi\totimes d\w+d\w\totimes d\phi}_g\\
	&-4\left(\s{d\phi}{d\w}_g^2-16|\partial\phi\totimes \partial\w|_{WP}^2\right)|\H|^2
	-\bs{\s{\Delta_g\w}{d\phi}}{\s{\H}{d\w}}_g-2\bs{\s{\H}{(\D d\w)^{\perp}}}{d\phi\totimes d\w+d\w\totimes d\phi}_g\\
	&+\frac{1}{2}\left|\Delta^{\perp}_g\w+\mathscr{A}(\w)\right|^2+2\s{d\phi}{d\w}_g\s{\H}{\Delta_g^{\perp}\w+\mathscr{A}(\w)}+|\H|^2\left(|d\w|_g^2-16|\partial\phi\totimes \partial\w|_{WP}^2\right) \bigg\}\left(1+|\H|^2\right)d\vg\\
	&+2\int_{S^2}\s{\Delta_g^{\perp}\w+\mathscr{A}(\w)}{\H}^2d\vg+4\int_{\Sigma}\s{d\phi}{d\w}_g\s{\H}{\Delta_{g}^{\perp}\w+\mathscr{A}(\w)}\left(1+|\H|^2\right)d\vg\\
	&+\int_{S^2}\left(|d\w|_g^2-16|\partial\phi\totimes \partial\w|_{WP}^2\right)\left(1+|\H|^2\right)^2d\vg.
	\end{align*}
	Let us write $T$ the quadratic form on $\w$ such that
	\begin{align*}
	D^2F(\phi)[\w,\w]=\int_{S^2}T_{\phi}(\w)d\mathrm{vol}_{g_{\phi}}
	\end{align*}	
	First, on $S^2\setminus \cup_{i=1}^mB_{f(\alpha)}(a^i)$, by the strong convergence in all $C^l(S^2\setminus \cup_{i=1}^mB_{f(\alpha)}(a^i))$ for all $l\in \N$, we have (recall that $\eta\,\w_k^l$ is supported in $S^2\setminus \cup_{i=1}^mB_{f(\alpha)}(a^i)$)
	\begin{align*}
	\lim_{k\rightarrow \infty}D^2F(\phi_k)[\eta\,\w_k^l,\eta\,\w_k^l]=\lim\limits_{k\rightarrow \infty}\int_{S^2\setminus \cup_{i=1}^nB_{f(\alpha)}(a^i)}T_{\phi_k}(\eta\w_k^l)d\mathrm{vol}_{g_{\phi_k}}=\int_{S^2\setminus\cup_{i=1}^m B_{f(\alpha)}(a^i)}T_{\phi_{\infty}}(\eta\w^l)d\mathrm{vol}_{g_{\phi_{\infty}}}
	\end{align*}
	which is a finite quantity as $\phi_{\infty}$ is real analytic on $S^2\setminus \bigcup_{i=1}^m B_{f(\alpha)}(a^i)$. In particular, we have
	\begin{align}
	\lim_{k\rightarrow \infty}\sigma_k^2D^2F(\phi_k)(\eta\,\w_k^l,\eta\,\w_k^l)=0.
	\end{align}
	and \emph{a fortiori}
	\begin{align*}
	\lim\limits_{\alpha\rightarrow 0}\limsup_{k\rightarrow \infty}\left|\sigma_k^2D^2F(\phi_k)(\eta\,\w_k^l,\eta\,\w_k^l)\right|=0
	\end{align*}
	which concludes the proof of this step. By fixing a Aubin gauge (se \cite{eversion}) for the parameters $\alpha_k:S^2\rightarrow \R$ such that $g_{\phi_k}=e^{2\alpha_k}g_{0,k}$ for some constant Gauss curvature metric $g_{0,k}$ of volume $1$ the estimate
	\begin{align}
		\frac{1}{6}\int_{S^2}|d\alpha_k|_{g_k}^2d\mathrm{vol}_{g_k}\leq \mathscr{O}(\phi_k)
	\end{align}
	which implies by the proof of Corollary \ref{endproof}, we have 
	\begin{align}\label{key}
	\left|D^2\mathscr{O}(\phi_k)(\eta\,\w_k^l,\eta\,\w_k^l)\right|&\leq C\left(1+W(\phi_k)+\int_{S^2}|d\alpha_{k}|^2_{g_{\phi_k}}d\mathrm{vol}_{g_{\phi_k}}\right)\nonumber\\
	&\qquad\qquad\times\left(\np{|d(\eta\w_k^l)|_g^2}{\infty}{S^2}^2+\int_{S^2}|\Delta_{g_{\phi_k}}^{\perp}(\eta\w_k^l)|^2_{g_{\phi_k}}d\mathrm{vol}_{g_{\phi_k}}\right)\nonumber\\
	&\leq C\left(1+W(\phi_k)+6\mathscr{O}(\phi_k)\right)\left(\np{|d(\eta\w_k^l)|_g}{\infty}{S^2}^2+\int_{S^2}|\Delta_{g_{\phi_k}}^{\perp}(\eta\w_k^l)|^2_{g_{\phi_k}}d\mathrm{vol}_{g_{\phi_k}}\right)\nonumber\\
	&\leq C\left(1+W(\phi_k)+6\,\mathscr{O}(\phi_k)\right)
	\end{align}
	by hypothesis on $\w^l$, the second component in the product of the right-hand side of \eqref{key} is uniformly bounded independently of $k$ thanks of \eqref{boundvar2}. Furthermore, we have the estimate
	\begin{align*}
		W(\phi_k)\leq W_{\sigma_k}(\phi_k)=\beta(\sigma_k)\conv{k\rightarrow \infty}\beta_0,
	\end{align*}
	so we have for $k$ large enough
	\begin{align*}
		\left|D^2\mathscr{O}(\phi_k)(\eta\,\w_k^l,\eta\,\w_k^l)\right|\leq C\left(1+\beta_0+6\,\mathscr{O}(\phi_k)\right)
	\end{align*}
	and for some constant $C$ depending only on $\beta_0$ we obtain by the entropy condition
	\begin{align*}
		\frac{1}{\log\left(\frac{1}{\sigma_k}\right)}\left|D^2\mathscr{O}(\phi_k)(\eta\,\w_k^l,\eta\,\w_k^l)\right|\leq \frac{C}{\log\left(\frac{1}{\sigma_k}\right)}+\frac{C}{\log\left(\frac{1}{\sigma_k}\right)}\mathscr{O}(\phi_k)\leq \frac{C}{\log\left(\frac{1}{\sigma_k}\right)}+\frac{C}{\log\log\left(\frac{1}{\sigma_k}\right)}\conv{k\rightarrow \infty}0.
	\end{align*}
	
	\textbf{Step 5: Passage to the limit in the Willmore energy of the bubbles.}
	
	Recall that for all $1\leq i\leq m$, and for all $1\leq j\leq Q^i$, we have
   \begin{align*}
		B(i,j,\alpha,k)=B_{\alpha^{-1}\rho_{k}^{i,j}}(x_k^{i,j})\setminus\bigcup_{j'\in I^{i,j}}B_{\alpha \rho_k^{i,j}}(x_k^{i,j'}).
    \end{align*}
    Now, we remark that for all $\tilde{\w}_k\in \mathrm{Var}(\phi_k)$, we have for $\w_k(y)=\tilde{\w}_k(\rho_k^{i,j}y+x_{k}^{i,j})$ and
    \begin{align*}
    	{D}(i,j,\alpha,k)=D_{\C}(0,\alpha^{-1})\setminus \cup_{j'\in I^{i,j}}D_{\C}(a^{j,j'}_i(k),\alpha)
    \end{align*}
     the identities
    \begin{align}\label{changemap}
        &\np{|d\tilde{\w}_k|_{g_{\phi_k}}}{\infty}{\mathrm{dom}(\tilde{\w}_k)}=\np{|d\w_k|_{g_{\phi_k^{i,j}}}}{\infty}{\mathrm{dom}(\w_k)}\nonumber\\
    	&\int_{B(i,j,\alpha,k)}|\H_{\phi_k}|^2d\mathrm{vol}_{g_{\phi_k}}=\int_{D(i,j,\alpha,k)}|\H_{\phi_k^{i,j}}|^2d\mathrm{vol}_{g_{\phi_k^{i,j}}}\nonumber\\
    	&\int_{B(i,j,\alpha,k)}Q_{\phi_k}(\tilde{\w}_k,\tilde{\w}_k)d\mathrm{vol}_{g_{\phi_k}}=\int_{D(i,j,\alpha,k)}Q_{\phi_k}(\w_k,\w_k)d\mathrm{vol}_{g_{\phi_k^{i,j}}}\nonumber\\
    	&\int_{B(i,j,\alpha,k)}T_{\phi_k}(\tilde{\w}_k,\tilde{\w}_k)d\mathrm{vol}_{g_{\phi_k}}=\int_{D(i,j,\alpha,k)}T_{\phi_k}({\w}_k,{\w}_k)d\mathrm{vol}_{g_{\phi_k^{i,j}}}.
    \end{align}
    Now, let $\w^1,\cdots,\w^{d_{i,j}}\in \mathrm{Var}(\phi_{\infty}^{i,j})$ an orthonormal basis in $L^2(S^2,g_0)$ (where $g_0$ is the standard constant Gauss curvature metric on $S^2$) of negative variations of $\phi^{i,j}_{\infty}$, and $\w_k^l=P(\phi_k^{i,j})\w^l$ for all $1\leq l\leq d_{i,j}$, we have thanks of the strong convergence on $D(i,j,\alpha,k)$ that for any test function $\eta\in C^{\infty}_{c}(\bar{D}(i,j,\infty,\alpha))$ that
    \begin{align*}
    	&\int_{D(i,j,k,\alpha)}Q_{\phi^{i,j}_k}(\eta\,\w_k^l)d\mathrm{vol}_{g_{\phi_k^{i,j}}}\conv{k\rightarrow \infty}\int_{D(i,j,\infty,\alpha)}Q_{\phi^{i,j}_{\infty}}(\eta\,\w^l)d\mathrm{vol}_{g_{\phi_{\infty}^{i,j}}}\\
    	&\int_{D(i,j,k,\alpha)}T_{\phi^{i,j}_k}(\eta\,\w_k^l)d\mathrm{vol}_{g_{\phi_k^{i,j}}}\conv{k\rightarrow \infty} \int_{D(i,j,\infty,\alpha)}T_{\phi^{i,j}_{\infty}}(\eta\,\w^l)d\mathrm{vol}_{g_{\phi_{\infty}^{i,j}}}
    \end{align*}
    where the last quantity is finite as $\phi_{\infty}^{i,j}$ is real analytic in an open neighbourhood of $D_{\C}(0,\alpha^{-1})\setminus \cup_{j'\in I^{i,j}}D_{\C}(a^{j,j'}_i,\alpha)$ for all $\alpha>0$.
	Therefore, we obtain
	\begin{align}\label{zerolim}
		\lim\limits_{k\rightarrow\infty}\sigma_k^2D^2F(\phi_k)(\tilde{\eta}\tilde{\w}_k,\tilde{\eta}\tilde{\w}_k)=0.
	\end{align}
	Now, by convenience of notation, we replace the bubble domain $D(i,j,\infty,\alpha)$ by $D(i,j,\infty,\alpha^2)$ and we let $\eta:D_{\C}(0,\alpha^{-1})\setminus \cup_{j'\in I^{i,j}}D_{\C}(a^{j,j'}_i,\alpha)$, such that 
	\begin{align*}
		\begin{alignedat}{1}
		&\eta=1\quad \text{in}\;\, D_{\C}(0,\alpha^{-1})\setminus \bigcup_{j'\in I^{i,j}}D_{\C}(a^{j,j'}_i,\alpha)\\
		&\eta(z)=2+\frac{\log|z|}{\log(\alpha)}\quad \text{for}\;\, z\in D_{\C}(0,\alpha^{-2})\setminus D_{\C}(0,\alpha^{-1})\\
		&\eta(z)=2-\frac{\log|z-a_i^{j,j'}|}{\log(\alpha)}\quad \text{for}\;\, z\in D_{\C}(a^{j,j'}_i,\alpha)\setminus D_{\C}(a^{j,j'}_i,\alpha^2).
		\end{alignedat}
	\end{align*}
	Now, at this step, \emph{by the exact same argument} as the one given in the pervious step, we have
	\begin{align}\label{zerolim2}
		\lim\limits_{\alpha\rightarrow 0}\int_{D(i,j,\infty,\alpha)\setminus D(i,j,\infty,\alpha^2)}Q_{\phi_{\infty}^{i,j}}(\eta\w^l)d\mathrm{vol}_{g_{\phi_{\infty}^{i,j}}}=0
	\end{align}
	and as $\w^l\in \mathrm{Var}(\phi^{i,j}_{\infty})$, we deduce that
	\begin{align}\label{negvar}
		\int_{D(i,j,\infty,\alpha^2)}Q_{\phi^{i,j}_{\infty}}(\eta\w^l)d\mathrm{vol}_{g_{\phi_{\infty}^{i,j}}}=\int_{D(i,j,\infty,\alpha^2)}Q_{\phi^{i,j}_{\infty}}(\w^l)d\mathrm{vol}_{g_{\phi_{\infty}^{i,j}}}\conv{\alpha\rightarrow 0}D^2W(\phi^{i,j}_{\infty})[\w,\w]<0.
	\end{align}
	Here we suppose that the bubble is compact, which can always be assumed by taking a suitable inversion. 
	In particular, if $\delta_l>0$ is such that $D^2W(\phi^{i,j}_{\infty})[\w,\w]<-\delta_l$, thanks of \eqref{negvar} and \eqref{negvar}, there exists some $\alpha_l>0$ such that for all $0<\alpha<\alpha_l$
	\begin{align*}
		\int_{D(i,j,\infty,\alpha^2)}Q_{\phi_{\infty}^{i,j}}(\eta\,\w^l,\eta\,\w^l)d\mathrm{vol}_{g_{\phi_k^{i,j}}}\leq -\frac{\delta_l}{2}
	\end{align*}
	which implies in turn by \eqref{changemap} and \eqref{zerolim} that for $k$ large enough and $0<\alpha<\alpha_l$, we have
	\begin{align*}
		D^2W_{\sigma_k}(\phi_k)[\tilde{\eta}\,\tilde{\w}_k^l,\tilde{\eta}\,\tilde{\w}_k^l]\leq -\frac{\delta_l}{4}.
	\end{align*}
	Therefore, $\displaystyle 0<\alpha<\min_{1\leq l\leq d_{i,j}}\alpha_l$, there exist $K\geq 1$ such that for all $k\geq K$ the $\eta\,\w_k^l$ are linearly independent and
	\begin{align*}
		D^2W_{\sigma_k}(\phi_k)[\tilde{\eta}\,\tilde{\w}_k^l,\tilde{\eta}\,\tilde{\w}_k^l]\leq -\frac{\delta_l}{4}<0
	\end{align*}
	so that
	\begin{align*}
		\mathrm{Ind}_{W}(\phi_k^{i,j})\leq \liminf_{k\rightarrow\infty}\mathrm{Ind}_{W_{\sigma_k}}(\phi_k).
	\end{align*}
	
	\textbf{Step 6 : Conclusion.}

	Therefore, there cannot be any linear relations between the negative variations of $\phi_{\infty}$,  $\phi^{i,j}_{\infty}$ and  $\phi_{\infty}^{i',j'}$ for $(i,j)\neq (i',j')$ once projected on $\phi_k$ for $k$ large enough and $\alpha$ small enough as they have disjoint support so we obtain the claimed inequality
	\begin{align*}
		\mathrm{Ind}_W(\phi_{\infty})+\sum_{i=1}^{m}\sum_{j=1}^{Q^i}\mathrm{Ind}_W(\phi_{\infty}^{i,j})\leq \liminf_{k\rightarrow\infty}\mathrm{Ind}_{W_{\sigma_k}}(\phi_k)\leq d.
	\end{align*}
	and this concludes the proof of the theorem.
\end{proof}
\begin{rem}
	We could also obtain the reverse bound by adding the nullity (and taking a co-dual, homotopic or cohomotopic admissible family \cite{index2}). However, due to the conformal group of $\R^n$, the nullity of branched Willmore surfaces is always at least equal to $3$, and as we do not have any upper bound for the number of surfaces realising the min-max, this would not yield much more informations. 
\end{rem}

\nocite{}
\bibliographystyle{plain}
\bibliography{biblio}

\end{document}